\documentclass[12 pt]{article}
\usepackage{graphicx} 
\usepackage{amsmath}
\usepackage{bm}
\usepackage{amsthm}
\usepackage[english]{babel}
\usepackage{hyperref}
\usepackage{geometry}
\usepackage{bbold}
\usepackage{stmaryrd}
\usepackage[all]{xy}
\usepackage{tikz}
\usepackage{MnSymbol}
\usepackage{pgf,tikz}
\usetikzlibrary{shapes}
\usetikzlibrary{arrows.meta}
\usetikzlibrary{graphs} 
\usetikzlibrary{graphs,quotes} 
\usetikzlibrary{graphs,shapes.geometric}
\usetikzlibrary{decorations.markings}
\usetikzlibrary{%
  arrows,
  knots,
  calc,
}
\tikzstyle directed=[postaction={decorate,decoration={markings,
		mark=at position .65 with {\arrow{latex}}}}]
\numberwithin{equation}{section}
\newtheorem{proposition}[equation]{Proposition}
\newtheorem{theorem}[equation]{Theorem}
\newtheorem{corollary}[equation]{Corollary}
\newtheorem{lemma}[equation]{Lemma}

\newtheorem{intro}{Theorem}
\newtheorem{introth}[intro]{Theorem}

\newtheorem{introprop}[intro]{Proposition}

\theoremstyle{definition}
\newtheorem{definition}[equation]{Definition}
\newtheorem{propdef}[equation]{Proposition-Definition}

\newtheorem{remark}[equation]{Remark}

\newcommand\Sub{\mathrm{Sub}}
\newcommand\Stab{\mathrm{Stab}}
\newcommand\Ph{\mathrm{Ph}}
\newcommand\dom{\mathrm{dom}}
\newcommand\rng{\mathrm{rng}}
\newcommand\trg{\mathbf{t}}
\newcommand\src{\mathbf{s}}
\newcommand\var{\mathbf{u}}
\newcommand\BS{\mathrm{BS}}
\newcommand\Sch{\mathrm{Sch}}
\newcommand\image{\mathrm{Im}}
\newcommand\id{\mathrm{id}}
\title{Perfect kernel of generalized Baumslag-Solitar groups}
\author{Sasha Bontemps}
\date{\today}

\newenvironment{acknowledgements}{%
	
	\begin{abstract}
	}{%
	\end{abstract}
}

\begin{document}
	
	\maketitle
	
	\begin{abstract}
		In this article, we study the space of subgroups of generalized Baumslag-Solitar groups (GBS groups), that is, groups acting cocompactly on an oriented tree without inversion and with infinite cyclic vertex and edge stabilizers. Our results generalize the study of Baumslag-Solitar groups in \cite{solitar1}. Given a GBS group $\Gamma$ defined by a graph of groups whose existence is given by Bass-Serre theory (\textit{cf.} \cite{serre}), we associate to any subgroup of $\Gamma$ an integer, which is a generalization of the phenotype defined in \cite{solitar1}. This quantity is invariant under conjugation and allows us to decompose the perfect kernel of $\Gamma$ into pieces which are invariant under conjugation and on which $\Gamma$ acts highly topologically transitively. To achieve this, we interpret graphs of subgroups of $\Gamma$ as "blown up and shrunk" Schreier graphs of transitive actions of $\Gamma$. We also describe the topology of the pieces which appear in the decomposition. 
	\end{abstract}
	
	
	\smallskip
	
	{
		\small	
		\noindent\textbf{{MSC-classification:}}	
		37B; 20E06; 20E08.
	}
	
	\tableofcontents

	\section{Introduction}
	
	A generalized Baumslag-Solitar group (GBS group) is a group that acts cocompactly without inversion on an oriented tree  with vertex and edge stabilizers isomorphic to $\mathbb{Z}$. By Bass-Serre theory, these groups are isomorphic to an iteration of HNN-extensions and amalgamated free products of infinite cyclic groups. Baumslag-Solitar groups $\BS(m,n)$ are a particular case of such groups. In this article, we are interested in the space of subgroups of those groups from a topological point of view.

	The set of subgroups of any infinite countable group $\Gamma$ is a closed subset of a Cantor space, hence is Polish. Cantor-Bendixson theory leads to a decomposition of the space of subgroup $\Sub(\Gamma)$ of $\Gamma$ into a closed subspace without isolated points called the \textbf{perfect kernel} and denoted by $\mathcal{K}(\Gamma)$ and a countable set. We are interested in the computation of the perfect kernel and the dynamics induced by the action by conjugation of $\Gamma$ on it.

	In the case of finitely generated groups, finite index subgroups are isolated. Hence, one has the following inclusion: 
	\[\mathcal{K}(\Gamma) \subseteq \Sub_{[\infty]}(\Gamma)\]
	where $\Sub_{[\infty]}(\Gamma)$ denotes the set of subgroups of $\Gamma$ that have infinite index. 
	Hence understanding when the previous inclusion is an equality is a natural question. 
	In \cite{totipotent}, the authors observed that equality holds in the case of non-abelian free groups.
	The authors of \cite{AG} proved that infinite ended groups also satisfy this equality. 
	In the case of non-solvable Baumslag-Solitar groups $\BS(m,n)$, equality holds if and only if $|m| \neq |n|$ (\textit{cf.} \cite{solitar1}). More generally, a subgroup of $\BS(m,n)$ lies in the perfect kernel if and only if its graph of groups decomposition given by its action on the Bass-Serre tree $\mathcal{A}$ induced by the standard presentation 
	\begin{equation}\label{bsmn}
		\BS(m,n) \simeq \left\langle b, t | tb^nt^{-1} = b^m \right\rangle
	\end{equation} is infinite. 
	
	The action of $\Gamma$ on itself by conjugation leads to an action of $\Gamma$ on $\Sub(\Gamma)$ by homeomorphisms, which preserves the perfect kernel. The dynamics induced on the perfect kernel has been studied in various cases. For instance, the authors of \cite{AG} proved that the action by conjugation of the free group $\mathbb{F}_k$ on $k$ generators on its perfect kernel is highly topologically transitive for every $k \geq 2$. Recall that an action $\alpha: \Gamma \curvearrowright X$ on a Polish space is said \textbf{highly topologically transitive} if it is $r$-topologically transitive for every $r \in \mathbb{N}$, \textit{i.e.} for every non-empty open subsets $U_1,...,U_r, V_1,...,V_r$ of $X$, there exists an element $\gamma \in \Gamma$ such that $\alpha(\gamma)U_i \cap V_i \neq \emptyset$ for every $i \in \llbracket 1,r \rrbracket$. A $1$-topologically transitive action is said \textbf{topologically transitive}. By Baire's theorem, this is equivalent to having a dense orbit. 
	
	In \cite{solitar1}, the authors obtained a decomposition of the perfect kernel of non-amenable Baumslag-Solitar groups into countably many pieces on which the action by conjugation is topologically transitive. In \cite{solitar2} the same authors announced that the action on each piece of the decomposition is highly topologically transitive. One of the pieces is closed and all the other ones are open, and also closed if and only if the group is \textbf{unimodular}. Recall that one characterization of unimodularity for a GBS group is the existence of an infinite cyclic normal subgroup (see \cite[Section 2]{levittisom} for instance). This decomposition is encoded by the \textbf{phenotype} of the subgroups of $\BS(m,n)$, an arithmetic quantity in $\mathbb{N} \cup \{\infty\}$ that is invariant under conjugation and which only depends on the intersection with the subgroup generated by an elliptic element (such as $b$ in the presentation \eqref{bsmn}).
	
	In this article, our goal is to extend these results to generalized Baumslag-Solitar group. We are interested in the non-amenable GBS groups, \textit{i.e.} which are not isomorphic to $\BS(1, n)$ for some $n \in \mathbb{Z}$. 
	We prove the following result (\textit{cf.} Theorem \ref{kernel1}): 
	\begin{introth}
		Let $\Gamma$ be a non-amenable GBS group. Then \begin{itemize}
			\item If $\Gamma$ is non-unimodular, then $\mathcal{K}(\Gamma) = \Sub_{[\infty]}(\Gamma)$.
			\item If $\Gamma$ is unimodular, then denoting by $C$ an infinite cyclic normal subgroup and $\pi: \Gamma \to \Gamma / C$ the canonical projection, one has $\mathcal{K}(\Gamma) = \pi^{-1}\left(\Sub_{[\infty]}(\Gamma/C)\right)$, \textit{i.e.} the set of subgroups of $\Gamma$ whose projection under $\pi$ has infinite index in $\Gamma / C$. 
		\end{itemize}
	\end{introth}
	Notice that any two infinite cyclic normal subgroups of a non-amenable unimodular GBS group $\Gamma$ (given by its action on its Bass-Serre tree $\mathcal{A}$) are commensurable (that is to say, share a common finite index normal subgroup): if $g \in \Gamma \setminus \{1\}$ generates an infinite cyclic normal subgroup of $\Gamma$ and $g$ is not elliptic, then $g$ is hyperbolic and by normality its axis is fixed by any element of $\Gamma$. Thus Bass-Serre theory tells us that $\Gamma$ is isomorphic to $\BS(1, \varepsilon)$ (for some $\varepsilon \in \{1, -1\}$), hence is amenable. Hence any infinite cyclic normal subgroup of a non-amenable GBS group is a non-trivial subgroup of a vertex stabilizer, so is commensurable to any vertex stabilizer. Hence the computation of the perfect kernel in this second case does not depend on the choice of the subgroup $C$.

     In Theorem \ref{toptrans}, we provide a study of the dynamics induced by the conjugation action on the perfect kernel $\mathcal{K}(\Gamma)$:

    \begin{introth}\label{maintheo}
        Let $\Gamma$ be a non-amenable GBS group. There exists an infinite countable subset $\mathcal{Q}$ of $\mathbb{N}^*$ and a partition \[\mathcal{K}(\Gamma) = \mathcal{K}_{\infty} \sqcup \bigsqcup_{P \in \mathcal{Q}}\mathcal{K}_P\] such that, for every $P \in \mathcal{Q} \sqcup \{\infty\}$: \begin{itemize}
            \item $\mathcal{K}_P$ is non-empty and invariant under conjugation;
            \item the conjugation action $\Gamma \curvearrowright \mathcal{K}_P$ is highly topologically transitive.
        \end{itemize}
    \end{introth}
    
	One also investigates the topology of the pieces in the decomposition of the space of subgroups (\textit{cf.} Proposition \ref{topologiepieces}): 
	
	\begin{introprop} Let $\Gamma$ be a non-amenable GBS group. In the decomposition provided by Theorem \ref{maintheo}, for the topology induced on $\mathcal{K}(\Gamma)$: \begin{itemize}
			\item the piece $\mathcal{K}_{\infty}$ is closed;
			\item for every $P \in \mathcal{Q}$, the piece $\mathcal{K}_{P}$ is open. It is also closed if and only if the group $\Gamma$ is unimodular. 
		\end{itemize}
	\end{introprop}
	
    To build this decomposition, we introduce a generalized notion of phenotype $\bm{\Ph}_{\mathcal{H},v}$ of a GBS group $\Gamma$ defined by a graph of groups $\mathcal{H}$ pointed at some vertex $v$. This is a function from the space of subgroups of $\Gamma$ to $\mathbb{N}^* \sqcup \{\infty\}$ whose image is infinite. With the notations of Theorem \ref{maintheo}, one has $\mathcal{Q} = \mathcal{Q}_{\mathcal{H},v} := \Ph_{\mathcal{H}, v}\left(\mathbb{N}\right)$ and, for any $P \in  \mathcal{Q} \sqcup \{\infty\}$, the piece $\mathcal{K}_P$ is exactly $\mathcal{K}(\Gamma) \cap \bm{\Ph}_{\mathcal{H},v}^{-1}(P)$ (and these pieces are non-empty). 
	
	It turns out that the piece $\bm{\Ph}_{\mathcal{H},v}^{-1}(\infty)$ depends neither on the choice of $\mathcal{H}$ nor of $v$. It consists of the set of subgroups that act freely on the Bass-Serre tree of $\Gamma$ associated to $\mathcal{H}$. Equivalently, this is the set of subgroups that intersect the subgroup generated by any elliptic element trivially. This set of elliptic elements is intrinsically characterized as the set of elements that are commensurable to all their conjugates (see Proposition \ref{ell}). In particular, the decomposition $\mathcal{K}(\Gamma) = \bigsqcup_{N \in \mathcal{Q}_{\mathcal{H}, v} \sqcup \{\infty \}}\mathcal{K}(\Gamma) \cap \bm{\Ph}_{\mathcal{H},v}^{-1}(N)$ depends neither on the choice of the graph $\mathcal{H}$ nor on the choice of the vertex $v$ (see Proposition \ref{intrinsic}). The open pieces of this decomposition are uniquely characterized as follows: they are the maximal topologically transitive open subsets for the action of $\Gamma$ on $\mathcal{K}(\Gamma) \smallsetminus \mathcal{K}_{\infty}$. 
	
	To obtain these results, we generalize the notion of preactions and $(m,n)$-graphs defined for Baumslag-Solitar groups in \cite{solitar1}. We introduce the notion of $\mathcal{H}$-\textbf{preaction} (\textit{cf.} Definition \ref{hpreac}) to establish a link between the Schreier graph of a subgroup $\Lambda$ of $\Gamma$ and its graph of groups $\Lambda \backslash \mathcal{A}$. This notion was already defined in \cite{hightrans} in the case of a single HNN-extension and in the case of an amalgamated free product (\textit{i.e.} in the case of a graph of groups $\mathcal{H}$ which consists in a single edge). Our notion of preaction is a little more tricky because it depends on the choice of a spanning tree of $\mathcal{H}$. 
	
		\begin{acknowledgements}
		I am deeply grateful to my PhD advisor Damien Gaboriau for encouraging me to work on this subject and for the enlightening discussions that gave rise to the writing of this article. I also want to express my gratitude to François Le Maître and Yves Stalder for their useful remarks on the preliminary versions of this work. I also thank François Le Maître for giving a dynamical interpretation of the phenotypical decomposition in Lemma \ref{decompdyn}. I warmly thank Rémi Coulon and Ashot Minasyan for their careful reading and valuable remarks on this text. Last but not least, I thank the anonymous referee for their incredible work, and for valuable suggestions that significantly improved the paper. This work was supported by a CDSN from ENS de Lyon.
	\end{acknowledgements}

	\section{Preliminaries and notations}
	
	We denote by $\mathcal{P}$ the set of prime numbers in $\mathbb{Z}$. For every integer $N$, we denote by $|N|_p$ its $p$-adic valuation, that is, the largest $n \in \mathbb{N}$ such that $p^n$ divides $N$. By convention, $|\infty|_p = \infty$ for every prime number. For every integers $m,n \in \mathbb{Z} \setminus \{0\}$ we denote by $m \wedge n$ their greatest common divisor, that is, the largest integer $k \in \mathbb{N}$ dividing both $m$ and $n$. By convention, $\infty \wedge n = n \wedge \infty = |n|$ for every $n \in \mathbb{Z}^*$. By ``countable'' we mean finite or in bijection with $\mathbb{N}$. 
	
	\subsection{Graphs}
	
	\subsubsection{Definitions and notations}\label{def}
	
	We follow the notations of \cite{serre}.
	\begin{definition}
		A \textbf{graph} $\mathcal{G}$ is the data of \begin{itemize}
			\item A set of \textbf{vertices} $\mathcal{V}(\mathcal{G})$;
			\item A set of \textbf{edges} $\mathcal{E}\left(\mathcal{G}\right)$; 
			\item A map $\src: \mathcal{E}\left(\mathcal{G}\right) \to \mathcal{V}\left(\mathcal{G}\right)$ called \textbf{source};
			\item A map $\trg: \mathcal{E}\left(\mathcal{G}\right) \to \mathcal{V}\left(\mathcal{G}\right)$ called \textbf{target};
			\item a fixed-point free involution $\begin{array}{ccccc}
				\mathcal{E}\left(\mathcal{G}\right) & \to & \mathcal{E}\left(\mathcal{G}\right) \\
				e & \mapsto & \overline{e} \\
			\end{array}$ such that $\src(\overline{e}) = \trg(e)$ and $\trg(\overline{e}) = \src(e)$.
		\end{itemize}
		An \textbf{orientation} of the graph is a partition $\mathcal{E}\left(\mathcal{G}\right) = \mathcal{E}^+(\mathcal{G}) \bigsqcup \mathcal{E}^-(\mathcal{G})$ whose pieces are exchanged by the involution where \begin{itemize}
			\item The elements of $\mathcal{E}^+(\mathcal{G})$ are called \textbf{positive edges};
			\item The elements of $\mathcal{E}^-(\mathcal{G})$ are called \textbf{negative edges}.
		\end{itemize}
		A graph which is equipped with an orientation is called an \textbf{oriented graph}. 
	\end{definition}
	
	Let $\mathcal{G}$ be an oriented graph and $v \in \mathcal{V}\left(\mathcal{G}\right)$. An edge $e \in \mathcal{E}\left(\mathcal{G}\right)$ is said to be \textbf{incident} to $v$ if $v \in \{\src(e), \trg(e)\}$. It is \textbf{incoming} in $v$ if $\trg(e) = v$. We say that $e$ is \textbf{outgoing} from $v$ if $\src(e) = v$. A \textbf{leaf} of $\mathcal{G}$ is a vertex $f \in \mathcal{V}\left(\mathcal{G}\right)$ which is incident to exactly two edges ($e$ and $\overline{e}$ for some $e \in \mathcal{E}(\mathcal{G})$). 
	
	An \textbf{edge path} of length $r$ of a graph $\mathcal{G}$ is a sequence $c$ of edges $e_1,...,e_r \in \mathcal{E}\left(\mathcal{G}\right)$ satisfying $\trg(e_i) = \src(e_{i+1})$ for every $i \in \llbracket 1,r-1 \rrbracket$. It is said to be \textbf{reduced} if for every $i \in \llbracket 1,r-1 \rrbracket$, one has $e_{i+1} \neq \overline{e_i}$. It is \textbf{based at} $s$ if $\src(e_1) = s$. It \textbf{connects} $s$ to $t$ if $\src(e_1) = s$ and $\trg(e_r) = t$. If $\trg(e_r) = \src(e_1)$, then $c$ is a \textbf{cycle} (based at $s = \src(e_1)$). A \textbf{loop} is a cycle of length $1$. An edge path $c = e_1, ..., e_r$ is \textbf{simple} if either $r=1$ or if it is reduced and the $\src(e_i)$ are pairwise distinct (for $i \in \llbracket 1, r \rrbracket$) and $\trg(e_r) \notin \{\src(e_i), i \in \llbracket 2, r \rrbracket\}$. If $c = e_1,...,e_r$ and $c' = e_{r+1},...,e_n$ are two edge paths that satisfy $\trg(e_r) = \src(e_{r+1})$, one denotes by $c*c'$ their \textbf{concatenation} $e_1,...,e_r,e_{r+1},...e_n$.
	
	The set of vertices of a graph $\mathcal{G}$ it equipped with a distance \[d_{\mathcal{G}}(x,y) = \inf \{r \in \mathbb{N} \cup \{\infty\} \mid \text{there exists an edge path $e_1,..., e_r$ connecting $x$ to $y$} \}\]
	(where $d_{\mathcal{G}}(x,y) = \infty$ if such path does not exist). 
	Let $\mathcal{G}$, $\mathcal{G}'$ be graphs. The graph $\mathcal{G}$ is a \textbf{subgraph} of $\mathcal{G}'$ if $\mathcal{V}(\mathcal{G}) \subseteq \mathcal{V}(\mathcal{G}')$, $\mathcal{E}(\mathcal{G}) \subseteq \mathcal{E}(\mathcal{G}')$ and the source and target maps and the fixed-point free involution of $\mathcal{G}$ and $\mathcal{G}'$ coincide on $\mathcal{E}(\mathcal{G})$. The \textbf{induced subgraph} of $\mathcal{G'}$ by a set of vertices $V \subseteq \mathcal{V}\left(\mathcal{G}'\right)$ is the subgraph of $\mathcal{G}'$ defined by its set of vertices $V$ and its set of edges $\left\{E \in \mathcal{E}\left(\mathcal{G}\right), (\src(E), \trg(E)) \in V^2 \right\}$. Given a graph $\mathcal{G}$ and a vertex $x \in \mathcal{V}\left(\mathcal{G}\right)$, the ball of center $x$ and radius $R > 0$ (denoted by $B_{\mathcal{G}}(x, R)$) is the subgraph induced by the set of vertices at distance at most $R$ from $x$ in $\mathcal{G}$. 
	
	The graph $\mathcal{G}$ is said to be \textbf{connected} if for any $(v,w) \in \mathcal{V}\left(\mathcal{G}\right)^2$, there exists an edge path connecting $v$ to $w$. The \textbf{connected component} of $v \in \mathcal{V}\left(\mathcal{G}\right)$ is the largest connected subgraph of $\mathcal{G}$ containing $v$. The \textbf{half-graph} of an edge $e$ of $\mathcal{G}$ is the connected component containing $\trg(e)$ of the subgraph of $\mathcal{G}$ defined by its set of vertices $\mathcal{V}(\mathcal{G})$ and its set of edges $\mathcal{E}\left(\mathcal{G}\right) \setminus \{e, \overline{e}\}$. The graph $\mathcal{G}$ is a \textbf{tree} if it connected and does not contain any reduced cycle.  It is called a \textbf{forest} if all its connected components are trees.
	
	Let $\mathcal{G}$ be a connected graph and $\left(\mathcal{G}^{(i)}\right)_{i \in I}$ be a collection of connected disjoint subgraphs of $\mathcal{G}$. The \textbf{quotient} $\mathcal{G}/\left(\mathcal{G}^{(i)}, i \in I\right)$ is the graph defined by: \begin{itemize}
		\item its set of vertices is $\left(\mathcal{V}\left(\mathcal{G}\right) \setminus \sqcup_{i \in I}\mathcal{V}\left(\mathcal{G}^{(i)}\right)\right) \sqcup I$;
		\item its set of positive edges is $\mathcal{E}^+\left(\mathcal{G}\right) \setminus \sqcup_{i \in I}\mathcal{E}^+\left(\mathcal{G}^{(i)}\right)$;
		\item its set of negative edges is $\mathcal{E}^-\left(\mathcal{G}\right) \setminus \sqcup_{i \in I}\mathcal{E}^-\left(\mathcal{G}^{(i)}\right)$;
		\item for every $E \in \mathcal{E}\left(\mathcal{G}\right) \setminus \sqcup_{i \in I}\mathcal{E}\left(\mathcal{G}^{(i)}\right)$ the source map is defined by $\src(E)$ if $\src(E) \notin \sqcup_{i \in I}\mathcal{V}\left(\mathcal{G}^{(i)}\right)$ and by $i$ if $\src(E) \in \mathcal{V}\left(\mathcal{G}^{(i)}\right)$;
		\item for every $E \in \mathcal{E}\left(\mathcal{G}\right) \setminus \sqcup_{i \in I}\mathcal{E}\left(\mathcal{G}^{(i)}\right)$ the target map is defined by $\trg(E)$ if $\trg(E) \notin \sqcup_{i \in I}\mathcal{V}\left(\mathcal{G}^{(i)}\right)$ and by $i$ if $\trg(E) \in \mathcal{V}\left(\mathcal{G}^{(i)}\right)$.
	\end{itemize}

	\subsubsection{Schreier graphs}\label{schreier}
	
	Let $\Gamma$ be a countable group. One has a correspondence between \begin{itemize}
		\item subgroups of $\Gamma$ and isomorphism classes of pointed transitive right actions $(X,x_0) \curvearrowleft \Gamma$; 
		\item conjugacy classes of subgroups of $\Gamma$ and isomorphism classes of transitive right actions $X \curvearrowleft \Gamma$.
	\end{itemize}
	It is given by the bijection $f: \begin{array}{ccccc}
		\Sub(\Gamma) & \to & \{\text{pointed transitive right actions of $\Gamma$}\} \\
		\Lambda & \mapsto & (\Lambda \backslash \Gamma, \Lambda) \curvearrowleft \Gamma \\
	\end{array}$
	whose inverse is given by $g: \begin{array}{ccccc}
		\{\text{pointed transitive right actions of $\Gamma$}\} & \to & \Sub(\Gamma) \\
		(X,x_0) \curvearrowleft \Gamma & \mapsto & \Stab(x_0) \\
	\end{array}$.
	Changing the base point of a transitive $\Gamma$-action $\alpha$ on a countable set $(X,x_0)$ amounts to conjugating the subgroup represented by $\alpha$ and $x_0$.
	
	Given any generating set $S$ of $\Gamma$, one can associate a graph $\Sch(\alpha)$ called \textbf{Schreier graph} to any transitive right action $X \curvearrowleft^{\alpha} \Gamma$ (hence to any subgroup of $\Gamma$ by the above correspondence): its vertex set is $X$ and for any $x \in X$, for any $s \in S$, define a positive edge labeled $s$ from $x$ to $x\alpha(s)$. It is pointed at $x_0$ if $\alpha$ is pointed at $x_0$. Two Schreier graphs $\Sch(\alpha,v)$ and $\Sch(\beta,v')$ are isomorphic (as labeled graphs) if and only if the subgroups represented by $\alpha$ and $\beta$ (with respect to the points $v$ and $v'$) are conjugated. For any $R > 0$, we denote $\Sch(\alpha,v) \simeq_R \Sch(\beta,v')$ if the $R$-balls around $v$ and $v'$ of the Schreier graphs of $\alpha$ and $\beta$ are isomorphic as labeled graphs. 
	
	\begin{remark}
	    One could also have considered left actions instead of right actions. The reason why we deal with right actions is that the left action of $\Gamma$ on a right Cayley graph induces an isometry of the graph.
	\end{remark}
	
	\subsection{Elements of Bass-Serre theory}\label{bassserre}
	
	In this section, we recall the fundamentals of Bass-Serre theory. We refer to \cite{serre} for more details. 
	
	Let $\mathcal{H}$ be an oriented graph equipped with a collection of \textbf{vertex groups} $G_v$ for $v \in \mathcal{V}\left(\mathcal{G}\right)$, a collection of \textbf{edge groups} $G_e$ for $e \in \mathcal{E}\left(\mathcal{H}\right)$ such that $G_e = G_{\overline{e}}$ for every edge $e$ and, for every $e \in \mathcal{E}\left(\mathcal{H}\right)$ and $\var \in \{\src, \trg\}$, injective homomorphisms $\iota_{e,\var}: G_e \hookrightarrow G_{\var(e)}$ such that $\iota_{e, \src} = \iota_{\overline{e}, \trg}$. This data is called a \textbf{graph of groups}.
	We associate a group to $\mathcal{H}$ called the \textbf{fundamental group} of $\mathcal{H}$, defined by the following procedure: fix a spanning tree $T$ in $\mathcal{H}$. Denote by $\{t_e, e \in \mathcal{E}\left(\mathcal{H}\right)\}$ a generating set of the free group $\mathbb{F}_{|\mathcal{E}\left(\mathcal{H}\right)|}$ of rank $|\mathcal{E}\left(\mathcal{H}\right)|$ and define \begin{align*} \Gamma = \left(*_{v \in \mathcal{V}\left(\mathcal{H}\right)}G_v * \mathbb{F}_{|\mathcal{E}\left(\mathcal{H}\right)|}\right) / \left\llangle \left(t_e^{-1}\iota_{e,\src}(g)t_e\iota_{e, \trg}(g)^{-1}\right)_{(e,g) \in \mathcal{E}\left(\mathcal{H}\right) \times G_e},  (t_et_{\overline{e}})_{e \in \mathcal{E}\left(\mathcal{H}\right)}, (t_e)_{e \in \mathcal{E}\left(T\right)} \right\rrangle \end{align*}
	Then, there exists a unique oriented tree $\mathcal{A}$, called the \textbf{Bass-Serre tree} of the graph of groups on which $\Gamma$ acts without inversion with quotient $\mathcal{H}$ and such that there exist sections $\mathcal{V}\left(\mathcal{H}\right) \to \mathcal{V}\left(\mathcal{A}\right)$ and $\mathcal{E}\left(\mathcal{H}\right) \to \mathcal{E}\left(\mathcal{A}\right)$ (which we denote by $v \mapsto \sigma{v}$ and $e \mapsto \sigma{e}$, respectively) of the projection $\pi : \mathcal{A} \to \mathcal{H}$ that satisfies the following conditions: \begin{itemize}
		\item $\Stab(\sigma{v}) = G_v \ \forall v \in \mathcal{V}(\mathcal{G})$
		\item $\Stab(\sigma{e}) = G_e \ \forall e \in \mathcal{E}(\mathcal{G})$
	\end{itemize}
	
	Moreover, the isomorphism class of $\Gamma$ does not depend on the choice of the spanning tree $T$ in $\mathcal{H}$ and $\mathcal{A}$ is unique up to unique isomorphism. 
	
	Conversely, any group $\Gamma$ acting on an oriented tree $\mathcal{A}$ is the fundamental group of the graph of groups whose underlying graph is the quotient $\mathcal{H} = \Gamma \backslash \mathcal{A}$ and whose vertex and edge groups are defined as follows: let $T$ be a spanning tree of $\mathcal{H}$. Let $\sigma : T \to \mathcal{A}$ be a lifting of $T$, which we extend to a section $\sigma : \mathcal{E}\left(\mathcal{H}\right) \to \mathcal{E}\left(\mathcal{A}\right)$ satisfying $\sigma(\overline{e}) = \overline{\sigma(e)}$ and $\trg\left(\sigma(e)\right) = \sigma\left(\trg(e)\right)$ for every $e \in \mathcal{E}^+\left(\mathcal{H}\right)$. Then define the vertex and edge groups by:\begin{itemize}
		\item $G_v = \Stab(\sigma(v))$ for every $v \in \mathcal{V}\left(\mathcal{H}\right)$;
		\item $G_e = \Stab(\sigma(e))$ for every $e \in \mathcal{E}\left(\mathcal{H}\right)$.
	\end{itemize} 
	and the injections $\iota_{e, \trg} : G_e \hookrightarrow G_{\trg(e)}$ as in \cite[Chapter 1, Section 5.4]{serre}. 
	
	Let $\Gamma$ be the fundamental group of a graph of groups $\mathcal{H}$ and let $\mathcal{A}$ be its Bass-Serre tree. Any subgroup $\Lambda$ of $\Gamma$ acts without inversion on the Bass-Serre tree $\mathcal{A}$ of $\Gamma$. Hence $\Lambda$ is the fundamental group of the graph of groups with underlying graph $\Lambda \backslash \mathcal{A}$ and vertex and edge groups defined as previously (hence any vertex (\textit{resp.} edge) group is an intersection of $\Lambda$ with some conjugate of a vertex (\textit{resp.} edge) group of $\mathcal{H}$). 
	
	\subsection{Generalized Baumslag-Solitar groups}\label{GBS group}
	
	A GBS group is the fundamental group of a finite graph of groups with vertex and edge stabilizers infinite cyclic. Equivalently, these groups are obtained by an iteration of HNN-extensions and amalgamated free products of $\mathbb{Z}$ over $\mathbb{Z}$ starting from $\mathbb{Z}$. \\
	As injective morphisms $\mathbb{Z} \to \mathbb{Z}$ are fully determined by the image in $\mathbb{Z} \setminus \{0\}$ of $1$, Bass-Serre theory tells us that a GBS group $\Gamma$ can be represented by an oriented graph $\mathcal{H}$ endowed with a function which associates an integer to each half-edge:  $$k: \begin{array}{ccccc}
		\mathcal{E}\left(\mathcal{H}\right) \times \{\src,\trg\} & \to & \mathbb{Z} \setminus \{0\} \\
		(e,\var) & \mapsto & k_{e,\var} \\
	\end{array}$$ such that for every edge $e$, one has $k_{e, \src} = k_{\overline{e}, \trg}$. 
	From now, we fix a spanning tree $T$ of $\mathcal{H}$. As a consequence of Bass-Serre theory, the group $\Gamma$ is defined by the following presentation: 
	
	\begin{align*} \Gamma \simeq \Bigg\langle (a_v)_{v \in \mathcal{V}\left(\mathcal{H}\right)}, (t_e)_{e \in \mathcal{E}\left(\mathcal{H}\right)} &\mid \left(t_e^{-1}a_{\src(e)}^{k_{e,\src}}t_e a_{\trg(e)}^{-k_{e,\trg}}  \text{\ and \ } t_et_{\overline{e}} \ \text{ \ for all \ } e \in \mathcal{E}\left(\mathcal{H}\right)\right) \end{align*} \begin{equation}\label{pres1}\text{ \ and \ } \left(t_e \text{ \ for all \ } e \in \mathcal{E}(T)\right)  \Bigg\rangle. \end{equation} 
	
		Let us give an overview of the construction of the Bass-Serre tree for GBS groups. The Bass-Serre tree can be thought of as a ``blown up and shrunk Cayley graph'': its vertices are the cosets of vertex groups, and its edges are the cosets of edge groups. More precisely, let $Y$ be the graph defined by \begin{itemize}
		\item its vertex set \[\mathcal{V}(Y) = \Gamma \times \mathcal{V}\left(\mathcal{H}\right);\]
		\item its positive edge set \[\mathcal{E}^+(Y) = \Gamma \times \{a_v, e \vert v \in \mathcal{V}\left(\mathcal{H}\right), e \in \mathcal{E}^+\left(\mathcal{H}\right) \};\]
		\item for every $\gamma \in \Gamma$ and $v \in \mathcal{V}\left(\mathcal{H}\right)$: \begin{itemize}
			\item $\src(\gamma, a_v) = (\gamma, v)$;
			\item $\trg(\gamma, a_v) = (\gamma a_v, v)$;
		\end{itemize}
		\item for every $\gamma \in \Gamma$ and $e \in \mathcal{E}^+\left(T\right)$: \begin{itemize}
			\item $\src(\gamma, e) = (\gamma, \src(e))$;
			\item $\trg(\gamma, e) = (\gamma, \trg(e))$;
		\end{itemize}
		\item for every $\gamma \in \Gamma$ and $e \in \mathcal{E}^+\left(\mathcal{H}\right) \setminus \mathcal{E}^+\left(T\right)$: \begin{itemize}
			\item $\src(\gamma, e) = (\gamma, \src(e))$;
			\item $\trg(\gamma, e) = (\gamma t_e, \trg(e))$.
		\end{itemize}
	\end{itemize}
	The graph $Y$ is obtained by ``blowing up'' the Cayley graph of $\Gamma$ with respect to the generating set $S = \{a_v, v \in \mathcal{V}\left(\mathcal{H}\right), t_e, e \in \mathcal{E}\left(\mathcal{H}\right) \setminus \mathcal{E}(T) \}$: shrinking all the edges of $Y$ with source $(\gamma, \src(e))$ and target $(\gamma, \trg(e))$ (for any $\gamma \in \Gamma$ and $e \in \mathcal{E}\left(T\right)$) leads to the Cayley graph of $\Gamma$ with respect to $S$. The group $\Gamma$ acts on the graph $Y$ by the (left) action: $\gamma_1 \cdot (\gamma_2,v) = (\gamma_1 \gamma_2, v)$, and the group $\mathbb{Z}$ acts on the set of vertices of $Y$ by the right action: for any $(\gamma, v)$ and $n \in \mathbb{Z}$, define $(\gamma, v)\cdot n = (\gamma a_v^{n}, v)$. The graph $Y / \mathbb{Z}$ is precisely the Bass-Serre tree of $\Gamma$. The fact that this is a tree results from the normal form theorem (\textit{cf.} \cite[Section 5.2, Corollary 3 of Theorem 11]{serre}). In particular, the quotient $\Gamma \backslash (Y / \mathbb{Z})$ is (up to labelling the vertices and edges) the graph of groups $\mathcal{H}$.
	
	The elements $a_v$ are called \textbf{vertex generators} and the elements $t_e$ are called \textbf{edge generators}. Let $c = e_1, ..., e_r$ be an edge path of $\mathcal{H}$. A \textbf{word} of \textbf{type} $c$ is a couple \[\left(c, \mu\right) = \left((e_1,...,e_r), \left(a_{\src(e_1)}^{k_1}, a_{\src(e_2)}^{k_2}, ..., a_{\src(e_r)}^{k_r}, a_{\trg(e_r)}^{k_{r+1}}\right)\right).\] 
	A \textbf{prefix} of $(c, \mu)$ is a word \[\left((e_1,...,e_i), \left(a_{\src(e_1)}^{k_1}, a_{\src(e_2)}^{k_2}, ..., a_{\src(e_i)}^{k_i}, a_{\trg(e_i)}^{k_{i+1}}\right)\right)\]
	for some $i \in \llbracket 1, r\rrbracket$. \\
	The \textbf{length} of $(c, \mu)$ is the integer $r$. The associated element of $\Gamma$ is \[\left|c, \mu\right| = a_{\src(e_1)}^{k_1} s_1 a_{\src(e_2)}^{k_2} \ldots a_{\src(e_r)}^{k_r} s_r a_{\trg(e_r)}^{k_{r+1}}\](where $s_i = t_{e_i}$ if $e_i \notin T$ and $s_i = 1$ otherwise). 
	The word $(c, \mu)$ is said to be \textbf{reduced} if \begin{itemize}
		\item either $r=0$ (\textit{i.e.} the edge path is empty) and $k_1 \neq 0$;
		\item or for every $i \in \llbracket 1, r-1 \rrbracket$ such that $\overline{e_{i+1}} = e_i$, the integer $k_{i+1}$ is not divisible by $k_{e_i, \trg}$.
	\end{itemize}
	Recall (\textit{cf.} \cite[Chapter 1, Section 5.2]{serre}) that a reduced word leads to a non-trivial element of $\Gamma$. If $(c, \mu) = \left((e_1, ..., e_r), (a_0, ..., a_r)\right)$ and $(d, \lambda) = \left((f_1, ..., f_s), (b_0, ..., b_s)\right)$ are two words satisfying $\trg(e_r) = \src(f_1)$, one defines their \textbf{concatenation} by \[(c, \mu)*(d, \nu) = \left((e_1, ..., e_r, f_1, ..., f_s), (a_0, ...,a_{r-1}, a_rb_0, b_1, ..., b_s)\right).\] 
	As explained in \cite[Section 2]{levitt}, any GBS group can be represented by a \textbf{reduced} graph of groups (which is not necessarily unique), that is to say, we will assume that the only edges $e$ such that $k_{e, \src} = \pm 1$ or $k_{e,\trg} = \pm 1$ are loops. 

    Examples of GBS groups include Baumslag-Solitar groups $\BS(m,n)$. The presentation associated to the graph of groups in Figure \ref{graphebsmn} is the one given in Equation \ref{bsmn}.
\begin{figure}[ht]
				\center
				\begin{tikzpicture}
					\node[draw,circle,fill=gray!50] (a) at (-3,0) {};
					\draw[>=latex, directed] (a) to [out=135,in=45,looseness=20] node[very near start, below left]{$m$} node[very near end, below right]{$n$} (a);
				\end{tikzpicture}
				\caption{Standard graph of groups of $\BS(m,n)$.}
				\label{graphebsmn}
			\end{figure}

    Other examples of GBS groups include fundamental groups of torus knots $G(p,q)$, \textit{i.e.} the fundamental group of the complement of a certain embedded circle in $\mathbb{R}^3$. Each of those groups is defined by two integer coprime parameters $(p,q)$ - the amount of windings of the knot around the longitudinal direction and around the meridional direction, respectively (see Figure \ref{knot}) - and the graph of groups that defines its standard presentation \begin{equation}\left\langle a, b \mid a^p = b^q\right\rangle\end{equation} is given in Figure \ref{graphegpq}.

    \begin{figure}[ht]
				\center
				\begin{tikzpicture}
					\node[draw,circle,fill=gray!50] (a) at (-3,0) {};
                    \node[draw,circle,fill=gray!50] (b) at (0,0) {};
					\draw[>=latex, directed] (a) to node[very near start, below left]{$p$} node[very near end, below right]{$q$} (b);
				\end{tikzpicture}
				\caption{Standard graph of groups of $G(p,q)$.}
				\label{graphegpq}
			\end{figure}

\begin{figure}[ht]
\center
\begin{tikzpicture}
\def\fst{(0,0) ellipse (5 and 3)}
\def\snd{(0,0.5) ellipse (2 and 0.5)}
\fill[cyan] \fst ;
\fill[white] \snd ;
\draw[>=latex, directed, very thick, color=red] (2.5,-2.6) to[bend left=60] (-2.05, 0.5);
\draw[>=latex, directed, dashed, very thick, color=red] (-2, 0.6) to[bend left=30] (-0.1,3);
\draw[>=latex, directed, very thick, color=red] (0.05, 2.97) to[bend left=40] (2,0.4);
\draw[>=latex, directed, dashed, very thick, color=red] (2,0.3) to[bend left=50] (-2.5,-2.6);
\draw[>=latex, directed, very thick, color=red] (-2.5, -2.6) to[bend left=75] (-2,1.5);
\draw[>=latex, directed, very thick, color=red] (-2,1.52) to[bend left=25] (1.9,0.7);
\draw[>=latex, directed, dashed, very thick, color=red] (2,0.5) to[bend left=30] (2.5,-2.6);
\draw[>=latex, very thick, color=blue] (-2,0.4) to[bend left=20] (-2.4,0.8);
\draw[>=latex, very thick, color=blue] (2,0.4) to[bend right=20] (2.2,0.8);
\end{tikzpicture}
\caption{$G(2,3)$ is the fundamental group of the complement of the trefoil in $\mathbb{R}^3$}
\label{knot}
\end{figure}

	Notice that a reduced graph of groups that represents a non-amenable GBS group $\Gamma$ is neither a loop one of whose labels is equal to $\pm 1$ nor a segment with both labels equal to $\pm 2$: in this last case, a presentation of $\Gamma$ is given by $\Gamma \simeq \langle a, b \mid a^2 = b^2 \rangle$. Hence the group $\Gamma$ is isomorphic to the fundamental group of the Klein Bottle $\BS(1, -1) \simeq \langle \tau, \beta \mid \tau \beta \tau^{-1} = \beta^{-1} \rangle$ \textit{via} the following isomorphism: 
	\[\begin{array}{ccccc}
		\Gamma & \to & \BS(1, -1) \\
		a & \mapsto & \tau \\
		b & \mapsto & \beta \tau
	\end{array}\]
	whose inverse is given by 
	\[\begin{array}{ccccc}
		\BS(1, -1) & \to & \Gamma \\
		\tau & \mapsto & a \\
		\beta & \mapsto & ba^{-1}
	\end{array}.\]
	Conversely, any GBS group defined by a reduced graph of groups which is neither a loop one of whose labels is equal to $\pm 1$ nor a segment with both labels equal to $\pm 2$ contains a free group, thus is non-amenable.
	
	Let us fix a reduced graph of groups $\mathcal{H}$ defining $\Gamma$ and let $\mathcal{A}$ be the associated Bass-Serre tree. An \textbf{elliptic} element is an element $\gamma \in \Gamma$ that fixes a vertex in $\mathcal{A}$. It is a power of a conjugate of a vertex generator $a_v$ for some $v \in \mathcal{V}\left(\mathcal{H}\right)$. If $\gamma \in \Gamma$ is not elliptic, $\gamma$ is said \textbf{hyperbolic}. For a non-amenable GBS group $\Gamma$, the elliptic elements can be defined intrinsically:
	\begin{proposition}\label{ell}
		Let $\Gamma$ be a non-amenable GBS group and let $g \in \Gamma$. The following are equivalent: \begin{enumerate}
			\item $g$ is elliptic (for the action of $\Gamma$ on its Bass-Serre tree);
			\item for every $\gamma \in \Gamma$, the elements $g$ and $\gamma g \gamma^{-1}$ are commensurable.
		\end{enumerate}
	\end{proposition}
	
	\begin{proof} 
		$(1) \implies (2)$ remains valid without assuming that $\Gamma$ is amenable. To prove this implication, we show that any two non-trivial elliptic elements are commensurable. As any two non-trivial subgroups of $\mathbb{Z}$ are commensurable, an induction on the length $r$ of an edge path $E_1, ..., E_r$ of $\mathcal{A}$ shows that \[\bigcap_{i=1}^r\Stab(E_i) \neq \{1\} \text{ \ for every edge path $E_1, ..., E_r$ in $\mathcal{A}$.}\] 
		In particular, for any vertices $V, W \in \mathcal{V}(\mathcal{A})$, denoting by $E_1, ..., E_r$ the unique reduced edge path connecting $V$ to $W$ in $\mathcal{A}$, the group $\bigcap_{i=1}^r\Stab(E_i)$ is a non-trivial subgroup of both $\Stab(V)$ and $\Stab(W)$, so has finite index in both $\Stab(V)$ and $\Stab(W)$. This proves the first implication. 
		
		We argue by contradiction to prove $(2) \implies (1)$. If $g \in \Gamma$ is hyperbolic and commensurable to its conjugates, then $g$ and $\gamma g \gamma^{-1}$ share the same axis $L$ for any $\gamma \in \Gamma$, so $\Gamma$ stabilizes $L$. The quotient graph of groups $\Gamma \backslash L$ being reduced, the action of $\Gamma$ on $L$ has either one or two orbits of vertices. Thus, Bass-Serre theory tells us that $\Gamma$ is the fundamental group of a graph of groups defined either by a loop labeled $(1, \pm 1)$ if the action of $\Gamma$ on $L$ has a single orbit of vertices or by a segment labeled $(2, \pm 2)$ otherwise. Hence $\Gamma$ is 
		isomorphic to $\BS(1, \varepsilon)$ for some $\varepsilon \in \{1, -1\}$, so is amenable.  
	\end{proof}
	
	For any GBS group $\Gamma$, one can define the \textbf{modular homomorphism} $\Delta_{\Gamma}: \Gamma \to \mathbb{Q}^*$ as follows: fix a non-trivial elliptic element $a \in \Gamma$. For every $\gamma \in \Gamma$, there exist non-zero integers $m,n \in \mathbb{Z}$ such that $\gamma a^m \gamma^{-1} = a^n$. One defines $\Delta_{\Gamma}(\gamma) = \frac{m}{n}$; we easily check that it is well-defined and that it does not depend on the choice of the elliptic element $a$. The group $\Gamma$ is \textbf{unimodular} if $\image(\Delta_{\Gamma}) \subseteq \{1,-1\}$. This condition can be rephrased in terms of the labels of the graph $\mathcal{H}$: the group $\Gamma$ is unimodular if and only if for every cycle $c$, one has $|\prod_{e \in c}k_{e,\src}| = |\prod_{e \in c}k_{e,\trg}|$ (in particular, any GBS group which is defined by a tree of groups is unimodular). One has the following characterization of unimodularity (\textit{cf.} \cite[Section 2]{levitt}): a GBS group $\Gamma$ is unimodular if and only if one of the following equivalent conditions holds: \begin{itemize}
		\item $\image(\Delta_{\Gamma}) \subseteq \{1,-1\}$;
		\item $\Gamma$ contains an infinite cyclic normal subgroup;
		\item $\Gamma$ is virtually isomorphic to $\mathbb{F}_n \times \mathbb{Z}$ for some $n \geq 1$.  
	\end{itemize}
	
	\subsection{Space of subgroups and perfect kernel}
	
	Given a countable group $\Gamma$, we denote by $\Sub(\Gamma)$ its space of subgroups. By using indicator functions, one can embed $\Sub(\Gamma)$ in $\{0,1\}^{\Gamma}$. Equipped with the product topology, $\{0,1\}^{\Gamma}$ is a totally disconnected compact metrizable space of which $\Sub(\Gamma)$ is a closed subspace. An explicit basis of open sets is given by the following clopen sets: \[\mathcal{V}(O,I) = \{\Lambda \in \Sub(\Gamma), \Lambda \cap O = \emptyset \text{ \ and \ } I \subseteq \Lambda  \}\]
	where $O$ and $I$ are finite subsets of $\Gamma$. In other words, a sequence of subgroups $\Lambda_n$ converges to a subgroup $\Lambda$ if and only if \begin{itemize}
		\item for every $\lambda \in \Lambda$, there exists $n_0 \in \mathbb{N}$, for all $n \geq n_0$, $\lambda \in \Lambda_n$. 
		\item for every $\lambda \notin \Lambda$, there exists $n_0 \in \mathbb{N}$, for all $n \geq n_0$, $\lambda \notin \Lambda_n$.
	\end{itemize}
	
	In the formalism of Schreier graphs, given a finite generating set $S$ of $\Gamma$, a basis of neighborhoods of a pointed transitive right action $[\alpha,v]$: $(X,v) \curvearrowleft \Gamma$ is given by \[V_R = \{[\beta,v'], \Sch(\beta,v') \simeq_R \Sch(\alpha, v)\}.\]
	
	Applying Cantor-Bendixson theorem (see \cite[Chapter 1, Section 6]{Kechris} for instance) to the Polish space $\Sub(\Gamma)$ leads to a unique decomposition $\Sub(\Gamma) = K \sqcup D$ where $D$ is countable and $K$ is a closed subspace of $\Sub(\Gamma)$ without isolated points. The set $K$ is called the \textbf{perfect kernel} of $\Gamma$ and denoted by $\mathcal{K}(\Gamma)$. The perfect kernel of $\Gamma$ is the largest closed subset of $\Sub(\Gamma)$ without isolated points. It is also the set of subgroups all of whose neighborhoods are uncountable. \newline
	\newline
	From now we fix a GBS group $\Gamma$ defined by a reduced graph of groups $\mathcal{H}$ and a spanning tree $T$ in $\mathcal{H}$. For every edge $e \in \mathcal{E}\left(\mathcal{H}\right)$, we denote by $k_{e,\src}$ (\textit{resp.} $k_{e,\trg}$) the label of the half-edge of $e$ containing $\src(e)$ (\textit{resp.} $\trg(e)$). 
	A path $e_1,...,e_r$ in $\mathcal{H}$ is said to be \textbf{labeled} $(k_1,l_1),...,(k_r,l_r)$ if $k_{e_i,\src} = k_i$ and $k_{e_i,\trg} = l_i$ for any $i \in \llbracket 1,r \rrbracket$. For each vertex $s \in \mathcal{V}\left(\mathcal{H}\right)$, choose a generator $a_s$ of the vertex group $G_s$ and for every $e \notin T$ denote by $t_e$ the associated edge generator in the presentation \eqref{pres1} of $\Gamma$ defined by $T$. We denote by $\mathcal{A}$ the associated Bass-Serre tree, $\pi: \mathcal{A} \to \mathcal{H} = \Gamma \backslash \mathcal{A}$ the projection and a lifting of $T$ which we extend to a section $\sigma : \mathcal{E}\left(\mathcal{H}\right) \to \mathcal{E}\left(\mathcal{A}\right)$ satisfying
	\begin{itemize}
		\item $\sigma(\overline{e}) = \overline{\sigma(e)}$ for every $e \in \mathcal{E}^+\left(\mathcal{H}\right)$;
		\item $\trg\left(\sigma(e)\right) = \sigma\left(\trg(e)\right)$ for every $e \in \mathcal{E}^+\left(\mathcal{H}\right)$;
		\item $\Stab\left(\sigma(v)\right) = \langle a_v \rangle$ for every $v \in \mathcal{V}\left(\mathcal{H}\right)$;
		\item $\Stab\left(\sigma(e)\right) = \left\langle a_{\trg(e)}^{k_{e, \trg}} \right\rangle$ for every $e \in \mathcal{E}^+\left(\mathcal{H}\right)$.
	\end{itemize}
	
\section{\texorpdfstring{From an $\mathcal{H}$-preaction to an $\mathcal{H}$-graph}{From an H-preaction to an H-graph}}\label{corr}	
	The goal of this section is to introduce some tools to approximate a $\Gamma$-action by perturbing a graph associated to it far from the base point. To achieve this, we introduce the notion of \textbf{$\mathcal{H}$-preaction}, that takes the defining relations of $\Gamma$ into account. To any $\mathcal{H}$-preaction will be associated an \textbf{$\mathcal{H}$-graph} (that is related to its Schreier graph). It is a labeled graph whose labels and degrees satisfy some arithmetical conditions. The inspiration comes from the notion of core graphs in the case of free groups: to any subgroup $H$ of $\mathbb{F}_r$ is associated a pointed graph endowed with an immersion to the bouquet of $r$ circles. Perturbing this graph very far from the origin amounts to building a subgroup of $\mathbb{F}_r$ that is very close from $H$. The new notions we introduce are generalizations of the notions of preactions and $(m,n)$-graphs introduced in \cite{solitar1}. The authors of \cite{hightrans} also introduced a similar notion of Bass-Serre graphs and of preactions in the case of HNN extensions and amalgamated free products. Our definition of an $\mathcal{H}$-preaction extends to any graph of GBS group $\mathcal{H}$ and depends on the choice of the spanning tree $T$ of $\mathcal{H}$.

Before entering into technical details, let us motivate the notion of $\mathcal{H}$-preactions. If one wants to perturb the Schreier graph of an $\mathbb{F}_2$-action (with respect to a basis $\{a,b\}$ of $\mathbb{F}_2$), one has a lot of freedom: we only have to ensure that each vertex has exactly one incoming edge and one outgoing edge labeled $a$ (\textit{resp.} $b$). When the group we start with is not free, nothing similar can be done in general. In the case of the fundamental group of a graph of groups, we can adapt the naive strategy that worked for free groups, but the situation is significantly more subtle. For instance, in the case of Baumslag-Solitar groups, the defining relation forces the cardinals of the $\langle b \rangle$-orbits of any Schreier graph to satisfy some arithmetical relations: for example, one $\langle b \rangle$-orbit is infinite if and only if all $\langle b \rangle$-orbits are. The idea of introducing the notion of preaction is to get rid of this difficulty by imposing that the non-saturated actions we work with have saturated $\langle b \rangle$-orbits - and the defining relation imposes that, whenever $t$ (\textit{resp.} $t^{-1}$) is defined on some point $x$, then $t$ has to be defined on $x \cdot b^n$ (\textit{resp.} $x \cdot b^m$). In the case of an arbitrary graph of GBS group, this means that all the orbits of the vertex groups are saturated, and that all the relations of the group are ``memorized'' by the preaction. Let us make this definition more precise:
	
	\begin{definition}\label{preac}
		
		An \textbf{$\mathcal{H}$-preaction} of $\Gamma$ on a set $X$ is a set of partially defined bijections $\alpha_s : \dom(\alpha_s) \to \rng(\alpha_s)$ of $X$ (for each $s \in \mathcal{V}\left(\mathcal{H}\right)$) and $\tau_e : \dom(\tau_e) \to \rng(\tau_e)$ of $X$ (for each edge $e \notin T$) satisfying the following conditions:  \begin{enumerate}\label{hpreac}
			\item For every vertex $s$ of $\mathcal{H}$, $\dom(\alpha_s) = \rng(\alpha_s)$;
			\item For every edge $e$ of $T$, for every $x \in \dom \left(\alpha_{\src(e)}\right) \cap \dom \left(\alpha_{\trg(e)} \right)$, \[x \alpha_{\src(e)}^{k_{e,\src}} = x \alpha_{\trg(e)}^{k_{e,\trg}};\]
			\item For every edge $e \in \mathcal{E}\left(\mathcal{H}\right) \setminus \mathcal{E}\left(T\right)$ \begin{itemize}
				\item $\dom(\tau_e)$ is $\alpha_{\src(e)}^{k_{e,\src}}$-invariant;
				\item $\dom(\tau_e) = \rng(\tau_e^{-1})$ and $\tau_e = \tau_{\overline{e}}^{-1}$
			%	\item $\dom(\tau_e) \subseteq \dom(\alpha_{\src(e)})$
			\end{itemize}
			and for every $x \in \dom(\tau_e)$ \[x \alpha_{\src(e)}^{k_{e,\src}}\tau_e = x \tau_e \alpha_{\trg(e)}^{k_{e,\trg}};\]
			\item \label{chem1} For every edge $e \in \mathcal{E}\left(\mathcal{H}\right) \setminus \mathcal{E}\left(T\right)$, one has \[\dom(\tau_e) \subseteq \dom(\alpha_{\src(e)});\]
		%	\item \label{chem1} For every vertex $s \in \mathcal{V}\left(\mathcal{H}\right)$ and edge $e \in \mathcal{E}\left(\mathcal{H}\right) \setminus \mathcal{E}\left(T\right)$, denoting by $e_1,...,e_n$ the unique edge path in $T$ connecting $s$ to $\src(e)$, one has, for every $i \in \llbracket 1,n \rrbracket$ and $\var \in \{\src,\trg\}$: \[\dom(\alpha_s) \cap \dom(\tau_e) \subseteq \dom(\alpha_{{\var}(e_i)});\]
			\item \label{chem2} For every vertex $v,w \in \mathcal{V}\left(\mathcal{H}\right)$, denoting by $e_1,...,e_n$ the unique edge path in $T$ connecting $v$ to $w$, one has, for every $i \in \llbracket 1,n \rrbracket$ and $\var \in \{\src, \trg\}$: \[\dom(\alpha_v) \cap \dom(\alpha_w) \subseteq \dom(\alpha_{{\var}(e_i)}).\]
		\end{enumerate}
	\end{definition}

	Let $\alpha$ be an $\mathcal{H}$-preaction on a countable set $X$ defined by a set of partial bijections $\left\{\alpha_s, \tau_e | s \in \mathcal{V}\left(\mathcal{H}\right), e \in \mathcal{E}\left(\mathcal{H}\right) \setminus \mathcal{E}(T)\right\}$. An $\mathcal{H}$-preaction $\alpha'$ on a countable set $X'$ defined by partial bijections $\left\{\alpha'_s, \tau'_e | s \in \mathcal{V}\left(\mathcal{H}\right), e \in \mathcal{E}\left(\mathcal{H}\right) \setminus \mathcal{E}(T)\right\}$ \textbf{extends} the $\mathcal{H}$-preaction $\alpha$ if there exists an injection $\iota : X \hookrightarrow X'$ such that \begin{itemize}
		\item for every vertex $s \in \mathcal{V}\left(\mathcal{H}\right)$, $\iota(\dom(\alpha_s)) \subseteq \dom(\alpha'_s)$, and on $\dom(\alpha_s)$ one has $\alpha'_s \circ \iota = \iota \circ \alpha_s$;
		\item for every edge $e \in \mathcal{E}\left(\mathcal{H}\right) \setminus \mathcal{E}(T)$, $\iota(\dom(\tau_e)) \subseteq \dom(\tau'_e)$, and on $\dom(\tau_e)$ one has $\tau'_e \circ \iota = \iota \circ \tau_e$.
	\end{itemize}
	In that case we say that $\alpha$ is a \textbf{sub $\mathcal{H}$-preaction} of $\alpha'$.
	Given two $\mathcal{H}$-preactions $\alpha^{(i)} = \left\{\alpha^{(i)}_s, \tau^{(i)}_e \vert s \in \mathcal{V}\left(\mathcal{H}\right), e \in \mathcal{E}\left(\mathcal{H}\right) \setminus \mathcal{E}(T)\right\}$ (for $i = 1, 2$) on countable sets $X_1$ and $X_2$, respectively, one can define the \textbf{disjoint union} $$\alpha^{(1)} \sqcup \alpha^{(2)} = \left\{\alpha_s, \tau_e \vert s \in \mathcal{V}\left(\mathcal{H}\right), e \in \mathcal{E}\left(\mathcal{H}\right) \setminus \mathcal{E}(T)\right\}$$ as an $\mathcal{H}$-preaction on $X_1 \sqcup X_2$ by \begin{itemize}
		\item $\dom\left(\alpha_s\right) = \dom\left(\alpha^{(1)}_s\right) \sqcup \dom\left(\alpha^{(2)}_s\right)$ with $\alpha_s\Big|_{\dom\left(\alpha^{(1)}_s\right)} = \alpha^{(1)}_s$ and $\alpha_s\Big|_{\dom\left(\alpha^{(2)}_s\right)} = \alpha^{(2)}_s$ for every $s \in \mathcal{V}\left(\mathcal{H}\right)$;
		\item $\dom\left(\tau_e\right) = \dom\left(\tau^{(1)}_e\right) \sqcup \dom\left(\tau^{(2)}_e\right)$ with $\tau_e\Big|_{\dom\left(\tau^{(1)}_e\right)} = \tau^{(1)}_e$ and $\tau_e\Big|_{\dom\left(\tau^{(2)}_e\right)} = \tau^{(2)}_e$.
	\end{itemize}
	The $\mathcal{H}$-preaction $\alpha^{(1)} \sqcup \alpha^{(2)}$ extends both $\alpha^{(1)}$ and $\alpha^{(2)}$. 
	
	An $\mathcal{H}$-preaction $\alpha = \{\alpha_s, \tau_e \vert s \in \mathcal{V}\left(\mathcal{H}\right), e \in \mathcal{E}\left(\mathcal{H}\right)\setminus \mathcal{E}(T)\}$ on a countable set $X$ is called \textbf{transitive} if for any $x, y \in X$, there exist $\delta_1, ..., \delta_r \in \alpha$ such that $y = x \delta_1 \ldots \delta_r$. It is \textbf{saturated} if it comes from a genuine action of $\Gamma$ (\textit{i.e.} the domains of the partial bijections defining $\alpha$ are equal to $X$). If $\alpha$ is a transitive and saturated $\mathcal{H}$-preaction, we will simply call $\alpha$ an \textbf{$\mathcal{H}$-action}. 
	
	\begin{remark}
	The second and third conditions in the definition of an $\mathcal{H}$-preaction force the partial bijections to satisfy the relations in the presentation \eqref{pres1} of $\Gamma$. The last two conditions (which did not appear in the definition of a preaction in \cite{solitar1} because the spanning tree was reduced to a single vertex) are necessary to extend a given $\mathcal{H}$-preaction into a genuine action of $\Gamma$. For instance let $\Gamma$ be the fundamental group of the graph of groups defined in Figure \ref{figureexemple}.
	
	\begin{figure}[ht]
		\center
		\begin{tikzpicture}
			\node[draw,circle,fill=gray!50] (se) at (0,0) {};
			\node[draw,circle,fill=white] (te) at (2,-1) {};
			\node[draw,circle,fill=black] (tf) at (4,0) {};
			\draw[>=latex, directed] (se) -- (te) node[midway, below left]{$e$} node[very near start, below left]{2} node[very near end, below left]{3};
			\draw[>=latex, directed] (te) -- (tf) node[midway, below right]{$f$} node[very near start, below right]{3} node[very near end, below right]{2};
			\draw[red, >=latex, directed] (se) -- (tf) node[midway, above]{$g$} node[very near start, above]{2} node[very near end, above]{2};
		\end{tikzpicture}
		\caption{The graph $\mathcal{H}$.}
		\label{figureexemple}
	\end{figure}

	Let us chose the set of edges $\{e,f\}$ to define the spanning tree $T$.
	
	Let us consider the collection of partial bijections defined in Figure \ref{figurenonextarbre} (here, whenever some vertex $x$ has no outgoing edge labeled $\alpha_s$ (\textit{resp.} $\tau_e$), it means that $x \notin \dom(\alpha_s)$ (\textit{resp.} $x \notin \dom(\tau_e)$)). These partial bijections satisfy all conditions in the definition of an $\mathcal{H}$-preaction except the fourth one (because $\dom\left(\tau_{\overline{g}}\right) \subsetneq \dom\left(\alpha_{\trg(g)}\right)$), and cannot be extended into a genuine preaction because the relation $t_g^{-1} a_{\src(e)}^2 t_g = a_{\trg(f)}^2 = a_{\trg(e)}^3 = a_{\src(e)}^2$ cannot be satisfied (because $t_g^{-1}a_{\src(e)}^2t_g$ sends the black square vertex to the other $\langle  \alpha_{\src(e)}\rangle$-orbit).
	
		\begin{figure}[ht]
		\center
		\begin{tikzpicture}
			\node[draw, circle, fill=gray] (a) at (0,0) {};
			\node[draw,circle, fill=gray] (b) at (-1.5,1.5) {};
			\node[draw,circle, fill=gray] (c) at (-3,0) {};
			\draw[red, >=latex, directed] (c) to[out=240,in=150,looseness=20] node[left]{$\tau_g$} (c);
			\node[draw,circle,fill=gray] (d) at (-1.5,-1.5) {};
			\node[draw, fill=black] (a') at (1.5,0) {};
			\node[draw,circle,fill=gray] (e) at (3,1.5) {};
			\node[draw,circle,fill=gray] (f) at (4.5,0) {};
			\node[draw,circle,fill=gray] (g) at (3,-1.5) {};
			\draw[gray!60, >=latex, directed] (a) to[bend right] node[midway, below left]{$\alpha_{\src(e)}$} (b);
			\draw[gray!60, >=latex, directed] (b) to[bend right] (c);
			\draw[gray!60, >=latex, directed] (c) to[bend right] (d);
			\draw[gray!60, >=latex, directed] (d) to[bend right] (a);
			\draw[gray!60, >=latex, directed] (a') to[bend left] node[midway, below right]{$\alpha_{\src(e)}$} (e);
			\draw[gray!60, >=latex, directed] (e) to[bend left] (f);
			\draw[gray!60, >=latex, directed] (f) to[bend left] (g);
			\draw[gray!60, >=latex, directed] (g) to[bend left] (a');
			\draw[red, >=latex, directed] (a) -- (a') node[midway, above]{$\tau_g$};
		\end{tikzpicture}
		\caption{A collection of partial bijections which cannot be extended into a genuine $\Gamma$-action.}
		\label{figurenonextarbre}
	\end{figure}

	Likewise, let us consider the set of partial bijections defined in Figure \ref{figurenonext}.

	\begin{figure}[ht]
		\center
		\begin{tikzpicture}
			\node[draw, circle, fill=gray] (a) at (0,0) {};
			\node[draw,circle,fill=gray] (b) at (-1.5,1.5) {};
			\node[draw,circle,fill=gray] (c) at (-3,0) {};
			\node[draw,circle,fill=gray] (d) at (-1.5,-1.5) {};
			\node[draw,circle,fill=gray] (e) at (1.5,1.5) {};
			\node[draw,circle,fill=gray] (f) at (3,0) {};
			\node[draw,circle,fill=gray] (g) at (1.5,-1.5) {};
			\draw[gray!50, >=latex, directed] (a) to[bend right] node[midway, below left]{$\alpha_{\src(e)}$} (b);
			\draw[gray!50, >=latex, directed] (b) to[bend right] (c);
			\draw[gray!50, >=latex, directed] (c) to[bend right] (d);
			\draw[gray!50, >=latex, directed] (d) to[bend right] (a);
			\draw[black, >=latex, directed] (a) to[bend left] node[midway, below right]{$\alpha_{\trg(f)}$} (e);
			\draw[black, >=latex, directed] (e) to[bend left] (f);
			\draw[black, >=latex, directed] (f) to[bend left] (g);
			\draw[black, >=latex, directed] (g) to[bend left] (a);
		\end{tikzpicture}
		\caption{A collection of partial bijections which cannot be extended into a genuine $\Gamma$-action.}
		\label{figurenonext}
	\end{figure}
	
	Then, the first four conditions for being an $\mathcal{H}$-preaction are satisfied, but not the fifth one. This cannot be extended into a genuine action of $\Gamma$ because the relation $a_{\src(e)}^2 = a_{\trg(e)}^3 = a_{\trg(f)}^2$ cannot be satisfied.

	\end{remark}
	
	From Conditions \ref{chem1} and \ref{chem2} we get the following lemma, which will be useful to merge preactions:
	\begin{lemma}\label{defdisjoint}
	    Let $\alpha$ be an $\mathcal{H}$-preaction on a countable set $X$. Let $x,y \in X$ such that there exists an edge $e_0 \in \mathcal{E}(\mathcal{H})$ such that, denoting by $s_0, t_0 := \src(e_0), \trg(e_0)$, one has:\begin{itemize}
	        \item $x \in \dom\left(\alpha_{s_0}\right)$ and $x \notin \dom\left(\alpha_{t_0}\right)$;
	        \item $y \in \dom\left(\alpha_{t_0}\right)$ and $y \notin \dom\left(\alpha_{s_0}\right)$.
	    \end{itemize}
	    Then, the set of partial bijections defined on $x$ is disjoint from the set of partial bijections defined on $y$.
	\end{lemma}
	
	\begin{proof}
	Let us assume by contradiction the existence of a vertex $s \in \mathcal{V}\left(\mathcal{H}\right)$ such that $x \in \dom(\alpha_s)$ and $y \in \dom\left({\alpha}_s\right)$. By symmetry between $e_0$ and $\overline{e_0}$, one can assume that there exists a reduced edge path $e_0,e_1,...,e_r \in \mathcal{E}(T)$ with $\trg(e_r) = s$. Hence $x \in \dom(\alpha_{s_0}) \cap \dom(\alpha_s)$ which implies that $x \in \dom\left(\alpha_{t_0}\right)$ by Condition \ref{chem2} in Definition \ref{hpreac}. Hence we get a contradiction in this case. If there exists an edge $e \in \mathcal{E}\left(\mathcal{H}\right)$ such that $x \in \dom\left(\tau_e\right)$ and $y \in \dom\left(\tau_e\right)$, then Condition \ref{chem1} in Definition \ref{hpreac} implies that $x$ and $y$ belong to $\dom(\alpha_{\src(e)})$, which leads to a contradiction by the previous argumentation.
	\end{proof}
	
		Now we introduce the notion of $\mathcal{H}$-graph of a preaction, which will allow us to perturb a Schreier graph. Informally, the $\mathcal{H}$-graph of a preaction is its Schreier graph all of whose orbits of vertex groups have been shrunk. For every vertex $s \in \mathcal{V}\left(\mathcal{H}\right)$, we draw one vertex per $\left\langle \alpha_s \right\rangle$-orbit and label it $(s,N)$ where $N$ is the cardinality of the orbit. There is an edge labeled $e \in \mathcal{E}(\mathcal{H})$ between $x \langle \alpha_s \rangle$ and $x' \langle \alpha_{s'} \rangle$ if and only if $\src(e) = s$, $\trg(e) = s'$, and \begin{itemize}
		    \item either $e \in \mathcal{E}(T)$ and $x \langle \alpha_s \rangle \cap x' \langle \alpha_{s'} \rangle \neq \emptyset$;
		    \item or $e \notin \mathcal{E}(T)$ and $x \langle \alpha_s \rangle \tau_e \cap x' \langle \alpha_{s'} \rangle \neq \emptyset$.
		\end{itemize}
	
	More formally:
	
	\begin{definition}\label{graph}
		Let $\alpha$ be an $\mathcal{H}$-preaction of $\Gamma$ on a countable set $X$. The \textbf{$\mathcal{H}$-graph} $\mathcal{G}$ of $\alpha$ is defined by the following data: \begin{itemize}
			\item Its vertex set is \[\mathcal{V}\left(\mathcal{G}\right) = \bigsqcup_{s \in \mathcal{V}\left(\mathcal{H}\right)}\dom(\alpha_s) / \left\langle \alpha_s \right\rangle;\]
			\item Its edge set is $\mathcal{E}\left(\mathcal{G}\right) = \mathcal{E}^+(\mathcal{G}) \bigsqcup \mathcal{E}^-(\mathcal{G})$ where
			\begin{align*}\mathcal{E}^+&(\mathcal{G}) = \bigsqcup_{e \in \mathcal{E}^+(T)} \dom \left(\alpha_{\src(e)} \right) \cap \dom \left(\alpha_{\trg(e)} \right) / \left\langle \alpha_{\src(e)}^{k_{e,\src}} \right\rangle \\ &\bigsqcup \bigsqcup_{e \in \mathcal{E}^+(\mathcal{H}) \setminus \mathcal{E}^+(T)} \dom(\tau_e) / \left\langle \alpha_{\src(e)}^{k_{e,\src}} \right\rangle \end{align*}
			and 
			\begin{align*}\mathcal{E}^-&(\mathcal{G}) = \bigsqcup_{e \in \mathcal{E}^+(T)} \dom \left(\alpha_{\src(e)} \right) \cap \dom \left(\alpha_{\trg(e)} \right) / \left\langle \alpha_{\trg(e)}^{k_{e,\trg}} \right\rangle \\ &\bigsqcup \bigsqcup_{e \in \mathcal{E}^+(\mathcal{H}) \setminus \mathcal{E}^+(T)} \rng(\tau_e) / \left\langle \alpha_{\trg(e)}^{k_{e,\trg}} \right\rangle \end{align*}
			with source and target maps defined by 
			\begin{itemize}
				\item for every $e \in \mathcal{E}^+(T)$ and every $x \in \dom \left(\alpha_{\src(e)} \right) \cap \dom \left(\alpha_{\trg(e)}\right)$: \[\src \left(x \left\langle \alpha_{\src(e)}^{k_{e,\src}} \right\rangle \right) = x \left\langle \alpha_{\src(e)} \right\rangle\]
				and \[\trg \left(x \left\langle \alpha_{\src(e)}^{k_{e,\src}} \right\rangle \right) = x \left\langle \alpha_{\trg(e)} \right\rangle.\]
				Moreover the involution is given by \[\overline{x \left\langle \alpha_{\src(e)}^{k_{e,\src}} \right\rangle } = x \left\langle \alpha_{\trg(e)}^{k_{e,\trg}} \right\rangle;\]
				\item for every $e \in \mathcal{E}^+(\mathcal{H}) \setminus \mathcal{E}^+(T)$ and $x \in \dom(\tau_e)$ \[\src \left(x \left\langle \alpha_{\src(e)}^{k_{e,\src}} \right\rangle \right) = x \left\langle \alpha_{\src(e)} \right\rangle\] and \[\trg \left(x \left\langle \alpha_{\src(e)}^{k_{e,\src}} \right\rangle \right) = x \tau_e \left\langle \alpha_{\trg(e)} \right\rangle.\]
				Moreover \[\overline{x \left\langle \alpha_{\src(e)}^{k_{e,\src}} \right\rangle } = x \tau_e \left\langle \alpha_{\trg(e)}^{k_{e,\trg}} \right\rangle;\]
				\item Each vertex $x \left\langle \alpha_s \right\rangle$ is labeled $(s,L)$ where $L$ is the cardinality of the $\left\langle \alpha_s \right\rangle$-orbit $L = |x \left\langle \alpha_s \right\rangle|$;
				\item Each edge $x \left\langle \alpha_{\src(e)}^{k_{e,\src}} \right\rangle$ is labeled $e$;
				\item Each edge $x \left\langle \alpha_{\trg(e)}^{k_{e,\trg}} \right\rangle$ is labeled $\overline{e}$. 
			\end{itemize}
		\end{itemize}
	\end{definition}
	Notice that the $\mathcal{H}$-graph of an $\mathcal{H}$-preaction $\alpha$ is connected if and only if $\alpha$ is transitive.
	Given an $\mathcal{H}$-preaction $\alpha$ on a countable set $X$, a vertex $s \in \mathcal{V}\left(\mathcal{H}\right)$ and a vertex $V = x \langle \alpha_s \rangle$ of its $\mathcal{H}$-graph (for some $x \in \dom(\alpha_s)$), we say that the vertex $V$ \textbf{derives from} $x$. Likewise, given an edge $e \in \mathcal{E}\left(\mathcal{H}\right)$ and an edge $E = x \left\langle \alpha_{\src(e)}^{k_{e, \src}} \right\rangle$, we say that the edge $E$ \textbf{derives from} $x$. \\
	Observe that if an $\mathcal{H}$-preaction $\alpha'$ extends $\alpha$, then the $\mathcal{H}$-graph of $\alpha$ is a subgraph of the $\mathcal{H}$-graph of $\alpha'$. 
	
	\begin{definition}
		A vertex $V$ of an $\mathcal{H}$-graph $\mathcal{G}$ of an $\mathcal{H}$-preaction which is labeled $(s,N)$ is a vertex of \textbf{type} $s$. An edge $E$ of an $\mathcal{H}$-graph which is labeled $e$ is an edge of \textbf{type} $e$.
	\end{definition}

\begin{definition}
    The $\mathcal{H}$-preaction associated to a subgroup $\Lambda$ of $\Gamma$ is defined by the $\Gamma$-action $\Lambda \backslash \Gamma \curvearrowleft \Gamma$. The $\mathcal{H}$-graph of a subgroup $\Lambda$ of $\Gamma$ is the $\mathcal{H}$-graph of the associated $\Gamma$-action. 
\end{definition}

\begin{remark}
    Two conjugate subgroups of $\Gamma$ have isomorphic $\mathcal{H}$-graphs.
\end{remark}
	
	There is a correspondence between $\mathcal{H}$-graphs of subgroups and graphs of subgroups $\Lambda \backslash \mathcal{A}$. Keeping the same notations as in Section \ref{bassserre}, recall that the Bass-Serre tree is the quotient graph $Y / \mathbb{Z}$, thus the quotient graph $\Lambda \backslash (Y/\mathbb{Z})$ is precisely $\Lambda \backslash \mathcal{A}$. As left and right actions commute, this graph is also obtained by quotienting $Y$ by the $\Lambda$-(left) action, and then by the $\mathbb{Z}$-(right) action. The quotient $\left(\Lambda \backslash Y\right) / \mathbb{Z}$ is (up to labelling the vertices and edges) the $\mathcal{H}$-graph of $\Lambda$.  
	
	Each vertex $V = (\Lambda \gamma \left\langle a_s \right\rangle, s)$ of $\Lambda \backslash (Y / \mathbb{Z})$ is labeled $(s,N)$, where \begin{align*}
		N &= \left|\left\{\Lambda \gamma a_s^k \vert k \in \mathbb{Z}\right\}\right| \\
		&= \left[\left\langle a_s \right\rangle: \gamma^{-1}\Lambda\gamma \cap \left\langle a_s \right\rangle\right] \\
		&=\left[\Stab_{\Gamma}(\gamma \sigma{s}): \Stab_{\Lambda}(\gamma \sigma{s})\right] 
	\end{align*} (recall that $\sigma$ is the section defined in Section \ref{bassserre}).
	
	Hence the data of the label of $V$ (in the $\mathcal{H}$-graph of $\Lambda$) is equivalent to the data of the vertex group (in the graph of groups of $\Lambda$). 

\begin{propdef}\label{ggraph}
Let $\mathcal{G}$ be the $\mathcal{H}$-graph of an $\mathcal{H}$-preaction $\alpha$. Then it satisfies the following conditions: 
	\begin{enumerate}
			\item Each vertex $V$ of $\mathcal{G}$ is labeled $(s, N)$ for some $s \in \mathcal{V}\left(\mathcal{H}\right)$ and $N \in \mathbb{N} \cup \{\infty\}$. The vertex $s$ is the \textbf{type} of $V$.
			\item Each edge $E$ of $\mathcal{G}$ is labeled $e$ for some $e \in \mathcal{E}\left(\mathcal{H}\right)$ and the vertices $\src(E)$ and $\trg(E)$ are of type $\src(e)$ and $\trg(e)$, respectively. The edge $e$ is the \textbf{type} of $E$.
			\item (\textbf{Transfer Equation}) For every edge $E$ of $\mathcal{G}$ labeled $e$ with source labeled $(\src(e),N)$ and target labeled $(\trg(e), M)$, one has \begin{equation}\label{gbs1transfert} \frac{N}{N \wedge k_{e,\src}} = \frac{M}{M \wedge k_{e,\trg}}; \end{equation}
			\item For every vertex $V$ of $\mathcal{G}$ labeled $(\src(e),N)$, there are at most $N \wedge k_{e,\src}$ incident edges to $V$ labeled $e$ in $\mathcal{G}$;
			\item For every vertex $V$ of $\mathcal{G}$ labeled $(\trg(e),N)$, there are at most $N \wedge k_{e,\trg}$ incident edges to $V$ labeled $\overline{e}$ in $\mathcal{G}$.
		\end{enumerate}
		An abstract graph that satisfies the  previous conditions $(1)$ to $(5)$ is called an \textbf{$\mathcal{H}$-graph}. The $\mathcal{H}$-preaction $\alpha$ is saturated if and only if equality holds in the last two points.
\end{propdef}

	\begin{proof}
		The first two conditions are clear by definition of the $\mathcal{H}$-graph of an $\mathcal{H}$-preaction. Let us check the conditions $(3)$ to $(5)$. \\
		Let $E$ be an edge of $\mathcal{G}$ and let us denote by $e \in \mathcal{E}\left(\mathcal{H}\right)$ its label. Let us denote by $(\src(e),N)$ the label of $\src(E)$ and by $(\trg(e), M)$ the label of $\trg(E)$. By Definition \ref{graph}, the edge $E$ witnesses the existence of \begin{itemize}
		    \item an $\langle \alpha_{\src(e)} \rangle$-orbit $O_1$ of cardinal $N$;
		    \item an $\langle \alpha_{\trg(e)} \rangle$-orbit $O_2$ of cardinal $M$
		\end{itemize}
		such that \begin{itemize}
		    \item if $e \in T$, then there exists $x \in O_1 \cap O_2$;
		    \item if $e \notin T$, then there exists $x \in O_1$ such that $x \in \dom(\tau_e)$ and $x \tau_e \in O_2$.
		\end{itemize}

		Let us define $f = \tau_e$ if $e \notin T$ and $f=\id$ otherwise. One has: \begin{align*}
			N &= \left|x \left\langle \alpha_{\src(e)}\right\rangle\right|; \\
			M &= \left|x f \left\langle \alpha_{\trg(e)}\right\rangle\right|.
		\end{align*}
		Then, by the definition of an $\mathcal{H}$-preaction \ref{preac} (item 2 and 3), $f$ induces a bijection between $x \left\langle \alpha_{\src(e)}^{k_{e,\src}} \right\rangle$ and $xf \left\langle \alpha_{\trg(e)}^{k_{e,\trg}} \right\rangle$, which have cardinalities $\frac{N}{N \wedge k_{e,\src}}$ and $\frac{M}{M \wedge k_{e, \trg}}$, respectively (or $\infty$ if $N = \infty$). Hence we deduce the Transfer Equation \ref{gbs1transfert}. 
		
		The number of $\left\langle \alpha_{\src(e)}^{k_{e,\src}} \right\rangle$-orbits contained in $x \left\langle \alpha_{\src(e)}\right\rangle$ is $N \wedge k_{e, \src}$. As the number of incident edges to $\src(E)$ labeled $e$ is precisely the number of $\left\langle \alpha_{\src(e)}^{k_{e,\src}}\right\rangle$-orbits in $x \left\langle \alpha_{\src(e)}\right\rangle$ on which $f\alpha_{\trg(e)}$ is defined, one deduces the fourth item. Replacing $e$ by $\overline{e}$ leads to the fifth item. Furthermore, equality holds if and only if $f$ is defined on the entire $\left\langle \alpha_{\src(e)} \right\rangle$-orbit. Hence equality holds for every vertices and edges if and only if every partial bijection defining $\alpha$ is defined on the entire set $X$, that is to say, if and only if $\alpha$ is a $\Gamma$-action. 
	\end{proof}
	
	\begin{definition}
		Let $\mathcal{G}$ be an $\mathcal{H}$-graph and $V \in \mathcal{V}\left(\mathcal{G}\right)$ a vertex labeled $(s,N)$. The vertex $V$ is \textbf{saturated relatively to an edge} $e \in \mathcal{E}^+(\mathcal{H})$ with source $s$ if equality holds in the last two points of Proposition-Definition~\ref{ggraph}. It is said to be \textbf{saturated} if it is saturated relatively to every edge $e$ with source $s$. The graph $\mathcal{G}$ is called saturated if every vertex of $\mathcal{G}$ is saturated.
	\end{definition}
	Notice that the $\mathcal{H}$-graph of a transitive $\mathcal{H}$-preaction $\alpha$ is saturated if and only if $\alpha$ is saturated. 
	
\section{\texorpdfstring{From an $\mathcal{H}$-graph to an $\mathcal{H}$-preaction}{From an H-graph to an H-preaction}}

Our goal is now to build an $\mathcal{H}$-preaction given an $\mathcal{H}$-graph, the final goal being \begin{enumerate}
    \item the merging of $\mathcal{H}$-preactions (using Constructions \hyperlink{ext}{A} and \hyperlink{quot}{B});
    \item the extension of a non-saturated $\mathcal{H}$-preaction (using Constructions \hyperlink{ext'}{A'} and \hyperlink{quot'}{B'});
    \item using these procedures to obtain an $\mathcal{H}$-preaction from an $\mathcal{H}$-graph (Lemma \ref{arete} and Lemma \ref{retour}).
\end{enumerate}
To achieve this, we prove that suitable mergings (or extensions) of $\mathcal{H}$-graphs of $\mathcal{H}$-preactions lead to mergings (or extensions) of the $\mathcal{H}$-preactions. 

	 Now let us enter into the details of Constructions \hyperlink{ext}{A} and \hyperlink{quot}{B}. Given an $\mathcal{H}$-preaction $\alpha$ defined by a set of partial bijections $\{\alpha_s, \tau_e \vert (s,e) \in \mathcal{V}\left(\mathcal{H}\right) \times \mathcal{E}\left(\mathcal{H}\right) \setminus \mathcal{E}(T)\}$ on a countable set $X$, these two procedures will allows us \begin{itemize}
		\item either to extend $\alpha$ into another $\mathcal{H}$-preaction (Construction \hyperlink{ext}{A});
		\item or to identify some points in $X$ and to define the $\mathcal{H}$-preaction induced by $\alpha$ on the quotient set (Construction \hyperlink{quot}{B}).
	\end{itemize}
	The effect of these constructions is to add an edge between two non-saturated vertices in the $\mathcal{H}$-graph of $\alpha$.
	
	We start with a non-saturated preaction $\alpha = \{\alpha_s, \tau_e \vert (s,e) \in \mathcal{V}\left(\mathcal{H}\right) \times \mathcal{E}\left(\mathcal{H}\right) \setminus \mathcal{E}(T)\}$ on a countable set. We draw the attention of the reader to the fact that $\alpha$ need not be transitive. \\
	\\
	\textbf{Construction \hypertarget{ext}{A}:} \textbf{Extending $\mathbf{\alpha}$.}
	Let $e_0 \in \mathcal{E}\left(\mathcal{H}\right) \setminus \mathcal{E}(T)$ and $x, y \in X$ such that, denoting by $(s_0,t_0) = (\src(e_0), \trg(e_0))$ and $(k,l) = \left(k_{e_0, \src}, k_{e_0, \trg}\right)$: \begin{enumerate}
		\item $x \in \dom\left(\alpha_{s_0}\right)$ and $x \notin \dom\left(\tau_{e_0}\right)$;
		\item $y \in \dom\left(\alpha_{t_0}\right)$ and $y \notin \rng\left(\tau_{e_0}\right)$;
		\item denoting by $N = \left|x\left\langle \alpha_{s_0} \right\rangle\right|$ and $M = \left|y\left\langle \alpha_{t_0} \right\rangle\right|$, one has \[\frac{N}{N \wedge k} = \frac{M}{M \wedge l}.\]
	\end{enumerate}
	In terms of $\mathcal{H}$-graphs, the first two items are equivalent to the existence of two vertices $V$, $W$ of type $s_0$ and $t_0$, respectively, in $\mathcal{G}$ which are not saturated relatively to $e_0$ and $\overline{e_0}$, respectively. 
	
	Let us extend the definition of $\tau_{e_0}$ as follows: define $$\begin{array}{ccccc} \widetilde{\tau_{e_0}} &:& 
		\dom\left(\tau_{e_0}\right) \cup x\left\langle \alpha_{s_0}^k \right\rangle & \to & \rng\left(\tau_{e_0}\right) \cup y\left\langle \alpha_{t_0}^l \right\rangle \\
		& & p & \mapsto & p \tau_{e_0} \text{ \ if $p \in \dom\left(\tau_{e_0}\right)$} \\
		& & x\alpha_{s_0}^{kj} & \mapsto & y\alpha_{t_0}^{lj}. \\
	\end{array}$$
	which is a partial bijection thanks to the equality $\frac{N}{N \wedge k} = \frac{M}{M \wedge l}$. Define $\widetilde{\tau_{\overline{e_0}}} := \widetilde{\tau_{e_0}}^{-1}$. Then the $\mathcal{H}$-preaction $\widetilde{\alpha} = \left(\alpha \setminus \{\tau_{e_0}, \tau_{\overline{e_0}}\}\right) \cup \{\widetilde{\tau_{e_0}}, \widetilde{\tau_{\overline{e_0}}}\}$ extends $\alpha$. 
	This construction is represented in Figure \ref{figureconsta}. 
	\begin{figure}[ht]
		\center
		\begin{tikzpicture}
			\node[draw,circle,fill=orange] (a) at (-2,0) {$s$};
			\node[draw,circle,fill=blue] (b) at (1,0) {\color{white} $t$};
			\draw[purple, >=latex, directed] (a) -- (b) node[midway, below]{\color{black} $e \notin T$} node[near start, above]{\color{black} $k=2$} node[near end, above]{\color{black} $l=3$};
			\node[draw,circle] (A) at (3,0) {};
			\node[draw,circle] (B) at (4.5,1.5) {};
			\draw (B) node[above]{$x \notin \dom(\tau_e)$};
			\node[draw,circle] (C) at (6,0) {};
			\node[draw,circle] (D) at (4.5,-1.5) {};
			\node[draw,circle] (E) at (7.5,1.5) {};
			\draw (E) node[above]{$y \notin \rng(\tau_e)$};
			\draw (D) node[below]{($N=4$)};
			\node[draw,circle] (F) at (7.5,-1.5) {};
			\draw (F) node[below]{($M=2$)};
			\draw[orange,->,>=latex] (B) to[bend right] node[midway, above left]{$\alpha_s$} (A);
			\draw[orange,->,>=latex] (A) to[bend right] (D);
			\draw[orange,->,>=latex] (D) to[bend right] (C);
			\draw[orange,->,>=latex] (C) to[bend right] (B);
			\draw[blue,->,>=latex] (E) to[bend left] node[right]{$\alpha_t$} (F);
			\draw[blue,->,>=latex] (F) to[bend left] (E);
			\draw[->,>=latex] (-1.75,-4.5) to node[midway, above]{Construction A} (1.25,-4.5);
			\node[draw,circle] (A') at (3,-4.5) {};
			\node[draw,circle] (B') at (4.5,-3) {};
			\draw (B') node[above]{$x$};
			\node[draw,circle] (C') at (6,-4.5) {};
			\node[draw,circle] (D') at (4.5,-6) {};
			\node[draw,circle] (E') at (7.5,-3) {};
			\draw (E') node[above]{$y$};
			\node[draw,circle] (F') at (7.5,-6) {};
			\draw[orange,->,>=latex] (B') to[bend right] node[midway, above left]{$\alpha_s$} (A');
			\draw[orange,->,>=latex] (A') to[bend right] (D');
			\draw[orange,->,>=latex] (D') to[bend right] (C');
			\draw[orange,->,>=latex] (C') to[bend right] (B');
			\draw[blue,->,>=latex] (E') to[bend left] node[right]{$\alpha_t$} (F');
			\draw[blue,->,>=latex] (F') to[bend left] (E');
			\draw[purple,->,>=latex] (B') to[bend right=15] node[midway, above left]{$\tau_e$} (E');
			\draw[purple,->,>=latex] (D') to[bend left=15] (F');
		\end{tikzpicture}
		\caption{Construction A.}
		\label{figureconsta}
	\end{figure} 
	\text{ } 	In terms of $\mathcal{H}$-graphs, the effect of Construction \hyperlink{ext}{A} is to add an edge labeled $e_0$ with source $V$ and target $W$.\\
	\\
	\textbf{Construction \hypertarget{quot}{B}:} \textbf{Quotienting $\mathbf{X}$.}
	Let $e_0 \in \mathcal{E}\left(T\right)$ and $x, y \in X$ such that, denoting by $(s_0,t_0) = (\src(e_0), \trg(e_0))$ and $(k,l) = (k_{e_0, \src}, k_{e_0, \trg})$: \begin{enumerate}
		\item $x \in \dom\left(\alpha_{s_0}\right)$ and $x \notin \dom\left(\alpha_{t_0}\right)$;
		\item $y \in \dom\left(\alpha_{t_0}\right)$ and $y \notin \dom\left(\alpha_{s_0}\right)$;
		\item denoting by $N = \left|x\left\langle \alpha_{s_0} \right\rangle\right|$ and $M = \left|y\left\langle \alpha_{t_0} \right\rangle\right|$ one has \[\frac{N}{N \wedge k} = \frac{M}{M \wedge l}.\]
	\end{enumerate}
	In terms of $\mathcal{H}$-graphs, the first two items are equivalent to the existence of two vertices $V$, $W$, of type $s_0$ and $t_0$ in $\mathcal{G}$ which are not saturated relatively to $e_0$ and $\overline{e_0}$, respectively. Let us define the quotient \[X' = X/\left(x\alpha_{s_0}^{kj} = y\alpha_{t_0}^{lj} \ \forall j \in \mathbb{Z}\right).\]
	\begin{figure}[ht]
		\center
		\begin{tikzpicture}
			\node[draw,circle,fill=orange] (a) at (-2,0) {$s$};
			\node[draw,circle,fill=blue] (b) at (1,0) {\color{white} $t$};
			\draw[purple, >=latex, directed] (a) -- (b) node[midway, below]{\color{black} $e \in T$} node[near start, above]{\color{black} $k=2$} node[near end, above]{\color{black} $l=3$};
			\node[draw,circle] (A) at (3,0) {};
			\node[draw,circle] (B) at (4.5,1.5) {};
			\draw (B) node[above]{$x \notin \dom(\alpha_t)$};
			\node[draw,circle] (C) at (6,0) {};
			\node[draw,circle] (D) at (4.5,-1.5) {};
			\node[draw,circle] (E) at (7.5,1.5) {};
			\draw (E) node[above]{$y \notin \dom(\alpha_s)$};
			\draw (D) node[below]{($N=4$)};
			\node[draw,circle] (F) at (7.5,-1.5) {};
			\draw (F) node[below]{($M=2$)};
			\draw[orange,->,>=latex] (B) to[bend right] node[midway, above left]{$\alpha_s$} (A);
			\draw[orange,->,>=latex] (A) to[bend right] (D);
			\draw[orange,->,>=latex] (D) to[bend right] (C);
			\draw[orange,->,>=latex] (C) to[bend right] (B);
			\draw[blue,->,>=latex] (E) to[bend left] node[right]{$\alpha_t$} (F);
			\draw[blue,->,>=latex] (F) to[bend left] (E);
			\draw[->,>=latex] (-1.75,-4.5) to node[midway, above]{Construction B} (1.25,-4.5);
			\node[draw,circle] (A') at (3,-4.5) {};
			\node[draw,circle] (B') at (4.5,-3) {};
			\draw (B') node[above]{$x=y \in \dom(\alpha_t) \cap \dom(\alpha_s)$};
			\node[draw,circle] (C') at (6,-4.5) {};
			\node[draw,circle] (D') at (4.5,-6) {};
			\draw[orange,->,>=latex] (B') to[bend right] node[midway, above left]{$\alpha_s$} (A');
			\draw[orange,->,>=latex] (A') to[bend right] (D');
			\draw[orange,->,>=latex] (D') to[bend right] (C');
			\draw[orange,->,>=latex] (C') to[bend right] (B');
			\draw[blue,->,>=latex] (B') to[bend left] node[right]{$\alpha_t$} (D');
			\draw[blue,->,>=latex] (D') to[bend left] (B');
		\end{tikzpicture}
		\caption{Construction B.}
		\label{figureconstb}
	\end{figure} 
	
	Observe that both orbits $x\left\langle \alpha_{s_0} \right\rangle$ and $y\left\langle \alpha_{t_0} \right\rangle$ inject in $X'$ because of the equality \begin{align*}
		\left|x \left\langle \alpha_{s_0}^k \right\rangle\right| &= \frac{N}{N \wedge k} \\
		&= \frac{M}{M \wedge l} \\
		&= \left|y \left\langle \alpha_{t_0}^l \right\rangle\right|. 
	\end{align*}
	By Lemma \ref{defdisjoint}, the domains and ranges of all the previously defined partial bijections inject in $X'$. Thus, $\alpha$ defines a genuine preaction on $X'$. This construction is represented in Figure \ref{figureconstb}. In terms of $\mathcal{H}$-graphs, the effect of Construction \hyperlink{quot}{B} is to add an edge labeled $e_0$ with source $V$ and target $W$.
	
	\begin{remark}\label{nonextingeneral}
	    Applying Construction \hyperlink{ext}{A} to an $\mathcal{H}$-preaction $\alpha$ always leads to an extension of $\alpha$. In contrast, Construction \hyperlink{quot}{B} cannot lead to an extension of $\alpha$. As an example, let us consider the GBS group defined by the graph of GBS groups represented in Figure \ref{exnonextbgr} (where the spanning tree is defined by the edge $f$), and let us define the preaction $\alpha$ as in Figure \ref{defalpha}.
	    
	    \begin{figure}[ht]
		\center
		\begin{tikzpicture}
	    	\node[draw,circle,fill=orange] (a) at (-2,0) {$s$};
			\node[draw,circle,fill=blue] (b) at (1,0) {\color{white} $t$};
			\draw[>=latex, directed, color=purple] (a) to[bend left] node[midway, above]{$e$} node[near start, above]{\color{black} $2$} node[near end, above]{\color{black} $3$} (b);
			\draw[>=latex, directed] (a) to[bend right] node[midway, below]{$f$} node[near start, below]{\color{black} $2$} node[near end, below]{\color{black} $3$} (b);
	   \end{tikzpicture}
	   \caption{A graph of GBS group} 
	   \label{exnonextbgr}
	   \end{figure}
	    
	    	\begin{figure}[ht]
		\center
		\begin{tikzpicture}
			\node[draw,circle] (A') at (3,-4.5) {};
			\node[draw,circle] (B') at (4.5,-3) {};
			\draw (B') node[above]{$x$};
			\node[draw,circle] (C') at (6,-4.5) {};
			\node[draw,circle] (D') at (4.5,-6) {};
			\node[draw,circle] (E') at (7.5,-3) {};
			\draw (E') node[above]{$y$};
			\node[draw,circle] (F') at (7.5,-6) {};
			\draw[orange,->,>=latex] (B') to[bend right] node[midway, above left]{$\alpha_s$} (A');
			\draw[orange,->,>=latex] (A') to[bend right] (D');
			\draw[orange,->,>=latex] (D') to[bend right] (C');
			\draw[orange,->,>=latex] (C') to[bend right] (B');
			\draw[blue,->,>=latex] (E') to[bend left] node[right]{$\alpha_t$} (F');
			\draw[blue,->,>=latex] (F') to[bend left] (E');
			\draw[purple,->,>=latex] (B') to[bend right=15] node[midway, above left]{$\tau_e$} (E');
			\draw[purple,->,>=latex] (D') to[bend left=15] (F');
		\end{tikzpicture}
		\caption{The $\mathcal{H}$-preaction $\alpha$}
		\label{defalpha}
	\end{figure} 
	    Then, $x \notin \dom(\alpha_t)$ and $y \notin \dom(\alpha_s)$, so applying Construction \hyperlink{quot}{B} to the edge $f$ leads to the preaction $\beta$ defined in Figure \ref{resultpreac}, which is not an extension of $\alpha$. 
	    
	    	\begin{figure}[ht]
		\center
		\begin{tikzpicture}
			\node[draw,circle] (A') at (3,-4.5) {};
			\node[draw,circle] (B') at (4.5,-3) {};
			\draw (B') node[above]{$x=y$};
			\node[draw,circle] (C') at (6,-4.5) {};
			\node[draw,circle] (D') at (4.5,-6) {};
			\draw[orange,->,>=latex] (B') to[bend right] node[midway, above left]{$\alpha_s$} (A');
			\draw[orange,->,>=latex] (A') to[bend right] (D');
			\draw[orange,->,>=latex] (D') to[bend right] (C');
			\draw[orange,->,>=latex] (C') to[bend right] (B');
			\draw[blue,->,>=latex] (D') to[bend left] node[right]{$\alpha_t$} (B');
			\draw[blue,->,>=latex] (B') to[bend left] (D');
			\draw[purple,->,>=latex] (B') to [out=135,in=45,looseness=20] node[midway, above left]{$\tau_e$} (B');
			\draw[purple,->,>=latex] (D') to [out=-135,in=-45,looseness=20] (D');
		\end{tikzpicture}
		\caption{The $\mathcal{H}$-preaction $\beta$}
		\label{resultpreac}
	\end{figure} 

	    Nevertheless, in the case where $\alpha$ is not transitive, applying these constructions to points lying in disjoint sub-$\mathcal{H}$-preactions of $\alpha$ leads to an extension of both sub $\mathcal{H}$-preactions. 
	\end{remark}
	
	Despite the subtlety highlighted in Remark \ref{nonextingeneral}, these constructions allow to see an $\mathcal{H}$-graph as the $\mathcal{H}$-graph of an $\mathcal{H}$-preaction, as evidenced by the following lemma: 
	
		\begin{lemma}\label{arete}
		Let $\mathcal{G}$ be a finite and connected $\mathcal{H}$-graph. Then there exists an $\mathcal{H}$-preaction whose $\mathcal{H}$-graph is $\mathcal{G}$.
	\end{lemma}
	
	\begin{proof}
		The proof works by induction on the number of edges of $\mathcal{G}$: write $\mathcal{G}$ as an increasing (finite) union of $\mathcal{G}$-graphs $\cup_{n=0}^N\mathcal{G}_n$ such that, for every $n \in \llbracket 0, N-1\rrbracket$, one has $\mathcal{G}_{n+1} = \mathcal{G}_n \cup \{E_n\}$ for some edge $E_n \in \mathcal{E}\left(\mathcal{G}\right)$ and $\mathcal{G}_0$ is the subgraph of $\mathcal{G}$ defined by its set of vertices $\mathcal{V}\left(\mathcal{G}\right)$ and no edge. For every $V \in \mathcal{V}\left(\mathcal{G}\right)$, let us denote by $(s_V, N_V)$ the label of the vertex $V$ and let us define the $\mathcal{H}$-preaction $\alpha_0 = \{\alpha_s \vert s \in \mathcal{V}\left(\mathcal{H}\right)\}$ as follows: \[X_0 = \bigsqcup_{V \in \mathcal{V}\left(\mathcal{G}\right)}\mathbb{Z}/N_V\mathbb{Z}\]
		and, for every $s \in \mathcal{V}\left(\mathcal{H}\right)$: \[\dom(\alpha_s) = \bigsqcup_{V \in \mathcal{V}\left(\mathcal{G}\right) \mid s_V = s}\mathbb{Z}/N_V\mathbb{Z}\] (where, by convention, $\mathbb{Z}/\infty\mathbb{Z} = \mathbb{Z}$) with $\alpha_s$ acting on $\mathbb{Z}/N_V\mathbb{Z}$ by translation by $1$, so that the $\mathcal{H}$-graph of $\alpha_0$ is $\mathcal{G}_0$. Using Constructions \hyperlink{ext}{A} or \hyperlink{quot}{B} according on whether the label of $E_n$ is in $T$ or not, we build inductively a sequence of $\mathcal{H}$-preactions $(\alpha_n)_{n \in \llbracket 0, N\rrbracket}$ on countable sets $(X_n)_{n \in \llbracket 0, N \rrbracket}$ such that the $\mathcal{H}$-graph of $\alpha_n$ is $\mathcal{G}_n$ for every $n \in \llbracket 0, N\rrbracket$, which leads to the conclusion. 
	\end{proof}
	
	Notice that, in the previous Lemma, $\alpha_{n+1}$ is not necessarily an extension of $\alpha_n$: because of the use of Construction \hyperlink{quot}{B}, the set $X_n$ need not be included in $X_{n+1}$. 
	
	The following constructions \hyperlink{ext'}{A'} and \hyperlink{quot'}{B'} furnish a tool to extend a non-saturated $\mathcal{H}$-preaction $\alpha$ defined on a countable set $X$: let $x \in X$ such that there exists an edge $e \in \mathcal{E}\left(\mathcal{H}\right)$ satisfying $x \in \dom\left(\alpha_{\src(e)}\right)$ and: \begin{itemize}
		\item either $e \notin \mathcal{E}(T)$ and  $x \notin \dom\left(\tau_e\right)$;
		\item or $e \in \mathcal{E}(T)$ and $x \notin \dom\left(\alpha_{\trg(e)}\right)$.
	\end{itemize}
	Denoting by $N = \left|x\left\langle \alpha_{\src(e)} \right\rangle \right|$, then for any $M \in \mathbb{N} \cup \{\infty\}$ satisfying $\frac{N}{N \wedge k_{e, \src}} = \frac{M}{M \wedge k_{e, \trg}}$, one can extend $\alpha$ by using one of the two following constructions:\\
	\\
	\textbf{Construction \hypertarget{ext'}{A'}}: if $e \notin \mathcal{E}(T)$, apply Construction \hyperlink{ext}{A} to the disjoint union of $\alpha$ and a new $\mathcal{H}$-preaction defined by a single $\left \langle \alpha_{\trg(e)} \right \rangle$-orbit of cardinality $M$ in order to define the image of $x$ under $\tau_e$ (and thus of $x \alpha_{\src(e)}^{k_{e, \src}j}$ for every $j \in \mathbb{Z}$) as an arbitrary element of this new orbit and the $\left \langle \alpha_{\trg(e)} \right \rangle$-orbit of $x \tau_e$ (which is the one of $x \alpha_{\src(e)}^{k_{e, \src}j}\tau_e$ for every $j \in \mathbb{Z}$) of cardinality $M$. \\
	\\
	\textbf{Construction \hypertarget{quot'}{B'}}: if $e \in \mathcal{E}(T)$, apply Construction \hyperlink{quot}{B} to the disjoint union of $\alpha$ and a new $\mathcal{H}$-preaction defined by a single $\left \langle \alpha_{\trg(e)} \right \rangle$-orbit of cardinality $M$ in order to define the $\left \langle \alpha_{\trg(e)} \right \rangle$-orbit of $x$ (and thus of $x \alpha_{\src(e)}^{k_{e, \src}j}$ for every $j \in \mathbb{Z}$) of cardinality $M$. \\
	\\
	The effect of these two constructions is to add an edge with source a non-saturated vertex of the $\mathcal{H}$-graph of $\alpha$ and target a new vertex.
	
	Iterating these constructions leads to the following lemma: 
	
	\begin{lemma}\label{retour}
		Let $({\alpha^{(i)}})_{i \in I}$ ($I$ being countable) be a collection of $\mathcal{H}$-preactions on countable sets $\left(X^{(i)}\right)_{i \in I}$ whose $\mathcal{H}$-graphs are connected. Let $\mathcal{G}^{(i)}$ be the $\mathcal{H}$-graph of $\alpha^{(i)}$ and let $\mathcal{G}$ be a connected $\mathcal{H}$-graph such that: \begin{itemize}
			\item $\bigsqcup_{i \in I}\mathcal{G}^{(i)} \subseteq \mathcal{G}$;
			\item The quotient $\mathcal{G}/\left(\mathcal{G}^{(i)}, i \in I\right)$ (see Section \ref{def}) is a tree.
		\end{itemize}
		Then there exists an $\mathcal{H}$-preaction $\alpha$ whose $\mathcal{H}$-graph is $\mathcal{G}$ such that $\alpha$ extends ${\alpha^{(i)}}$ for every $i \in I$.
	\end{lemma}
	
	\begin{proof}
		Let us write $\mathcal{G}$ as an increasing union of $\mathcal{H}$-graphs: $\mathcal{G} = \bigcup_{n \in \mathbb{N}}\mathcal{G}_n$ with $\mathcal{G}_0 = \mathcal{G}^{(i_0)}$ for some $i_0 \in I$ and such that  for every $n \in \mathbb{N}$ \begin{enumerate}
			\item\label{it: out} either $\mathcal{G}_{n+1} = \mathcal{G}_n \cup \{E\}$ for some edge $E$ with source $V := \src(E) \in \mathcal{G}_n$ and target $W := \trg(E) \notin \mathcal{G}^{(i)}$ for every $i \in I$;
			\item\label{it: in} or $\mathcal{G}_{n+1} = \mathcal{G}_n \cup \{E\} \cup \mathcal{G}^{(i)}$ for some edge $E$ with source $V := \src(E) \in \mathcal{G}_n$ and target $W := \trg(E) \in \mathcal{G}^{(i)}$.
		\end{enumerate}
		One defines a sequence of $\mathcal{H}$-preactions $\beta_n$ on countable sets $X_n$ such that \begin{itemize} 
			\item $\beta_0 = \alpha^{(i_0)}$ and $X_0 = X^{(i_0)}$; 
			\item for every $n \in \mathbb{N}$, the $\mathcal{H}$-graph of $\beta_n$ is $\mathcal{G}_n$;
			\item $X_n \subseteq X_{n+1}$ and $\beta_{n+1}$ extends $\beta_n$;
			\item for every $i \in I$ there exists $n \in \mathbb{N}$ such that $X^{(i)} \subseteq X_n$ and $\beta_n$ extends ${\alpha^{(i)}}$.
			
		\end{itemize}
		To achieve this, we proceed by induction on $n$. Let us assume that $(X_n,\beta_n)$ is defined and let us define $(X_{n+1},\beta_{n+1})$. We denote by $\{\alpha_s, \tau_e \vert s \in \mathcal{V}\left(\mathcal{H}\right), e \in \mathcal{E}\left(\mathcal{H}\right) \setminus \mathcal{E}(T)\}$ the set of partial bijections defining $\beta_n$. We denote by $E$ the edge provided by Item~\ref{it: out} or Item~\ref{it: in}, and by $e_0$ its label. We also introduce $(s_0,t_0) := (\src(e_0),\trg(e_0))$ and $(k,l) := (k_{e_0,\src}, k_{e_0,\trg})$. \\
		\\
		\textbf{1st case}: Assume first that $\mathcal{G}_{n+1} = \mathcal{G}_n \cup \{E\}$ with $\src(E) = V \in \mathcal{G}_n$ and $\trg(E) = W \notin \bigsqcup_{i \in I}\mathcal{G}^{(i)}$. Since we assumed that the quotient $\mathcal{G}/(\mathcal{G}^{(i)}; i \in I)$ is a tree, one has $W \notin \mathcal{G}_n$
	 One has $V = x\left\langle \alpha_{s_0} \right\rangle$ for some $x \in X_n$. Let $(s_0,N)$, $(t_0,M)$ be the labels of $V$ and $W$, respectively. By the definition of an $\mathcal{H}$-graph (see Proposition-Definition~\ref{ggraph}), we have \[\frac{N}{N \wedge k} = \frac{M}{M \wedge l}.\] 
		Assume first that $e_0 \in \mathcal{E}\left(T\right)$. As the existence of the edge $E$ in $\mathcal{G}_{n+1}$ implies that the vertex $V$ is not saturated relatively to $e_0$ in $\mathcal{G}_n$, there exists $y \in x\left\langle \alpha_{s_0} \right\rangle$ such that $y \notin \dom\left(\alpha_{t_0}\right)$. 
		We apply Construction \hyperlink{quot'}{B'} in order to define the $\left\langle \alpha_{t_0}\right\rangle$-orbit of $y$ by imposing its cardinality $M$. We obtain an $\mathcal{H}$-preaction $\beta_{n+1}$ defined on a countable set $X_{n+1}$ containing $X_n$ that extends $\beta_n$ and whose $\mathcal{H}$-graph is $\mathcal{G}_{n+1}$. 
		
		Now let us assume that $e_0 \notin \mathcal{E}\left(T\right)$. As the existence of the edge $E$ in $\mathcal{G}_{n+1}$ implies that the vertex $V$ is not saturated relatively to $e_0$ in $\mathcal{G}_n$, there exists $y \in x\left\langle \alpha_{s_0} \right\rangle$ such that $y \notin \dom\left(\tau_{e_0}\right)$. We use Construction \hyperlink{ext'}{A'} in order to define a new element $z$ which is the image of $y$ under $\tau_{e_0}$ and the $\left\langle \alpha_{t_0} \right\rangle$-orbit of $z$ of cardinality $M$. We get an $\mathcal{H}$-preaction $\beta_{n+1}$ that satisfies the required conditions. \\
		\\
		\textbf{2nd case}: $\mathcal{G}_{n+1} = \mathcal{G}_n \cup \{E\} \cup \mathcal{G}^{(i)}$ for some $i \in I$ with $\src(E) \in \mathcal{G}_n$ and $\trg(E) \in \mathcal{G}^{(i)}$. The vertex $V$ derives from an $\left\langle\alpha_{s_0}\right\rangle$-orbit $V = x \left\langle \alpha_{s_0} \right\rangle$ for some $x \in X_n$ and the vertex $W$ derives from an $\left\langle{\alpha^{(i)}_{t_0}}\right\rangle$-orbit $W = x' \left\langle {{\alpha}^{(i)}_{t_0}} \right\rangle \subseteq X^{(i)}$ for some $x' \in X^{(i)}$. Let $(k,l) = (k_{e_0, \src}, k_{e_0, \trg})$.  
		
		Assume first that $e_0 \notin \mathcal{E}\left(T\right)$. As the existence of $E$ in $\mathcal{G}_{n+1}$ implies that $V$ is not saturated relatively to $e_0$ in $\mathcal{G}_n$ and that $W$ is not saturated relatively to $\overline{e_0}$ in $\mathcal{G}^{(i)}$, there exists $y \in x \langle \alpha_{s_0} \rangle$ and $y' \in x'\left\langle \alpha^{(i)}_{t_0} \right\rangle$ such that $\tau_{e_0}$ is not defined on $y$ and ${\tau^{(i)}_{e_0}}^{-1}$ is not defined on $y'$. Thus we can apply Construction \hyperlink{ext}{A} to the $\mathcal{H}$-preaction $\beta_n \sqcup \alpha_i$ and to the elements $y, y'$ and the edge $e_0$ (which amounts to sending the $\left\langle \alpha_{s_0}^k \right\rangle$-orbit of $y$ to the $\left\langle {\alpha_{t_0}^{(i)}}^l \right\rangle$-orbit of $y'$ \textit{via} $t_{e_0}$). The resulting $\mathcal{H}$-preaction $\beta_{n+1}$ is defined on a countable set $X_{n+1}$ in which both $X_n$ and $X^{(i)}$ inject and extends both $\beta_n$ and $\alpha^{(i)}$. Its $\mathcal{H}$-graph is $\mathcal{G}_{n+1}$.
		
		Now let us assume that $e_0 \in \mathcal{E}\left(T\right)$. As $V,W$ are not saturated relatively to $e_0, \overline{e_0}$, respectively, there exists $y \in \dom(\alpha_{s_0}) \setminus \dom\left(\alpha_{t_0}\right)$ and $y' \in \dom\left({{\alpha_{t_0}^{(i)}}}\right) \setminus \dom\left({\alpha_{s_0}^{(i)}}\right)$. So we can apply Construction \hyperlink{ext}{B} to the $\mathcal{H}$-preaction $\beta_n  \sqcup \alpha^{(i)}$ and to the elements $y, y'$ and the edge $e_0$ (which amounts to identifying the $\left\langle \alpha_{s_0}^k\right\rangle$-orbit of $y$ with the $\left\langle {\alpha_{t_0}^{(i)}}^{l}\right\rangle$-orbit of $y'$) to get an $\mathcal{H}$-preaction $\beta_{n+1}$ that extends both $\beta_n$ and $\alpha^{(i)}$ and whose $\mathcal{H}$-graph is $\mathcal{G}_{n+1}$. \end{proof}

The following lemma is crucial, because it relates approximations in $\Sub(\Gamma)$ to approximations on the level of $\mathcal{H}$-graphs. It states that, under some assumptions, an $\mathcal{H}$-graph that contains the $\mathcal{H}$-graph of an $\mathcal{H}$-preaction $\alpha$ is the $\mathcal{H}$-graph of a saturated $\mathcal{H}$-preaction that extends $\alpha$.	

	\begin{lemma}\label{completion1}
		Let $\mathcal{G}$ be a finite, connected and non-saturated $\mathcal{H}$-graph. There exists a saturated $\mathcal{H}$-graph $\tilde{\mathcal{G}}$ such that \begin{itemize}
			\item $\mathcal{G}$ is a sub-$\mathcal{H}$-graph of $\tilde{\mathcal{G}}$;
			\item the quotient $\tilde{\mathcal{G}} / \mathcal{G}$ is an infinite tree.
		\end{itemize}
		Moreover, one can chose $\tilde{\mathcal{G}}$ such that, for every edge $E \in \mathcal{E}\left(\tilde{\mathcal{G}}\right) \setminus \mathcal{E}\left(\mathcal{G}\right)$ labeled $e$ whose target is labeled $(\trg(e), N)$, if the half graph of $E$ is in $\tilde{\mathcal{G}} \setminus \mathcal{G}$, then the integer $N$ is divisible by $k_{e, \trg}$.
	\end{lemma}
	
	\begin{proof}
		Let $V$ be a non-saturated vertex of $\mathcal{G}$ with label $(s,N)$. By hypothesis, there exists an edge $e$ of $\mathcal{H}$ such that $\src(e) = s$ and such that $V$ has strictly less than $N \wedge k_{e,\src}$ incident edges labeled $e$. Let us denote by $d$ the number of these edges and let us enlarge $\mathcal{G}$ by adding $N \wedge k_{e, \src} - d$ new incident edges labeled $e$, with $N \wedge k_{e, \src} - d$ new targets labeled $(\trg(e), M)$ where $M = \frac{N|k_{e, \trg}|}{N \wedge k_{e, \src}}$ so that the Transfer Equation is satisfied and $k_{e, \trg}$ divides $M$. We apply this construction to all the edges in $\mathcal{H}$ relatively to which $V$ is not saturated and to all the vertices of $\mathcal{G}$ which are not saturated and we obtain a new $\mathcal{H}$-graph $\mathcal{G}'$ which satisfies the following properties: 
		\begin{itemize}
			\item $\mathcal{G} \subseteq \mathcal{G}'$;
			\item all vertices of $\mathcal{G}$ are saturated in $\mathcal{G}'$; 
			\item the subgraph induced by the set of vertices of $\mathcal{G}' \setminus \mathcal{G}$ is a forest. 
		\end{itemize}
		By induction, we construct an increasing sequence of $\mathcal{H}$-graphs $\mathcal{G}_n$ such that \begin{itemize}
			\item all vertices of $\mathcal{G}_n$ are saturated in $\mathcal{G}_{n+1}$;
			\item the subgraph induced by the set of vertices of $\mathcal{G}_{n+1} \setminus \mathcal{G}$ is a forest.
		\end{itemize}
		The $\mathcal{H}$-graph defined by $\cup_{n \in \mathbb{N}}\mathcal{G}_n$ satisfies the required conditions. 
	\end{proof}
	
	The following lemma provides an example of construction of an $\mathcal{H}$-graph.
	
	\begin{lemma}\label{boucle}
		Assume that $\mathcal{H}$ is neither a loop one of whose labels is equal to $\pm 1$, nor a segment with both labels equal to $\pm 2$. Let $e \in \mathcal{E}\left(\mathcal{H}\right)$ and $N \in \mathbb{N} \cup \{\infty\}$ such that $k_{e, \src}$ divides $N$ if $\trg(e) \neq \src(e)$. Then there exists a finite $\mathcal{H}$-graph \begin{itemize}
			\item that is non-simply connected;
			\item that contains a non-saturated vertex $V$ labeled $(\src(e), N)$;
			\item that contains another vertex that is non-saturated.
		\end{itemize}
	\end{lemma}
	
	\begin{proof}
		Let $e \in \mathcal{E}(\mathcal{H})$ and $N \in \mathbb{N} \cup \{\infty\}$ satisfying the assumptions of Lemma \ref{boucle}. Let us denote $(s, t) = \left(\src(e), \trg(e)\right)$ and $(k,l) = \left(k_{e, \src}, k_{e, \trg}\right)$. \\
		\\
		Assume first that $e$ is not a loop. 
		\begin{itemize}
			\item
			If $|k| > 2$ or $|l| > 2$ (assume for instance $|l| > 2$), one construct an edge labeled $e$ with source a vertex $V$ labeled $(s,N)$ and target a vertex $W$ labeled $\left(t, \frac{N|l|}{N \wedge k}\right)$. Then we build two new vertices $X, Y$ labeled $(t,\frac{N|k|}{N \wedge k})$ and two edges labeled $\overline{e}$ with source $W$ and targets $X$ and $Y$. Finally we construct a new vertex $Z$ labeled $(s, \frac{N|l|}{N \wedge k})$ and edges labeled $e$ with sources $X$ and $Y$ and target $Z$. As $|k| \geq 2$ divides $N$, the vertex $V$ is not saturated relatively to $e$, and as $|l|\geq 3$, the vertex $Z$ is not saturated relatively to $\overline{e}$ (\textit{cf.} Figure \ref{figurelosange}). \\
			\begin{figure}[ht]
				\center
				\begin{tikzpicture}
					\node[draw,circle,fill=gray!50] (a) at (0,0) {$N$};
					\node[draw,circle,fill=white] (b) at (3,0) {$\frac{N|l|}{N \wedge k}$};
					\node[draw,circle,fill=gray!50] (c) at (5,1) {$\frac{N|k|}{N \wedge k}$};
					\node[draw,circle,fill=gray!50] (d) at (5,-1) {$\frac{N|k|}{N \wedge k}$};
					\node[draw,circle,fill=white] (e) at (7,0) {$\frac{N|l|}{N \wedge k}$};
					\draw[>=latex, directed] (a) -- (b);
					\draw[>=latex, directed] (c) -- (b);
					\draw[>=latex, directed] (d) -- (b);
					\draw[>=latex, directed] (c) -- (e);
					\draw[>=latex, directed] (d) -- (e);
				\end{tikzpicture}
				\caption{Case where $e$ is not a loop and $\max\left(|k|, |l|\right) > 2$.}
				\label{figurelosange}
			\end{figure}
			
			The case where $|l|=2$ and $|k| \geq 3$ is similar (one can take four vertices $W = V, X, Y, Z$ labeled $(s,N), \left(t, \frac{N|k|}{N \wedge k}\right), \left(t, \frac{N|k|}{N \wedge k}\right)$ and $(s,N)$, respectively). 
	
			\item
			If $|k|=|l|=2$, then as $\Gamma$ is non-amenable, $\mathcal{H}$ contains another edge $f$, and one can assume that $\src(f) = s$. Let us denote $(k_{f, \src}, k_{e, \trg}) = (m,n)$ and $u = \trg(f)$.  
			\begin{itemize}
				\item If $f$ is not a loop, then $|n| \geq 2$. Hence we can construct an edge labeled $e$ with source a vertex $V$ labeled $(s, N)$ and target a vertex $W$ labeled $\left(t, \frac{2N}{N\wedge 2}\right)$, an edge labeled $f$ with source $V$ and target a vertex $X$ labeled $\left(u, \frac{N|n|}{N \wedge m} \right)$, a new vertex $Y$ labeled $(s, N)$, an edge labeled $e$ with source $Y$ and target $W$ and an edge labeled $f$ with source $Y$ and target $X$. 
				
				\begin{figure}[ht]
					\center
					\begin{tikzpicture}
						\node[draw,circle,fill=gray!50] (a) at (-5,1) {};
						\node[draw,circle,fill=white] (b) at (-3,0) {};
						\node[draw,circle,fill=black] (c) at (-1,1) {};
						\draw[gray, >=latex, directed] (b) -- (a) node[midway, below left]{$f$} node[very near start, below left]{$m$} node[very near end, below left]{$n$};
						\draw[>=latex, directed] (b) -- (c) node[midway, below right]{$e$} node[very near start, below right]{$k$} node[very near end, below right]{$l$};
						\node[draw,circle,fill=gray!50] (A) at (1,0) {$\frac{N|n|}{N \wedge m}$};
						\node[draw,circle,fill=white] (B) at (6,0) {$N$};
						\node[draw,circle,fill=white] (C) at (1,2) {$N$};
						\node[draw,circle,fill=black] (D) at (6,2) {\color{white} $\frac{2N}{N \wedge 2}$};
						\draw[gray, >=latex, directed] (B) -- (A);
						\draw[gray, >=latex, directed] (C) -- (A);
						\draw[>=latex, directed] (B) -- (D);
						\draw[>=latex, directed] (C) -- (D);
					\end{tikzpicture}
					\caption{Case where $e$ is not a loop, and $|k|=|l|=2$, and $f$ is not a loop.}
					\label{figuredeuxsegments}
				\end{figure}
				
				Moreover, as $2$ divides $N$, neither $V$ nor $Y$ is saturated relatively to $\overline{e}$ (\textit{cf.} Figure \ref{figuredeuxsegments}).
				\item If $f$ is a loop, we build an edge labeled $f$ between a vertex $V$ labeled $(s, N)$ and a vertex $U$ labeled $(s, M)$ with $M := \frac{N|n|}{N \wedge m}$. Then we build an edge labeled $e$ with source $U$ and target a new vertex $X$ labeled $\left(t, \frac{2M}{M \wedge 2}\right)$ and an edge labeled $e$ with source $V$ and target a vertex $W$ labeled $\left(t,\frac{2N}{N \wedge 2} \right)$. Finally, we create two vertices $Y$ and $Z$ labeled by $(s, N)$ and $(s, M)$ respectively, an edge labeled $f$ with source $Y$ and target $Z$, an edge labeled $e$ with source $Z$ and target $X$ and an edge labeled $e$ with source $Y$ and target $W$ (\textit{cf.} Figure \ref{figureuneboucleunsegment}). As $|k| \geq 2$ divides $N$, neither $V$ nor $Y$ is saturated relatively to $e$.
				\begin{figure}[ht]
					\center
					
					\begin{tikzpicture}
						\node[draw,circle,fill=gray!50] (a) at (-5,0) {};
						\node[draw,circle,fill=white] (b) at (-3,0) {};
						\draw[>=latex, directed] (a) -- (b) node[midway, below]{$e$} node[very near start, below]{$k$} node[very near end, below]{$l$};
						\draw[gray, >=latex, directed] (a) to [out=135,in=45,looseness=20] node[above]{$f$} node[very near start, left]{$m$} node[very near end, right]{$n$} (a);
						\node[draw,circle,fill=gray!50] (A) at (0,-1.5) {$N$};
						\node[draw,circle,fill=white] (B) at (3,-3) {$\frac{2N}{N \wedge 2}$};
						\node[draw,circle,fill=gray!50] (C) at (0,1.5) {$M$};
						\node[draw,circle,fill=white] (D) at (3,3) {$\frac{2M}{M \wedge 2}$};
						\node[draw,circle,fill=gray!50] (E) at (6,-1.5) {$N$};
						\node[draw,circle,fill=gray!50] (F) at (6,1.5) {$M$};
						\draw[gray, >=latex, directed] (A) to (C);
						\draw[gray, >=latex, directed] (E) to (F);
						\draw[>=latex, directed] (C) to (D);
						\draw[>=latex, directed] (F) to (D);
						\draw[>=latex, directed] (A) to (B);
						\draw[>=latex, directed] (E) to (B);
					\end{tikzpicture}
					\caption{Case where $e$ is not a loop, and $|k|=|l|=2$, and $f$ is a loop.}
					\label{figureuneboucleunsegment}
				\end{figure}
				
			\end{itemize}
			
		\end{itemize}
		Now assume that $e$ is a loop. 
		
		\begin{itemize}
			\item If $|k|\geq 2$ and $|l| \geq 2$, we start with a vertex $V$ labeled $(s, N)$ and create an edge labeled $e$ with source $V$ and target a new vertex $W$ labeled $(s,M)$ with $M = \frac{N|l|}{N \wedge k}$. Then we build an edge labeled $e$ with source $W$ and target a new vertex $Y$ labeled $\left(s, \frac{M|l|}{M \wedge k}\right)$ and an edge labeled $\overline{e}$ with source $W$ and target a new vertex $X$ labeled $\left(s, \frac{N|k|}{N \wedge k} \right)$. Finally we define a new vertex $Z$ labeled $(s,M)$, an edge labeled $e$ with source $X$ and target $Z$ and an edge labeled $\overline{e}$ with source $Y$ and target $Z$ (\textit{cf.} Figure \ref{figurebsmn}).
			
			\begin{figure}[ht]
				\center
				\begin{tikzpicture}
					\node[draw,circle,fill=gray!50] (a) at (-3,0) {};
					\draw[>=latex, directed] (a) to [out=135,in=45,looseness=20] node[above]{$e$} node[very near start, below left]{$k$} node[very near end, below right]{$l$} (a);
					\node[draw,circle,fill=gray!50] (A) at (0,2) {$N$};
					\node[draw,circle,fill=gray!50] (B) at (4,2) {$M$};
					\node[draw,circle,fill=gray!50] (C) at (8,2) {$\frac{N|k|}{N \wedge k}$};
					\node[draw,circle,fill=gray!50] (D) at (4,0) {$\frac{M|l|}{M \wedge k}$};
					\node[draw,circle,fill=gray!50] (E) at (8,0) {$M$};
					\draw[>=latex, directed] (A) to (B);
					\draw[>=latex, directed] (C) to (B);
					\draw[>=latex, directed] (B) to (D);
					\draw[>=latex, directed] (E) to (D);
					\draw[>=latex, directed] (C) to (E);
				\end{tikzpicture}
				\caption{Case where $e$ is a loop, and $\min\left(|k|,|l|\right)\geq 2$.}
				\label{figurebsmn}
			\end{figure}
			
			The vertex $Y$ is not saturated relatively to $e$, and, as $|k|\geq 2$ divides $N$, the vertex $V$ is not saturated relatively to $e$.
			\item If $|k|=1$ or $|l|=1$ (for instance $|k|=1$), the graph $\mathcal{H}$ contains another edge $f$ with source $s$. Denote $(k_{f, \src}, k_{f, \trg}) = (m,n)$. By the previous case, one can assume that $f$ is a loop and that $|m|=1$. One create a vertex $V$ labeled $(s,N)$, an edge labeled $e$ connecting $V$ to a new vertex $W$ labeled $(s, N|l|)$, an edge labeled $f$ connecting $V$ to a new vertex $X$ labeled $(s, N|n|)$, a new vertex $Y$ labeled $(s, N|ln|)$ and edges labeled $f$ and $e$ connecting $W$ to $Y$ and $X$ to $Y$ respectively (\textit{cf.} Figure \ref{figuredeuxboucles}). Then $V$ is not saturated relatively to $\overline{e}$, and $Y$ is not saturated relatively to $e$.
			\begin{figure}[ht]
				\center
				\begin{tikzpicture}
					\node[draw,circle,fill=gray!50] (a) at (-5,1) {};
					\draw [gray, >=latex, directed] (a) to [out=300,in=30,looseness=20] node[right]{$f$} node[very near start, below]{$m$} node[very near end, above]{$n$} (a);
					\draw [>=latex, directed] (a) to [out=240,in=150,looseness=20] node[left]{$e$} node[very near start, below]{$k$} node[very near end, above]{$l$} (a);
					\node[draw,circle,fill=gray!50] (A) at (0,2) {$N$};
					\node[draw,circle,fill=gray!50] (B) at (4,2) {$N|l|$};
					\node[draw,circle,fill=gray!50] (C) at (0,-1) {$N|n|$};
					\node[draw,circle,fill=gray!50] (D) at (4,-1) {$N|ln|$};
					\draw [>=latex, directed] (A) to (B);
					\draw [gray, >=latex, directed] (A) to (C);
					\draw [>=latex, directed] (C) to (D);
					\draw [gray, >=latex, directed] (B) to (D);
				\end{tikzpicture}
				\caption{Case where $e$ is a loop, and $\min\left(|k|, |l|\right) = 1$.}
				\label{figuredeuxboucles}
			\end{figure}
			
		\end{itemize}
		
	\end{proof}
	
	\section{Phenotype of a GBS group}
	
	In \cite{solitar1}, the \textbf{phenotype} of a subgroup $\Lambda$ of $\BS(m,n)$ is defined as follows: 
	\begin{itemize}
		\item if $\Lambda \cap \langle b \rangle = \{1\}$, then $\bm{\Ph}_{m,n}\left(\Lambda\right) = \infty$;
		\item otherwise, denoting by $\Lambda \cap \langle b \rangle =: \langle b^N \rangle$ (for $N \in \mathbb{N}^*$), then $\bm{\Ph}_{m,n}(\Lambda) = \Ph_{m,n}(N)$ where $\Ph_{m,n}(N) := \prod_{p \in \mathcal{P}, |m|_p = |n|_p \text{ \ and \ } |N|_p > |n|_p}p^{|N|_p}$.
	\end{itemize}
	
	The authors of \cite{solitar1} proved that this quantity is invariant under conjugation and that the set $\mathcal{Q}_{m,n}$ of integers satisfying ${\Ph}_{m,n}^{-1}(N) \neq \emptyset$ is infinite. Moreover, they proved the following result:
	
	\begin{theorem}
		Let $(m,n) \in \left(\mathbb{Z} \setminus \{0\}\right)^2$ such that $|m| \neq 1$ and $|n| \neq 1$. Then, for any $N \in \mathcal{Q}_{m,n} \cup \{\infty\}$: \begin{itemize}
			\item the action of $\BS(m,n)$ by conjugation on $\mathcal{K}(\BS(m,n)) \cap \bm{\Ph}_{m,n}^{-1}(N)$ is topologically transitive. In particular, $\mathcal{K}\left(\BS(m,n)\right) \cap \bm{\Ph}_{m,n}^{-1}(N)$ contains a dense $G_{\delta}$ set of elements whose orbits are dense;
			\item if $N \in \mathbb{N}$, the piece $\bm{\Ph}_{m,n}^{-1}(N)$ is open;
			\item it is also closed if and only if $|m|=|n|$;
			\item the piece $\bm{\Ph}_{m,n}^{-1}(\infty)$ is closed and non-empty. 
		\end{itemize}  
	\end{theorem}
	In \cite{solitar2}, the same authors even proved that the action of $\BS(m,n)$ by conjugation on $\mathcal{K}(\BS(m,n)) \cap \bm{\Ph}_{m,n}^{-1}(N)$ is highly topologically transitive if $\bm{\Ph}_{m,n}^{-1}(N) \neq \emptyset$.
	The goal of this section is to define a new arithmetical invariant for the action by conjugation of any GBS group on its space of subgroups. We need some arithmetical tools. 
	
	\begin{definition}\label{defphenotype}
		Let us define the following subset of prime numbers: \[\mathcal{C}_{\mathcal{H}} = \left\{p \in \mathcal{P} \mid \sum_{i=1}^{r}|k_i|_p = \sum_{i=1}^r |l_i|_p
		\text{ for every cycle labeled } (k_1,l_1),...,(k_r,l_r) \right\}\]	(where $\sum_{i=1}^{0}|l_i|_p=0$). Let $v$ be a vertex of $\mathcal{H}$ and, for every $N \in \mathbb{N}$, let \begin{align*}
		\mathcal{P}_{\mathcal{H}, v}(N) := &\bigg\{p \in \mathcal{C}_{\mathcal{H}} \mid  
		|N|_p > \sum_{i=1}^{r}|k_i|_p - \sum_{i=1}^{r-1}|l_i|_p \text{ for every edge path based at $v$} \\ &\text{ labeled } (k_1,l_1),...,(k_{r},l_{r}) \bigg\}.\end{align*} One defines
		\[ \Ph_{\mathcal{H}, v}(N) := \prod_{p \in \mathcal{P}_{\mathcal{H},v}(N)}p^{|N|_p}. \]
		If $N = \infty$, one defines \[ \Ph_{\mathcal{H},v}(\infty) = \infty. \]
		We call this quantity the \textbf{$\mathcal{H}$-phenotype of the integer $N$ relatively to the vertex $v$}.
	\end{definition}
	
	\begin{remark}
		If $\mathcal{H}$ consists of a loop (based at $v$) with source labeled $n$ and target labeled $m$, then one has $\Ph_{\mathcal{H}, v} = \Ph_{m,n}$ as defined in \cite{solitar1}. 
	\end{remark}
	
	\begin{remark}\label{simplification}Notice that, in Definition \ref{defphenotype}, one can replace ``for every edge path based at $v$'' by ``for every \textit{simple} edge path based at $v$ that either is not a cycle or is a loop'' and ``for every cycle'' by ``for every \textit{simple} cycle'' (which need not contain $v$!)
	\end{remark}
	
	\begin{remark}
		The set $\Ph_{\mathcal{H},v}^{-1}(M)$ is non-empty if and only if for all prime number $p$ dividing $M$, one has $p \in \mathcal{P}_{\mathcal{H},v}(M)$. In that case, one has $M = \Ph_{\mathcal{H},v}(M)$. We denote \[\mathcal{Q}_{\mathcal{H}, v} := \Ph_{\mathcal{H},v}\left(\mathbb{N}\right).\]
	\end{remark}
	
	\subsection{\texorpdfstring{Some arithmetical properties of the $\mathcal{H}$-phenotype}{Some arithmetical properties of the H-phenotype}}
	
	The goal of this subsection is to establish some arithmetical properties of the $\mathcal{H}$-phenotype of an integer in order to define the $\mathcal{H}$-phenotype of a transitive $\mathcal{H}$-preaction (or, equivalently, of a connected $\mathcal{H}$-graph). In particular, we show that the $\mathcal{H}$-phenotype relatively to a given vertex $s$ is constant on the set of vertices of type $s$ of a connected $\mathcal{H}$-graph (\textit{cf.} Proposition \ref{stab1}). So considering saturated $\mathcal{H}$-preactions allows us to define the $\mathcal{H}$-phenotype of a subgroup. 
	
	\begin{proposition}\label{pgcd}
		Let $k_1,l_1,...,k_r,l_r$ be integers.
		Let $N_1,...,N_{r+1} \in \mathbb{N}$ satisfying \[\frac{N_i}{N_i \wedge k_i} = \frac{N_{i+1}}{N_{i+1} \wedge l_i} \text{ \ for every $i \in \llbracket 1,r \rrbracket$}.\]
		Then, for every prime number $p$ satisfying $|N_1|_p > \max_{i \in \llbracket 1,r \rrbracket}\left(\sum_{j=1}^{i}|k_j|_p - \sum_{j=1}^{i-1}|l_j|_p\right)$ one has \[|N_{r+1}|_p = |N_1|_p - \sum_{j=1}^{r}|k_j|_p + \sum_{j=1}^{r}|l_j|_p.\]
	\end{proposition}
	
	\begin{proof}
		
		By induction on $i \in \llbracket 2,r+1 \rrbracket$, let us show property $P_i$: ``$|N_i|_p = |N_1|_p - \sum_{j=1}^{i-1}|k_j|_p + \sum_{j=1}^{i-1}|l_j|_p$.'' 
		\paragraph{Base case} Given the fact that $|N_1|_p > |k_1|_p$, one has \begin{align*}|N_1|_p - |k_1|_p &= |N_2|_p - \min\left(|N_2|_p,|l_1|_p\right) \\
			&> 0
		\end{align*}
		hence necessarily $\min\left(|N_2|_p,|l_1|_p\right) = |l_1|_p$
		so \[|N_2|_p = |N_1|_p - |k_1|_p + |l_1|_p.\]
		Thus $P_2$ is true. 
		\paragraph{Induction step} Let us assume that $P_i$ is true, one has \begin{align*}
			|N_i|_p &= |N_1|_p - \sum_{j=1}^{i-1}|k_j|_p + \sum_{j=1}^{i-1}|l_j|_p \\
			&> |k_i|_p \text{ \ by the assumption made on $p$.}
		\end{align*} 
		We deduce that: \begin{align*}
			|N_{i+1}|_p - \min\left(|l_i|_p,|N_{i+1}|_p\right) &= |N_i|_p - \min\left(|N_i|_p,|k_i|_p\right) \\
			&= |N_i|_p - |k_i|_p \\
			&> 0
		\end{align*} hence necessarily $ \min\left(|l_i|_p,|N_{i+1}|_p\right) = |l_i|_p$ and \[|N_{i+1}|_p = |N_i|_p - |k_i|_p + |l_i|_p.\]
		The induction is proved.  
		
	\end{proof}
	
	\begin{proposition}\label{stab1}
		Let $\mathcal{G}$ be a connected $\mathcal{H}$-graph. Let $V$ and $W$ be two vertices of $\mathcal{G}$ with labels $(s, N)$ and $(s, M)$ respectively. One has $\Ph_{\mathcal{H},s}(N) = \Ph_{\mathcal{H}, s}(M)$. 
	\end{proposition}
	
	\begin{proof}
		Let $E_1,...,E_r \in \mathcal{E}\left(\mathcal{G}\right)$ be an edge path labeled $e_1,...,e_r \in \mathcal{E}\left(\mathcal{H}\right)$ such that $\src(E_1) = V$ and $\trg(E_r) = W$. Let us denote $V_i = \src(E_i)$ for every $i$ and $V_{r+1} = W$.   \\
		First notice that the edges $e_1,...,e_r$ form a cycle based at $s$ in $\mathcal{H}$. For every $i \in \llbracket 1,r \rrbracket$, we denote $s_i = \src(e_i)$ and $(k_i, l_i) = (k_{e_i, \src}, k_{e_i, \trg})$. \\
		For every $i \in \llbracket 1,r \rrbracket$, let $(s_i,N_i)$ be the label of the vertex $v_i$ in $\mathcal{H}$. One has: \[\frac{N_i}{N_i \wedge k_i} = \frac{N_{i+1}}{N_{i+1} \wedge l_i}.\]
		Let $p$ be a prime number such that $|N|_p > \max_{i \in \llbracket 1,r \rrbracket}\left(\sum_{j=1}^{i}|k_j|_p - \sum_{j=1}^{i-1}|l_j|_p\right)$. By Proposition \ref{pgcd}, one has \[|M|_p = |N|_p + \sum_{i=1}^r|k_i|_p - \sum_{i=1}^r|l_i|_p.\]
		Thus, for every $p \in \mathcal{P}_{\mathcal{H},s}(N)$: \[|M|_p = |N|_p.\] 
		We deduce that \[\Ph_{\mathcal{H},s}(N) \text{ \ divides } \Ph_{\mathcal{H},s}(M)\]
		and by symmetry (if the edge path $E_1,...,E_r$ connects $V$ to $W$, then the edge path $\overline{E_r},...,\overline{E_1}$ connects $W$ to $V$ in $\mathcal{G}$), we finally get \[\Ph_{\mathcal{H},s}(N) = \Ph_{\mathcal{H},s}(M).\]
	\end{proof}
	
	If $\alpha$ is a transitive $\mathcal{H}$-preaction, we proved that, for every $s \in \mathcal{V}\left(\mathcal{H}\right)$ and every vertices $V$, $W$ of $\mathcal{G}$ with labels $(s,N)$, $(s,M)$, respectively, one has \[\Ph_{\mathcal{H},s}(N) = \Ph_{\mathcal{H},s}(M).\]
	\begin{definition}
	One defines the \textbf{$\mathcal{H}$-phenotype of $\alpha$ relatively to the vertex $s$} as this quantity, and we denote it by $\bm{\Ph}_{\mathcal{H}, s}(\alpha)$. 
	\end{definition}
	In particular, given a subgroup $\Lambda$ of $\Gamma$, applying this observation to the (saturated) $\mathcal{H}$-preaction associated to $\Lambda$, we get, for every $\gamma \in \Gamma$: \[\Ph_{\mathcal{H},s}\left(\left[ \left\langle \alpha_s \right\rangle: \Lambda \cap \left\langle \alpha_s \right\rangle \right]\right) = \Ph_{\mathcal{H}, s}\left(\left[ \gamma \left\langle \alpha_s \right\rangle \gamma^{-1}: \Lambda \cap \gamma \left\langle \alpha_s \right\rangle \gamma^{-1}\right]\right).\]
	
	\begin{definition}
		We denote by \[\bm{\Ph}_{\mathcal{H},s}(\Lambda) := \Ph_{\mathcal{H},s}\left(\left[\left\langle a_s \right\rangle : \Lambda \cap \left\langle a_s \right\rangle \right]\right)\] this quantity and we call it \textbf{$\mathcal{H}$-phenotype of the subgroup $\Lambda$ of $\Gamma$ relatively to $s$}. The partition \[\Sub(\Gamma) = \bigsqcup_{N \in \mathcal{Q}_{\mathcal{H}, s} \cup \left\{ \infty \right\}} \bm{\Ph}_{\mathcal{H},s}^{-1}(N) \] is called the \textbf{phenotypical decomposition of $\Gamma$ relatively to the vertex $s$}. The action by conjugation preserves $\bm{\Ph}_{\mathcal{H},s}^{-1}(N)$ for every $N \in \mathcal{Q}_{\mathcal{H}, s} \cup \{\infty\}$.
	\end{definition}
	
	Though the $\mathcal{H}$-phenotype a priori depends on the choice of a given vertex, the property of sharing the same $\mathcal{H}$-phenotype does not depend on the choice of the vertex: 
	
	\begin{proposition}\label{decphen}
		Let $\mathcal{G}_1$ and $\mathcal{G}_2$ be connected $\mathcal{H}$-graphs and $V$, $W$ be vertices of $\mathcal{G}_1$, $\mathcal{G}_2$ with labels $(s,M)$, $(s,N)$ respectively such that \[\Ph_{\mathcal{H}, s}(M) = \Ph_{\mathcal{H},s}(N).\]
		For every vertices $V'$, $W'$ of $\mathcal{G}_1$, $\mathcal{G}_2$ respectively with labels $(s', M')$, $(s', N')$ one has \[\Ph_{\mathcal{H}, s'}\left(M'\right) = \Ph_{\mathcal{H}, s'}\left(N'\right).\]
	\end{proposition}
	
	\begin{proof}
		First, notice that, for every integer $n$, \[\mathcal{P}_{\mathcal{H},s}(n) = \left\{p \in \mathcal{P} \mid |\Ph_{\mathcal{H},s}(n)|_p > 0\right\}.\]
		In particular, the equality $\Ph_{\mathcal{H}, s}(M) = \Ph_{\mathcal{H},s}(N)$ implies that \[\mathcal{P}_{\mathcal{H}, s}(M) = \mathcal{P}_{\mathcal{H}, s}(N). \]
		Let $E_1,...,E_r \in \mathcal{E}\left(\mathcal{G}_1\right)$ (\textit{resp.} $F_1,..., F_u \in \mathcal{E}\left(\mathcal{G}_2\right)$) be an edge path labeled $e_1,...,e_r \in \mathcal{E}\left(\mathcal{H}\right)$ (\textit{resp.} $f_1,...,f_u \in \mathcal{E}\left(\mathcal{H}\right)$) which connects $V$ to $V'$ (\textit{resp.} $W'$ to $W$). We denote $\src(E_i) = V_i$ for every $i \in \llbracket 1, r \rrbracket$ (\textit{resp.} $\src(F_i) = W_i$ for every $i \in \llbracket 1, u \rrbracket$) and $V_{r+1} = V'$ (\textit{resp.} $W_{u+1} = W$). One denotes $(s_i,M_i)$ (\textit{resp.} $(t_i,N_i)$) the label of $V_i$ (\textit{resp.} $W_i$) for every $i \in \llbracket 1,r+1 \rrbracket$ (\textit{resp.} $i \in \llbracket 1,u+1 \rrbracket$). In particular, $s_1 = t_{u+1} = s$ and $s_{r+1} = t_1 = s'$. We also denote $(k_i,l_i) = (k_{e_i, \src}, k_{e_i, \trg})$ and $(m_i, n_i) = (k_{f_i, \src}, k_{f_i, \trg})$. Hence we have \[\frac{M_i}{M_i \wedge k_i} = \frac{M_{i+1}}{M_{i+1} \wedge l_i} \text{ \ for every $i \in \llbracket 1,r+1 \rrbracket$ }\]
		and
		\[\frac{N_i}{N_i \wedge m_i} = \frac{N_{i+1}}{N_{i+1} \wedge n_i} \text{ \ for every $i \in \llbracket 1,u+1 \rrbracket$ }.\]
		Let us fix a prime number $p \in \mathcal{P}_{\mathcal{H},s'}(M')$. In particular one has $\sum_{i=1}^r|k_i|_p = \sum_{i=1}^r|l_i|_p$ for every cycle labeled $(k_1, l_1), ..., (k_r, l_r)$ (this condition is independent of $s)$. Let us prove that $p \in \mathcal{P}_{\mathcal{H}, s}(M)$: let $g_1, ..., g_t$ be an edge path of $\mathcal{H}$ which is based at $s$ and labeled $(a_1, b_1), ..., (a_t, b_t)$. Notice that the path $\overline{e_r},...,\overline{e_1},g_1,...,g_t$ is based at $s'$ and labeled $(l_r,k_r),...,(l_1,k_1),(a_1,b_1),...,(a_t,b_t)$. Hence, as $p \in \mathcal{P}_{\mathcal{H},s'}(M')$, one has \[|M'|_p > \sum_{i=1}^r |l_i|_p - \sum_{i=1}^r|k_i|_p + \sum_{i=1}^t|a_i|_p - \sum_{i=1}^{t-1}|b_i|_p. \]
		By Proposition \ref{pgcd} applied to the path $\overline{E_r},...\overline{E_1}$, one has 
		\[|M|_p = |M'|_p + \sum_{i=1}^r |k_i|_p - \sum_{i=1}^r|l_i|_p. \]
		Hence \[|M|_p > \sum_{i=1}^t|a_i|_p - \sum_{i=1}^{t-1}|b_i|_p\]
		which proves that \[p \in \mathcal{P}_{\mathcal{H},s}(M).\]
		Thus, $p \in \mathcal{P}_{\mathcal{H},s}(N)$ and by Proposition \ref{pgcd}, 
		\[|N'|_p = |N|_p + \sum_{i=1}^u |m_i|_p - \sum_{i=1}^u|n_i|_p.\]
		As $\Ph_{\mathcal{H}, s}(M) = \Ph_{\mathcal{H},s}(N)$ and $p \in \mathcal{P}_{\mathcal{H},s}(M)$, we have $|N|_p = |M|_p$ so that \[|N'|_p = |M'|_p + \left(\sum_{i=1}^r|k_i|_p - \sum_{i=1}^r|l_i|_p \right) + \left(\sum_{i=1}^u|m_i|_p - \sum_{i=1}^u|n_i|_p \right). \]
		Moreover, the edge path $e_1,...,e_r,f_1,...,f_u$ is labeled $(k_1,l_1),...,(k_r,l_r),(m_1,n_1),...,(m_u,n_u)$ and is a cycle. We deduce that \[|N'|_p = |M'|_p \text{ \ for every $p \in \mathcal{P}_{\mathcal{H}, s'}\left(M'\right)$ }\]
		and by symmetry \[|N'|_p = |M'|_p \text{ \ for every $p \in \mathcal{P}_{\mathcal{H}, s'}\left(N'\right)$ }\]
		so finally \[\Ph_{\mathcal{H},s'}\left(N'\right) = \Ph_{\mathcal{H},s'}\left(M'\right).\]
		
	\end{proof}

	As a consequence of this result, we get the following corollary: 
	
	\begin{corollary} The phenotypical decomposition does not depend on the choice of the vertex.
	\end{corollary}
	
	\begin{definition}
	    We say that two $\mathcal{H}$-preactions $\alpha, \beta$ share the same phenotype if for some (equivalently, for every) vertex $s \in \mathcal{V}(\mathcal{H})$, one has $\bm{\Ph}_{\mathcal{H}, s}(\alpha) = \bm{\Ph}_{\mathcal{H}, s}(\beta)$.
	\end{definition}
	
	\subsection{\texorpdfstring{Merging preactions of a given $\mathcal{H}$-phenotype}{Merging preactions of a given H-phenotype}}

	To obtain topological transitivity, the goal of this subsection is to extend two preactions sharing the same phenotype into a single transitive preaction (in fact, we prove a stronger statement to obtain \textit{high} topological transitivity (\textit{cf.} Lemma \ref{liaison})). \\
	We first need the following lemma that establishes a link between the $\mathcal{H}$-phenotype and the $\mathcal{H}'$-phenotype for some given subgraph $\mathcal{H}'$ of $\mathcal{H}$: 
	
	\begin{lemma}\label{subhgraph}
		Let us assume that $\mathcal{H}$ consists of a connected subgraph $\mathcal{H}'$ and an edge $e$ with source in $\mathcal{H}'$. Then for any $N \in \mathbb{N}$: \begin{enumerate}
			\item \label{inclusion1} $\Ph_{\mathcal{H}, \src(e)} \circ \Ph_{\mathcal{H}', \src(e)}(N) = \Ph_{\mathcal{H}, \src(e)}(N)$;
			\item \label{feuille} if $\trg(e) \notin \mathcal{V}\left(\mathcal{H}'\right)$, denoting by $(l,k)$ the label of $e$: \[\Ph_{\mathcal{H}', \src(e)}\left(\frac{N|l|}{N \wedge k}\right) = \Ph_{\mathcal{H}', \src(e)}\left(\frac{\Ph_{\mathcal{H}, \trg(e)}(N)|l|}{\Ph_{\mathcal{H}, \trg(e)}(N) \wedge k}\right); \]
			\item \label{pasfeuille} otherwise, denoting by $C$ a cycle of edges $e, e_1, ..., e_r$ (with $e_i \in \mathcal{E}\left(\mathcal{H}'\right)$ for every $i \in \llbracket 1, r \rrbracket$), if $|N|_p = 0$ for every $p \notin \mathcal{P}_{C, \src(e)}(N)$, then: \[\Ph_{\mathcal{H}, \src(e)}(N) = \Ph_{\mathcal{H}', \src(e)}(N).\]
		\end{enumerate}
	\end{lemma}
	
	\begin{proof}
		As $\mathcal{H}'$ is a subgraph of $\mathcal{H}$ containing $\src(e)$, one has $\mathcal{P}_{\mathcal{H}, \src(e)}(M) \subseteq \mathcal{P}_{\mathcal{H}', \src(e)}(M)$ for every integer $M$.
		Let us show that $\mathcal{P}_{\mathcal{H}, \src(e)}(\Ph_{\mathcal{H}', \src(e)}(N)) = \mathcal{P}_{\mathcal{H}, \src(e)}(N)$. For every prime number $p \in \mathcal{P}_{\mathcal{H}, \src(e)}(N)$, one has $|\Ph_{\mathcal{H}', \src(e)}(N)|_p = |N|_p$. Hence $p \in \mathcal{P}_{\mathcal{H}, \src(e)}\left(\Ph_{\mathcal{H}', \src(e)}(N)\right)$. 
		Conversely, if $p \in \mathcal{P}_{\mathcal{H}, \src(e)}\left(\Ph_{\mathcal{H}', \src(e)}(N)\right)$, then one has $p \in \mathcal{P}_{\mathcal{H}', \src(e)}\left(\Ph_{\mathcal{H}', \src(e)}(N)\right)$ by the previous observation which implies that $|\Ph_{\mathcal{H}', \src(e)}(N)|_p = |N|_p$.
		Thus $p \in \mathcal{P}_{\mathcal{H}, \src(e)}(N)$. We deduce the following equality:
		\begin{align*}
			\Ph_{\mathcal{H}, \src(e)}\left(\Ph_{\mathcal{H}', \src(e)}(N)\right) &= \prod_{p \in \mathcal{P}_{\mathcal{H}, \src(e)}\left(\Ph_{\mathcal{H}', \src(e)}(N)\right)}p^{|\Ph_{\mathcal{H}', \src(e)}(N)|_p} \\
			&= \prod_{p \in \mathcal{P}_{\mathcal{H}, \src(e)}(N)}p^{|\Ph_{\mathcal{H}', \src(e)}(N)|_p} \\
			&= \prod_{p \in \mathcal{P}_{\mathcal{H}, \src(e)}(N)}p^{|N|_p} \\
			&= \Ph_{\mathcal{H}, \src(e)}(N).
		\end{align*} 
		We deduce Item~\ref{inclusion1}. Notice that we only used the fact that $\mathcal{H}'$ is a subgraph of $\mathcal{H}$ containing $\src(e)$. \\
		\\
		Let us assume that $\trg(e) \notin \mathcal{V}\left(\mathcal{H}'\right)$. Let us denote $P = \Ph_{\mathcal{H}, \trg(e)}(N)$. If an edge path $c$ of $\mathcal{H}'$ labeled $(m_1,n_1),...,(m_r,n_r)$ is based at $\src(e)$ then the path $\overline{e}*c$, labeled $(k,l), (m_1,n_1),...,(m_r,n_r)$, is based at $\trg(e)$. As moreover the reduced cycles of $\mathcal{H}$ are in $\mathcal{H}'$ as soon as $\trg(e) \notin \mathcal{V}\left(\mathcal{H}'\right)$ we deduce the equality: \[\mathcal{P}_{\mathcal{H}',\src(e)}\left(\frac{N|l|}{N \wedge k}\right) = \mathcal{P}_{\mathcal{H},\trg(e)}\left(N\right) \bigsqcup \left\{p \in \mathcal{P}_{\mathcal{H}',\src(e)}(l) \mid |N|_p \leq |k|_p \right\}. \]
		Likewise
		\[\mathcal{P}_{\mathcal{H}',\src(e)}\left(\frac{P|l|}{P \wedge k}\right) = \mathcal{P}_{\mathcal{H},\trg(e)}\left(P\right) \bigsqcup \left\{p \in \mathcal{P}_{\mathcal{H}',\src(e)}(l) \mid |P|_p \leq |k|_p \right\}.\]
		On the one hand, one has $\mathcal{P}_{\mathcal{H}, \trg(e)}\left(N\right) = \mathcal{P}_{\mathcal{H}, \trg(e)}\left(P\right)$. \\
		On the other hand, if $|P|_p \leq |k|_p$ and $p \in \mathcal{P}_{\mathcal{H}',\src(e)}(l)$, then $p \notin \mathcal{P}_{\mathcal{H},\trg(e)}\left(N\right)$ so there exists an edge path of $\mathcal{H}'$ labeled $(m_1,n_1),...,(m_r,n_r)$ based at $\src(e)$ such that \[\left|N\right|_p \leq |k|_p - |l|_p + \sum_{i=1}^r|m_i|_p-\sum_{i=1}^{r-1}|n_i|_p.\]
		As we also have $p \in \mathcal{P}_{\mathcal{H}',\src(e)}(l)$, we deduce \[\left|N\right|_p \leq |k|_p.\]
		Conversely, if $\left|N\right|_p \leq |k|_p$, then $p \notin \mathcal{P}_{\mathcal{H},\trg(e)}\left(N\right)$ so $\left|P\right|_p = 0$. We deduce \begin{align*}
			\mathcal{P}_{\mathcal{H}',\src(e)}\left(\frac{N|l|}{N \wedge k}\right) &= \mathcal{P}_{\mathcal{H}, \trg(e)}\left(P\right) \bigsqcup \left\{p \in \mathcal{P}_{\mathcal{H}',\src(e)}(l) \mid \left|N\right|_p \leq |k|_p \right\}\\
			&= \mathcal{P}_{\mathcal{H}',\src(e)}\left(\frac{P|l|}{P \wedge k}\right).
		\end{align*}    
		Therefore: 
		\begin{align*}
			\Ph_{\mathcal{H}',\src(e)}\left(\frac{N|l|}{N \wedge k}\right) &= \prod_{p \in \mathcal{P}_{\mathcal{H},\trg(e)}\left(N\right)} p^{\left|N\right|_p + |l|_p - \min\left(\left|N\right|_p,|k|_p\right)} \\ & \text{\ \ \ \ \ \ \ \ } \times \prod_{p \in \mathcal{P}_{\mathcal{H}',\src(e)}(l), \left|N\right|_p \leq |k|_p} p^{\left|N\right|_p + |l|_p - \min \left(\left|N\right|_p,|k|_p \right)} \\
			&=  \prod_{p \in \mathcal{P}_{\mathcal{H},\trg(e)}\left(P\right)} p^{|N|_p + |l|_p - |k|_p}  \prod_{p \in \mathcal{P}_{\mathcal{H}',\src(e)}(l), |N|_p \leq |k|_p} p^{|l|_p} \\
			&= \Ph_{\mathcal{H}',\src(e)}\left(\frac{P|l|}{P \wedge k}\right).
		\end{align*}
		This proves Item~\ref{feuille}. \\
		\\
		Let us assume that $\trg(e) \in \mathcal{V}\left(\mathcal{H}\right)$ and that there exists a cycle of edges $C = e, e_1, ..., e_r$ labeled $(l,k), (l_1, k_1), ..., (l_r, k_r)$ (with $e_i \in \mathcal{E}\left(\mathcal{H}'\right)$ for every $i \in \llbracket 1, r \rrbracket$) satisfying: $|N|_p = 0$ for every $p \notin \mathcal{P}_{C, \src(e)}(N)$. Let us show that $\mathcal{P}_{\mathcal{H}, \src(e)}(N) = \mathcal{P}_{\mathcal{H}', \src(e)}(N)$. As $\mathcal{H}'$ is a subgraph of $\mathcal{H}$, one has $\mathcal{P}_{\mathcal{H}, \src(e)}(N) \subseteq \mathcal{P}_{\mathcal{H}', \src(e)}(N)$. Conversely, let $p \in \mathcal{P}_{\mathcal{H}', \src(e)}(N)$. In particular, one has $|N|_p \neq 0$ hence $p \in \mathcal{P}_{C,\src(e)}(N)$ which implies: \begin{equation}\label{cyclenul}|k|_p + \sum_{i=1}^r|k_i|_p = |l|_p + \sum_{i=1}^r|l_i|_p.\end{equation}
		Using Remark \ref{simplification}, let $C' = f_1, ..., f_s$ be a simple cycle of $\mathcal{H}$ labeled $(n_1, m_1), ..., (n_s, m_s)$. If $e$ does not belong to $C'$, then $C'$ lies in $\mathcal{H}'$ so $\sum_{i=1}^s|n_i|_p = \sum_{i=1}^s|m_i|_p$ as $p \in \mathcal{P}_{\mathcal{H}', \src(e)}(N)$. Otherwise, up to changing the base point of $C'$, one can assume that $f_s = e$, hence $(n_s, m_s) = (l,k)$. Hence, the edge path $e_1, ..., e_r, \overline{f_{s-1}}, ..., \overline{f_1}$ is a cycle in $\mathcal{H}'$ which is labeled $(l_1, k_1), ..., (l_r, k_r), (m_{s-1}, n_{s-1}), ..., (m_1, n_1)$ so as $p \in \mathcal{P}_{\mathcal{H}', \src(e)}(N)$ we get:
		\[\sum_{i=1}^r|l_i|_p - |k_i|_p = \sum_{i=1}^{s-1}|n_i|_p - |m_i|_p.\]
		Combining this with Equation \eqref{cyclenul} leads to the following fact: 
		\[\sum_{i=1}^s|m_i|_p = \sum_{i=1}^s|n_i|_p \text{ \ for any cycle of $\mathcal{H}$ labeled $(m_1, n_1),..., (m_s, n_s)$}.\]
		Now (using again Remark \ref{simplification}) let $C'' = g_1, ..., g_t$ be a simple path of $\mathcal{H}$ based at $\src(e)$ and labeled $(i_1, j_1)$, ..., $(i_t, j_t)$ which is either a loop or not a cycle. If $C''$ lies in $\mathcal{H}'$, then, as $p \in \mathcal{P}_{\mathcal{H}', \src(e)}(N)$, one has: \[|N|_p > \left(\sum_{k=1}^{t-1}|i_k|_p - |j_k|_p\right) + |i_t|_p.\] Otherwise, if moreover $t \geq 2$, one necessarily has $g_1 = e$ (hence $(i_1, j_1) = (l,k)$). The edge path $\overline{e_r}, ..., \overline{e_1}, g_2, ..., g_t$ lies in $\mathcal{H}'$, is based at $\src(e)$ and labeled $(k_r, l_r)$, ..., $(k_1, l_1)$, $(i_2, j_2)$, ..., $(i_t, j_t)$. Thus, as $p \in \mathcal{P}_{\mathcal{H}', \src(e)}(N)$:
		\[|N|_p > \left(\sum_{i=1}^{r}|k_i|_p - |l_i|_p\right) + \left(\sum_{k=2}^{t-1}|i_k|_p - |j_k|_p\right) + |i_t|_p\]
		Combining this inequality with Equation \eqref{cyclenul} leads to:
		\[|N|_p > \left(\sum_{k=1}^{t-1}|i_k|_p - |j_k|_p\right) + |i_t|_p.\]
		If $t=1$ (\textit{i.e.} $e$ is not a loop and $C'' = e$ labeled $(l,k)$ or $e$ is a loop and $C'' \in \{e, \overline{e}\}$ labeled $(l,k)$ or $(k,l)$), as $p \in \mathcal{P}_{C, \src(e)}(N)$ we already know that $|N|_p > |l|_p$ and that $|k|_p = |l|_p$ if $e$ is a loop. 
		Thus, we proved that $p \in \mathcal{P}_{\mathcal{H}, \src(e)}(N)$ which immediately implies Item~\ref{pasfeuille}.
		
	\end{proof}
		
	To alleviate notations, given an $\mathcal{H}$-preaction $\alpha$ defined by a set of partial bijections $\{\alpha_s, \tau_e \vert (s, e) \in \mathcal{V}(\mathcal{H})\times \mathcal{E}\left(\mathcal{H}\right) \setminus \mathcal{E}(T) \}$ on a countable set $X$, we will simply denote by $a_s$ the partially defined bijection $\alpha_s$ and by $t_e$ the partially defined bijection $\tau_e$.

	\begin{lemma}\label{liaison}
		Let us assume that $\mathcal{H}$ is reduced and different from a loop with at least one label equal to $\pm 1$.
		Let $\alpha_1,...,\alpha_{\Sigma}, \beta_1, ..., \beta_{\Sigma}$ be collections of transitive $\mathcal{H}$-preactions defined on countable pointed sets $(X_1, x_{1,0})$, ..., $(X_{\Sigma},x_{\Sigma,0})$, $(Y_1,y_{1,0})$,..., $(Y_{\Sigma},y_{\Sigma,0})$ such that there exist an edge $e_0 \in \mathcal{E}(\mathcal{H})$ and words $\mathfrak{m}, \mathfrak{m}'$ in the vertex and edge generators of $\Gamma$ (defined by $\mathcal{H}$)  satisfying: \begin{itemize}
			\item  $x_i := x_{i,0} \cdot \mathfrak{m} \in \dom\left(a_{\src(e_0)}\right)$ and $y_i := y_{i,0} \cdot \mathfrak{m}' \in \dom\left(a_{\src(e_0)}\right)$; 
			\item if $e_0 \in \mathcal{E}(T)$, then $x_i, y_i \notin \dom\left(a_{\trg(e_0)}\right)$, and if $e_0 \notin \mathcal{E}(T)$, then $x_i, y_i \notin \dom\left(t_{e_0}\right)$;
			\item $\alpha_i$ and $\beta_i$ share pairwise the same $\mathcal{H}$-phenotype.
		\end{itemize} 
		Then there exists a word $\mathfrak{M}$ in the vertex and edge generators and, for every $i \in \llbracket 1, \Sigma \rrbracket$, there exist $\mathcal{H}$-preactions $\gamma_i$ on countable sets $Z_i$ containing $X_i \cup Y_i$ extending $\alpha_i$ and $\beta_i$ such that $x_{i,0} \cdot \mathfrak{M} = y_{i,0}$.
		
		Moreover, for every $i \in \llbracket 1, \Sigma \rrbracket$, denoting by $\mathcal{F}^{(i)}$ the $\mathcal{H}$-graph of $\alpha_i$ and by $\mathcal{G}^{(i)}$ the one of $\beta_i$, the resulting $\mathcal{H}$-graph $\mathcal{H}^{(i)}$ contains both $\mathcal{F}^{(i)}$ and $\mathcal{G}^{(i)}$ as disjoint sub-$\mathcal{H}$-graphs and the quotient $\mathcal{H}^{(i)} / \left(\mathcal{F}^{(i)} \sqcup \mathcal{G}^{(i)}\right)$ is a tree.
	\end{lemma}

	\begin{proof}
		We proceed by induction on the number of positive edges of $\mathcal{H}$. The idea is to use a technique of ``dévissage'': writing $\mathcal{H}$ as $\mathcal{H} = \mathcal{H}' \cup \{e'\}$ for some subgraph $\mathcal{H}'$ and some edge $e'$ with source in $\mathcal{V}\left(\mathcal{H}'\right)$, we will build words $\mathfrak{m}_0$, $\mathfrak{m}'_0$ in the vertex and edge generators of $\Gamma$ such that there exists an edge $e'_0 \in \mathcal{E}(\mathcal{H}')$ satisfying the following conditions: \begin{itemize}
			\item for every $i \in \llbracket 1, \Sigma \rrbracket$, the elements $x_{i, 0} \cdot \mathfrak{m}_0$ and $y_{i, 0} \cdot \mathfrak{m}_0'$ belong to $\dom\left(a_{\src\left(e'_0\right)}\right)$;
			\item if $e'_0 \in \mathcal{E}(T)$, then neither $x_{i, 0} \cdot \mathfrak{m}_0$ nor $y_{i, 0} \cdot \mathfrak{m}'_0$ belongs to $\dom\left(a_{\trg\left(e'_0\right)}\right)$, and if $e'_0 \notin \mathcal{E}(T)$, then neither $x_{i, 0} \cdot \mathfrak{m}_0$ nor $y_{i, 0} \cdot \mathfrak{m}'_0$ belongs to $\dom\left(t_{e_0'}\right)$;
			\item for every $i \in \llbracket 1, \Sigma \rrbracket$, the $\mathcal{H}$-preactions defined by the single $\left\langle a_{\src\left(e'_0\right)} \right\rangle$-orbits of $x_{i, 0} \cdot \mathfrak{m}_0$ and $y_{i, 0} \cdot \mathfrak{m}'_0$, respectively, share the same $\mathcal{H}'$-phenotype $\Ph_{\mathcal{H}', \src\left(e'_0\right)}$.
		\end{itemize}
		This will allow us to apply the induction hypothesis to the collection of $\mathcal{H}'$-preactions defined by the $\left\langle a_{\src\left(e'_0\right)} \right\rangle$-orbits of $x_{i, 0} \cdot \mathfrak{m}_0$ and $y_{i, 0} \cdot \mathfrak{m}'_0$ and to conclude, because an $\mathcal{H}'$-preaction is a particular case of an $\mathcal{H}$-preaction. 
		
		Let us prove the proposition $P_r$: ``for every reduced graph of groups $\mathcal{H}$ that has $r$ positive edges and that is different from a loop with at least one label equal to $\pm 1$, the conclusions of Lemma \ref{liaison} holds.''   
		\paragraph{Base case:} If $r=1$, then $\mathcal{H}$ consists of two edges $e, \overline{e}$. \\
		\\
		\textbf{Case 1: $\mathcal{H}$ has a single vertex.} If $\mathcal{H}$ has a single vertex, $\Gamma = \BS(m,n)$ for some integers $m,n$ which are both different from $\pm 1$. We adapt the algorithm introduced in \cite[Theorem 4.17]{solitar1} to obtain high topological transitivity.  Let us denote $(k_{e, \src}, k_{e, \trg}) = (m,n)$. Up to exchanging $e$ and $\overline{e}$, we can assume that $e_0 = e$. Denote by $b$ the vertex generator and let $s = \src(e_0)$. The idea is to get rid off all ``bad'' primer numbers in the cardinalities of the $\langle b \rangle$-orbits: following first the edge $e$ a sufficiently large amount of time, we will define inductively the point $x_i \cdot t_e^k$ and its $\langle b \rangle$-orbit in such a way that we delete all the prime numbers $p$ in the cardinal $N^{(i)}_k$ of $x_i \cdot t_e^k \langle b \rangle$ that satisfy $|n|_p < |m|_p$ or $\left|N^{(i)}_k\right|_p \leq |m|_p$ for some $i$. Then, we follow the edge $\overline{e}$ a sufficiently large amount of time; namely, we define inductively the point $x_i \cdot t_e^kbt_e^{-l}$ in such a way that we delete all the prime numbers $p$ in the cardinal $M^{(i)}_l$ of $x_i \cdot t_e^kbt_e^{-l} \langle b \rangle$ that satisfy $|n|_p > |m|_p$ or $\left|N^{(i)}_k\right|_p \leq |n|_p$ for some $i$. We do the same starting from $y_i$, and we connect the resulting $\langle b \rangle$-orbits.
		
		Let $N^{(i)}$ be the cardinality of the $\langle b \rangle$-orbit of $x_i$. Following Construction \hyperlink{ext'}{A'} we define by induction on $k$ the point $x_i \cdot t_e^k$ by imposing the cardinality $N^{(i)}_k$ of its $\langle b \rangle$-orbit: $N^{(i)}_0 = N^{(i)}$ and 
		\begin{equation*}
			\left|N^{(i)}_{k+1}\right|_p=
			\begin{cases}
				\left|N^{(i)}_k\right|_p + |n|_p - |m|_p & \text{if } \left|N^{(i)}_k\right|_p > |m|_p \\
				0  & \text{otherwise} 
			\end{cases} 
		\end{equation*} (legit because the edge $e$ and the integers $N^{(i)}_{k+1}$ and $N^{(i)}_k$ satisfy the Transfer Equation in~\ref{ggraph})
		and we repeat the process until $ \left|N^{(i)}_{k}\right|_p=0$ for every $i \in \llbracket 1, \Sigma \rrbracket$ such that $\alpha_i$ has a finite phenotype and every prime number $p$ satisfying \begin{itemize}
			\item either $\left|N^{(i)}\right|_p \leq |m|_p$ for some $i \in \llbracket 1, \Sigma \rrbracket$;
			\item or $|n|_p < |m|_p$.
		\end{itemize}
		If $\alpha_i$ has infinite phenotype, we still have $N_k^{(i)} = \infty$.
		We fix the $k \in \mathbb{N}$ obtained at the end of this process. Let $M^{(i)} := N_k^{(i)}$ be the cardinality of the $\langle b \rangle$-orbit of $x_i \cdot t_e^k$. Following Construction \hyperlink{ext'}{A'}, we define its image under $t_e$ by creating a new point whose $\langle b \rangle$-orbit has cardinality $\frac{M^{(i)}|n|}{M^{(i)} \wedge m}$ (legit, because $M^{(i)}$, $\frac{M^{(i)}|n|}{M^{(i)} \wedge m}$ and the edge $e$ satisfy the Transfer Equation). Likewise, since $|n| \geq 2$, we can define the image of $x_i \cdot t_e^{k+1} b$ under $t_e^{-1}$ by letting the cardinality $M^{(i)}_1$ of its $\langle b \rangle$-orbit to be equal to $M^{(i)}$. \\
		Then we apply successively Construction \hyperlink{ext'}{A'} in order to define recursively $x_i \cdot t_e^{k+1} b t_e^{-l}$ by imposing the cardinality of its $\langle b \rangle$-orbit: $M^{(i)}_1 = M^{(i)}$ and 
		\begin{equation*}
			\left|M^{(i)}_{l+1}\right|_p=
			\begin{cases}
				\left|M^{(i)}_l\right|_p + |m|_p - |n|_p & \text{if } \left|M^{(i)}_l\right|_p > |n|_p \\
				0  & \text{otherwise} 
			\end{cases} 
		\end{equation*}
		and we repeat the process until $ \left|M^{(i)}_{l}\right|_p=0$ for every $i \in \llbracket 1, \Sigma \rrbracket$ such that $\alpha_i$ has a finite phenotype and every $p$ satisfying \begin{itemize}
			\item either $\left|M^{(i)}_l\right|_p \leq |n|_p$ for some $i \in \llbracket 1, \Sigma \rrbracket$;
			\item or $|m|_p < |n|_p$.
		\end{itemize}      
		Hence, the cardinality of the $\langle b \rangle$-orbit of $x_i t_e^{k+1} b t_e^{-l}$ is equal to the $\mathcal{H}$-phenotype $P^{(i)} := \Ph_{\mathcal{H},s}\left(N^{(i)}\right)$. 
		Denote by $W$ the word $t_e^{k+1}bt_e^{-l}$ and $W'$ the word obtained by applying the same algorithm to the family of points $\left(y_i\right)_{i \in \llbracket 1, \Sigma \rrbracket}$. 
		As neither $x_i \cdot W$ nor $y_i \cdot W'$ belongs to $\dom(t_e^{-1})$, we can apply Construction \hyperlink{ext'}{A'} in order to define a new point $x_i \cdot Wt_e^{-1}$ and build its $\langle b \rangle$-orbit of cardinality $\frac{P^{(i)}|m|}{P^{(i)} \wedge n}$ and as $|m| \geq 2$, one can impose $x_i \cdot Wt_e^{-1}b = y_i \cdot W't_e^{-1}$ (hence the $\langle b \rangle$-orbit of $y_i \cdot W't_e^{-1}$ has cardinality $\frac{P^{(i)}|m|}{P^{(i)} \wedge n}$). Thus the word $\mathfrak{M} = \mathfrak{m}Wt_e^{-1}b(W't_e^{-1})^{-1}\mathfrak{m}'^{-1}$ is suitable. \\
		\\
		\textbf{Case 2: $\mathcal{H}$ has two vertices.} Otherwise, $\mathcal{H}$ consists of two edges $e, \overline{e}$ such that $\src(e) \neq  \trg(e)$. We denote $k_{e,\src} = k$ and $k_{e, \trg} = l$, so that $\Gamma \simeq \left\langle a, b | a^k = b^l \right\rangle$ with $|k|, |l| \geq 2$. Without loss of generality let us assume that $e_0 = e$ and let us denote $s = \src\left(e_0\right)$. 
		Let $N^{(i)}$ (\textit{resp.} $M^{(i)}$) be the cardinality of the $\langle a \rangle$-orbit of $x_i$ (\textit{resp.} $y_i$). Following Construction \hyperlink{quot'}{B'}, we can define the $\langle b \rangle$-orbit of $x_i$ (\textit{resp.} $y_i$) by its cardinality $\frac{N^{(i)}|l|}{N^{(i)} \wedge k}$ (\textit{resp.} $\frac{M^{(i)}|l|}{M^{(i)} \wedge k}$) because the integers $N^{(i)}$ and $\frac{N^{(i)}|l|}{N^{(i)} \wedge k}$ and the edge $e$ satisfy the Transfer Equation so that, as $|l| \geq 2$, the image of $x_i \cdot b$ under $a$ is not defined. Let us apply Construction \hyperlink{quot'}{B'} in order to define the $\langle a \rangle$-orbit of $x_i \cdot b$ by its cardinality $\frac{N^{(i)}|k|}{N^{(i)} \wedge k}$ (\textit{resp.} $\frac{M^{(i)}|k|}{M^{(i)} \wedge k}$).
		We define \begin{align*}
			P^{(i)} &= \Ph_{\mathcal{H},s}\left(N^{(i)}\right) \\
			&= \prod_{p \in \mathcal{P}, \left|N^{(i)}\right|_p > |k|_p}p^{\left|N^{(i)}\right|_p} \\
			&(= \Ph_{\mathcal{H},s}\left(M^{(i)}\right)).
		\end{align*}  Notice that \[\frac{N^{(i)}}{N^{(i)} \wedge k} = \frac{P^{(i)}}{P^{(i)} \wedge k}.\]
		Hence applying Construction \hyperlink{quot'}{B'} we can define the $\langle b \rangle$-orbit of both $x_i \cdot ba$ and $y_i \cdot ba$ (because $|k| \geq 2$) by its common cardinality $\frac{P^{(i)}|l|}{P^{(i)} \wedge k}$ and define $y_i \cdot ba = x_i \cdot bab$ (because $|l| \geq 2$). Hence the word $\mathfrak{M} = \mathfrak{m}baba^{-1}b^{-1}\mathfrak{m}'^{-1}$ is suitable. 
		The construction we did in terms of $\mathcal{H}$-graphs is represented in Figure \ref{figureinitialisation}. 
		
		\begin{figure}[ht]
			\center
			\begin{tikzpicture}
				\node at (0,0) {$\mathcal{G}^{(i)}$};
				\node at (0,6) {$\mathcal{F}^{(i)}$};
				\draw (0,0) circle (2);
				\draw (0,6) circle (2);
				\node[draw,circle,fill=gray!50] (A) at (2,6) {$N^{(i)}$};
				\node[draw,circle,] (B) at (5,6) {$\frac{N^{(i)}|l|}{N^{(i)} \wedge k}$};
				\node[draw,circle,fill=gray!50] (C) at (8,6) {$\frac{N^{(i)}|k|}{N^{(i)} \wedge k}$};
				\node[draw,circle,fill=gray!50] (D) at (2,0) {$M^{(i)}$};
				\node[draw,circle] (E) at (5,0) {$\frac{M^{(i)}|l|}{M^{(i)} \wedge k}$};
				\node[draw,circle,fill=gray!50] (F) at (8,0) {$\frac{M^{(i)}|k|}{M^{(i)} \wedge k}$};
				\node[draw,circle] (G) at (10,3) {$\frac{P^{(i)}|l|}{P^{(i)} \wedge k}$};
				\draw[>=latex, directed] (A) to (B);
				\draw[>=latex, directed] (C) to (B);
				\draw[>=latex, directed] (D) to (E);
				\draw[>=latex, directed] (F) to (E);
				\draw[>=latex, directed] (F) to (G);
				\draw[>=latex, directed] (C) to (G);
				\node at (3,-3) {$\mathcal{H} =$};
				\node[draw,circle,fill=gray!50] (a) at (4,-3) {};
				\node[draw,circle] (b) at (6,-3) {};
				\draw[>=latex, directed] (a) to node[very near start, below]{$k$} node[very near end, below]{$l$} (b);
			\end{tikzpicture}
			\caption{Base case (case of a segment).}
			\label{figureinitialisation}
		\end{figure}
		\text{ } \\
		Hence the theorem is proved for $r=1$. 
		\paragraph{Induction step:} Let us assume that $P_r$ is true and let us consider a GBS group $\Gamma$ which is defined by a reduced graph of groups $\mathcal{H}$ which consists of $r+1$ positive edges. The idea will be to decompose $\mathcal{H}$ as \begin{itemize}
		    \item a connected subgraph $\mathcal{H}'$;
		    \item an edge $e$ with source in $\mathcal{H}'$.
		\end{itemize}
		Then, using the same kind of constructions as in the base case, we will build a word $W$ (\textit{resp.} $W'$) of $\Gamma$ in such a way that, denoting by $K^{(i)}$ (\textit{resp.} $K'^{(i)}$) the cardinality of the $\langle b \rangle$-orbit of $x_i \cdot W$ (\textit{resp.} $y_i \cdot W'$), one has $\bm{\Ph}_{\mathcal{H}', \src(e)}\left(K^{(i)}\right) = \bm{\Ph}_{\mathcal{H}', \src(e)}(K'^{(i)})$. This equality will be obtained using Lemma~\ref{subhgraph}. Then, the induction hypothesis will lead to a word that connects $x_i \cdot W$ and $y_i \cdot W'$.
		
		We distinguish two cases: \\
		\\
		\textbf{Case 1: $\mathcal{H}$ is a tree.} In particular, $T = \mathcal{H}$ and, denoting by $f$ any leaf of $\mathcal{H}$, the subgraph of groups $\mathcal{H}'$ induced by the vertices of $\mathcal{V}\left(\mathcal{H}\right) \setminus \{f\}$ satisfies the induction hypothesis. Let us denote by $G$ the fundamental group of $\mathcal{H}'$. In this case, denoting by $a$ the vertex generator associated to $f$, $b$ the vertex generator associated to the neighbour $g$ of $f$ and $e$ the edge which connects $g$ to $f$ (labeled $k$ at $f$ and $l$ at $g$), one has \begin{align*}\Gamma &\simeq \left\langle G, a | \mathcal{R}_G, b^l = a^k \right\rangle \\
			&= G *_{b^l = a^k} \left\langle a \right\rangle \end{align*}
		(where $\mathcal{R}_G$ is a set of relations defining $G$). \\
		\\
		\textbf{Step 1.1: Reducing to the case where $\mathbf{e_0 = \overline{e}}$}. We first prove that one can assume that the edge $e_0$ of the statement of the lemma can be assumed to be the edge $\overline{e}$ with source the leaf $f$. We define $s := \src(e_0)$ and we denote by $e_1,...,e_r$ the unique (possibly empty) reduced edge path of $\mathcal{E}\left(T\right)$ which connects $\trg(e_0)$ to $f$ (so $e_r = e$). Let us denote $s_i = \src(e_i)$ and $(k_i,l_i) = (k_{e_i,\src}, k_{e_i,\trg})$ (hence $(k_r, l_r) = (l,k)$). Let $N^{(i)}$ (\textit{resp.} $M^{(i)}$) be the cardinality of the $\langle a_s \rangle$-orbit of $x_i$ (\textit{resp.} $y_i$). 
		Let us apply Construction \hyperlink{quot'}{B'} and let $\kappa = a_{\trg(e_0)}$. We define the $\langle a_{\trg(e_0)} \rangle$-orbit of $x_i$ (\textit{resp.} $y_i$) by its cardinality $\left| \frac{N^{(i)} l_0}{N^{(i)} \wedge k_0}\right|$ (\textit{resp.} $\left| \frac{M^{(i)} l_0}{M^{(i)} \wedge k_0}\right|$), and we let $x_i' = x_i \cdot \kappa$ (\textit{resp.} $y_i' = y_i \cdot \kappa$). In that case, as $e_0$ is not a loop, one has $x_i \neq x_i'$ (\textit{resp.} $y_i \neq y_i'$) because no label of $e_0$ is equal to $\pm 1$, so the only generator whose domain contains $x_i'$ (\textit{resp.} $y_i'$) is $a_{\trg(e_0)}$. Applying successively Construction \hyperlink{quot'}{B'}, we define the $\langle a_{s_j} \rangle$-orbit of $x_i'$ (\textit{resp.} $y_i'$) by induction on $j$ by its cardinality \[N^{(i)}_{j+1} = \left| \frac{N^{(i)}_j l_j}{N^{(i)}_j \wedge k_j}\right|,\] 
		with $s_{r+1} = f$ (\textit{resp.} $M^{(i)}_{j+1} = \left| \frac{M^{(i)}_j l_j}{M^{(i)}_j \wedge k_j}\right|$). Denote $N'^{(i)}$ (\textit{resp.} $M'^{(i)}$) the cardinality of the $\langle a \rangle$-orbit of $x_i'$ (\textit{resp.} of $y_i'$). By Proposition \ref{decphen}, $\Ph_{\mathcal{H},f}\left(N'^{(i)}\right) = \Ph_{\mathcal{H},f}\left(M'^{(i)}\right) =: P^{(i)}$. \\
		\\
		\textbf{Step 1.2: Applying the induction hypothesis.} Notice that $\Ph_{\mathcal{H}',g}\left(\frac{N'^{(i)}|l|}{N'^{(i)} \wedge k}\right) = \Ph_{\mathcal{H}',g}\left(\frac{P^{(i)}|l|}{P^{(i)} \wedge k}\right)$ by the point \ref{feuille} of Lemma \ref{subhgraph}.
		Hence the numbers $N''^{(i)} := \frac{N'^{(i)}|l|}{N'^{(i)} \wedge k}$ and $M''^{(i)} := \frac{M'^{(i)}|l|}{M'^{(i)} \wedge k}$ share the same $\mathcal{H}'$-phenotype (relatively to the vertex $g$). So the induction hypothesis applied to \begin{itemize}
		    \item the graph $\mathcal{H}'$ and any edge of $\mathcal{H}'$ with source $g$;
		    \item a collection of $\mathcal{H}'$-preactions $O_1^{(1)}, \ldots, O_1^{(\Sigma)}, O_2^{(1)}, \ldots, O_2^{(\Sigma)}$ such that $O_1^{(i)}$ consists of a single $\langle b \rangle$-orbit of cardinal $N''^{(i)}$ and $O_2^{(i)}$ consists of a single $\langle b \rangle$-orbit of cardinal $M''^{(i)}$;
		    \item arbitrary points $\omega_1^{(1)}, \ldots, \omega_1^{(\Sigma)}, \omega_2^{(1)}, \ldots, \omega_2^{(\Sigma)}$ of $O_1^{(1)}, \ldots, O_1^{(\Sigma)}, O_2^{(1)}, \ldots, O_2^{(\Sigma)}$, respectively,
		\end{itemize}
		allows us to produce a family of $\Sigma$ auxiliary $\mathcal{H}'$-preactions containing points $\omega_1^{(i)}$, $\omega_2^{(i)}$ whose $\langle b \rangle$-orbits $O_1^{(i)}$ and $O_2^{(i)}$ are disjoint and of cardinals $N''^{(i)}$ and $M''^{(i)}$, respectively, and a word $\mathfrak{M}'$ in the generators of $G$ satisfying ${\omega_1}^{(i)} \cdot \mathfrak{M}' = {\omega_2}^{(i)}$. As $|l_r| \geq 2$ and neither $x_i' \cdot a$ nor $y_i' \cdot a$ belongs to the domain of any generator of $\mathcal{H}'$, one can apply Construction \hyperlink{quot}{B} to the element $x_i' \cdot a$ and ${\omega_1}^{(i)}$ and to the element $y_i' \cdot a$ and ${\omega_2}^{(i)}$, respectively, and the edge $e$. With this construction, the $\langle b \rangle$-orbit of $x_i' \cdot a$ (\textit{resp.} $y_i' \cdot a$) has cardinality $\frac{N'^{(i)}|l|}{N'{(i)} \wedge k}$ (\textit{resp.} $\frac{M'^{(i)}|l|}{M'{(i)} \wedge k}$) and we get $x_i' \cdot a \mathfrak{M}' = y_i' \cdot a$. Hence the word $\mathfrak{M}$ defined by $\mathfrak{m}\kappa a \mathfrak{M}'a^{-1}\kappa^{-1}\mathfrak{m}'^{-1}$ is suitable. In terms of $\mathcal{H}$-graphs, the construction we did is represented in Figure \ref{figureinductiontree}.
		\begin{figure}[ht]
			\center
			\begin{tikzpicture}
				\node at (0,0) {$\mathcal{G}^{(i)}$};
				\node at (0,4) {$\mathcal{F}^{(i)}$};
				\node at (9.5,2) {induction hypothesis};
				\node[draw,circle,fill=black] (t) at (1.5,0) {\color{white} $M^{(i)}$};
				\node[draw,circle,fill=black] (s) at (1.5,4) {\color{white} $N^{(i)}$};
				\draw (0,0) circle (1.5);
				\draw (0,4) circle (1.5);
				\draw[dashed] (9.5,2) ellipse (3cm and 1.5cm);
				\node[draw,circle,fill=gray!50] (A) at (4,4) {$N'^{(i)}$};
				\node[draw,circle] (B) at (6,3) {$\frac{N'^{(i)}|l|}{N'^{(i)} \wedge k}$};
				\node[draw,circle,fill=gray!50] (C) at (4,0) {$M'^{(i)}$};
				\node[draw,circle] (D) at (6,1) {$\frac{M'^{(i)}|l|}{M'^{(i)} \wedge k}$};
				\draw[>=latex, directed] (B) to (A);
				\draw[>=latex, directed] (D) to (C);
				\node at (3,-3) {$\mathcal{H} =$};
				\node[draw,circle,fill=gray!50] (a) at (4,-3) {};
				\node[draw,circle] (b) at (6,-3) {};
				\node[draw,circle, fill=black] (c) at (7,-3) {};
				\draw[>=latex, directed] (b) to node[midway, below]{$e$} node[very near start, below]{$l$} node[very near end, below]{$k$} (a);
				\draw[dashed] (7,-3) ellipse (1cm and 0.5cm);
				\draw[>=latex, directed, dashed] (t) to (C);
				\draw[>=latex, directed, dashed] (s) to (A);
			\end{tikzpicture}
			\caption{Applying the induction hypothesis (1st case).}
			\label{figureinductiontree}
		\end{figure}
		\text{ } \\
		\\
		\textbf{Case 2: $\mathcal{H}$ is not a tree}. Let $T$ be a spanning tree and $e \in \mathcal{E}(\mathcal{H}) \setminus \mathcal{E}(T)$. One has \[\Gamma \simeq \left\langle G, t | ta^nt^{-1} = b^m \right\rangle \]
		where $G$ is the GBS group defined by the subgraph of groups $\mathcal{H}'$ whose set of vertices is $\mathcal{V}\left(\mathcal{H}\right)$ and whose set of edges is $\mathcal{E}\left(\mathcal{H}\right) \setminus \{e, \overline{e}\}$, and where we denote $(a, b) = \left(a_{\src(e)}, a_{\trg(e)}\right)$ and $(n,m) = \left(k_{e, \src}, k_{e, \trg}\right)$. \\
		\\
		\textbf{Case 2.A: Assume first that we can choose $e$ so that $\mathcal{H}'$ satisfies the induction hypothesis.} This means that $\mathcal{H}'$ is not a loop one of whose label is equal to $\pm 1$. Let us denote $(s_a,s_b) := (\src(e), \trg(e))$. As in Step 1.1 of the previous Case 1, we extend the $\mathcal{H}$-preactions $\alpha_i$ and $\beta_i$ to get a vertex or edge generator $\kappa$ such that \begin{itemize}
			\item both $x_i' := x_i \cdot \kappa \in \dom(a)$ and $y_i' := y_i \cdot \kappa$ belong to $\dom(a)$;
			\item neither $x_i'$ nor $y_i'$ belongs to $\dom(t_e)$.
		\end{itemize} 
		Let us denote by $N'^{(i)}$ (\textit{resp.} $M'^{(i)}$) the cardinality of the $\langle a \rangle$-orbit of $x_i'$ (\textit{resp.} of $y_i'$). One has $\Ph_{\mathcal{H},s_a}\left(N'^{(i)}\right) = \Ph_{\mathcal{H},s_a}\left(M'^{(i)}\right)$. \\
		\\
		\textbf{Case 2.A.a: Let us first assume that $e$ is not a loop} (\textit{i.e.} $s_a \neq s_b$). Let $e_1,...,e_r$ be the edge path connecting $s_a$ to $s_b$ in $T$ and $(k_1,l_1),...,(k_r,l_r)$ be the labels of its edges. For every $j \in \llbracket 1,r \rrbracket$, we denote $s_j = \src(e_j)$ and $s_{r+1} = s_b$. Let us define the cycle $C := e, \overline{e_r}, ..., \overline{e_1}$. For every prime number $p$, let us denote \[C_p = \max \left(|n|_p, \max_{s \in \llbracket 0, r \rrbracket}\left(|n|_p - |m|_p + |l_s|_p + \sum_{j=s+1}^r|l_j|_p - |k_j|_p\right)\right)\] and \[\overline{C}_p = \max\left(|m|_p + \sum_{j=1}^r|k_j|_p - |l_j|_p, \max_{s \in \llbracket 0, r \rrbracket}\left(|k_{s+1}|_p + \sum_{m=1}^s|k_m|_p - |l_m|_p\right)\right),\] so that $p \in \mathcal{P}_{C, s_a}(N)$ (\textit{cf.} Definition \ref{defphenotype}) if and only if $|N|_p > \max\left(C_p, \overline{C}_p\right)$ and $\sum_{i=1}^r|k_i|_p-|l_i|_p + |m|_p - |n|_p = 0$ (notice that this last equality implies that $C_p = \overline{C}_p$). \\
		\\
		To apply the induction hypothesis, the idea of our algorithm is to follow $C$ a certain amount of time to define recursively the cardinality $N_j^{(i)}$ of the $\langle a \rangle$-orbit of $x_i' \cdot t_e^j$ so that, for every $i$ such that $\alpha_i$ has a finite phenotype, the $p$-adic valuation of $N_j^{(i)}$ is strictly decreasing with $j$ if $\left(\sum_{j=1}^r|k_j|_p-|l_j|_p\right) + |m|_p - |n|_p  < 0$ (Step 2.A.a.1). When it reaches $0$ for every $i$ such that $\alpha_i$ has finite phenotype, we apply $b$ and then we follow $\overline{C}$ to define recursively the cardinality $K_l^{(i)}$ of the $\langle a \rangle$-orbit of $x_i' \cdot t_e^{j+1}bt_e^{-l}$ so that its $p$-adic valuation strictly decreases with $l$ if $\left(\sum_{j=1}^r|k_j|_p-|l_j|_p\right) + |m|_p - |n|_p  > 0$ (Step 2.A.a.2). Notice that in infinite phenotype, the cardinals of all orbits remain infinite. At the end, for any $i$ such that $\alpha_i$ has finite phenotype, we will get $|K_l^{(i)}|_p = 0$ as soon as $p \notin \Ph_{C, s_a}\left(K_l^{(i)}\right)$ so we will be able to apply the point \ref{pasfeuille} of \ref{subhgraph} (Step 3). \\
		\\
		\textbf{Step 2.A.a.1: Removing the prime numbers satisfying $\mathbf{\left(\sum_{i=1}^r|k_i|_p-|l_i|_p\right) + |m|_p -}$ $\mathbf{|n|_p < 0}$.} For any integers $m, n$ we define the function $\varphi_{m,n} : \mathbb{N} \cup \{\infty\} \to \mathbb{N} \cup \{\infty\}$ by its $p$-adic valuation for every prime number $p$: for every $N \in \mathbb{N}$, we define  \[|\varphi_{m,n}(N)|_p= \begin{cases}
			\left|N\right|_p + |m|_p - |n|_p & \text{if } \left|N\right|_p > |n|_p \\
			0  & \text{otherwise} 
		\end{cases}\] and $\varphi_{m,n}(\infty) = \infty$.
		We define $M^{(i)}_{r+1} := \varphi_{m,n}\left(N'^{(i)}\right)$, 
		so that $\frac{N'^{(i)}}{N'^{(i)} \wedge n} = \frac{M^{(i)}_{r+1}}{M^{(i)}_{r+1} \wedge m}$. Hence, using Construction \hyperlink{ext'}{A'}, we can define a new point $x_i' \cdot t_e$ whose $\langle b \rangle$-orbit has cardinality $M^{(i)}_{r+1}$. Then we define recursively (with $j$ decreasing): \[M^{(i)}_j := \varphi_{k_j,l_j}\left(M^{(i)}_{j+1}\right)\] and, applying successively Construction \hyperlink{quot'}{B'}, we define the $\langle a_{s_j} \rangle$-orbits of $x_i' \cdot t_e$ by their cardinals $M^{(i)}_j$. Eventually, the $\langle a_{s_1} \rangle$-orbit of $x_i' \cdot t_e$ has cardinality $M^{(i)}_1$ where
		\begin{equation*}
			\left|M^{(i)}_{1}\right|_p=
			\begin{cases}
				\left|N'^{(i)}\right|_p + |m|_p - |n|_p + \sum_{j=1}^r|k_j|_p - |l_j|_p &\text{if } \left|N'^{(i)}\right|_p > C_p \\
				0  &\text{otherwise}. 
			\end{cases} 
		\end{equation*}
		It follows that, for every prime number $p$ satisfying the condition \eqref{eq: cond1} defined as follows: \begin{equation}\label{eq: cond1}\left(\sum_{j=1}^r|k_j|_p-|l_j|_p\right) + |m|_p - |n|_p  < 0 \text{ \ or \ } |N'^{(i)}|_p \leq C_p, \tag{$1$} \end{equation} one has $\left|M^{(i)}_1\right|_p < \left|N'^{(i)}\right|_p$ if $\alpha_i$ has a finite phenotype (otherwise, $M^{(i)}_1 = \infty$) and $\left|N'^{(i)}\right|_p \neq 0$. One also has $\left|M^{(i)}_1\right|_p = 0$ if $\left|N'^{(i)}\right|_p = 0$. We denote $N^{(i)}_1 = M^{(i)}_1$ and we repeat the process to build inductively on $j$ the new point $x_i' \cdot t_e^j$ and its $\langle a \rangle$-orbit of cardinality $N_j^{(i)}$. As $\left|N_{j+1}^{(i)}\right|_p < \left|N_j^{(i)}\right|_p$ for every prime number $p$ satisfying Condition~\eqref{eq: cond1} and such that $\left|N^{(i)}_j\right|_p \neq 0$ and every $i$ such that $\alpha_i$ has finite phenotype by construction, and as $\left|N^{(i)}_{j+1}\right|_p = 0$ if $\left|N^{(i)}_j\right|_p = 0$, there exists a $j \in \mathbb{N}$ such that: $\left|N^{(i)}_j\right|_p = 0$ for every $i \in \llbracket 1, \Sigma\rrbracket$ such that $\alpha_i$ has a finite phenotype and for every prime number $p$ satisfying one of the two following conditions: \begin{itemize} \item $\left(\sum_{j=1}^r|k_j|_p-|l_j|_p\right) + |m|_p - |n|_p  < 0$;
			\item $\left|N^{(i)}_j\right|_p \leq C_p$ for some $i \in \llbracket 1,\Sigma \rrbracket$.
		\end{itemize} 
		If $\alpha_i$ has infinite phenotype, we will simply get $N^{(i)}_j = \infty$. \\
		\\
		\textbf{Step 2.A.a.2: Removing the prime numbers satisfying $\mathbf{\left(\sum_{i=1}^r|k_i|_p-|l_i|_p\right) + |m|_p -}$ $\mathbf{|n|_p > 0}$.} Notice that the image of $x_i' \cdot t_e^j$ under $t_e$ is not yet defined. Hence, applying Construction \hyperlink{ext'}{A'}, we define a new point $x_i' \cdot t_e^{j+1}$ and its $\langle b \rangle$-orbit by its cardinality $\frac{N^{(i)}_j|m|}{N^{(i)}_j \wedge n}$. As $|m| \geq 2$, one has $x_i' \cdot t_e^{j+1} b \notin \dom(t_e^{-1})$. Thus, using Construction \hyperlink{ext'}{A'}, we can define a new point $x_i' \cdot t_e^{j+1} b t_e^{-1} =: x_i''$ and its $\langle a \rangle$-orbit by its cardinality $N^{(i)}_j$. Furthermore,$t_e$ and $a^{\pm 1}$ are the only generators whose domains contain $x_i''$. Let us denote $L^{(i)}_1 = N^{(i)}_j$. 
		
		By induction we define $L^{(i)}_{J+1} := \varphi_{l_J,k_J}\left(L^{(i)}_J\right)$ for every $J < r$, then $L^{(i)}_{r+1} = \varphi_{n,m}\left(L^{(i)}_r\right)$ and we use several times Construction \hyperlink{quot'}{B'} in order to define the $\langle a_{s_J} \rangle$-orbit of $x_i''$ by its cardinality $L^{(i)}_J$, then the image of $x_i''$ under $t_e^{-1}$ and its $\langle a \rangle$-orbit by its cardinality $L^{(i)}_{r+1}$. By construction one has:
		\begin{equation*}
			\left|L^{(i)}_{r+1}\right|_p=
			\begin{cases}
				\left|L^{(i)}_1\right|_p + |n|_p - |m|_p + \sum_{j=1}^r|l_j|_p - |k_j|_p &\text{if } \left|L^{(i)}_1\right|_p > \overline{C}_p \\
				0  &\text{otherwise}. 
			\end{cases} 
		\end{equation*}
		It follows that, for every prime number $p$ satisfying the condition \eqref{eq: cond2} defined as follows:  \begin{equation}\label{eq: cond2}\left(\sum_{j=1}^r|k_j|_p-|l_j|_p\right) + |m|_p - |n|_p  > 0 \text{ \ or \ } \left|L^{(i)}_1\right|_p \leq \overline{C}_p,\tag{$2$}\end{equation} we have $\left|L^{(i)}_{r+1}\right|_p < \left|L^{(i)}_1\right|_p$ if $\alpha_i$ has a finite phenotype and $\left|L^{(i)}_1\right|_p \neq 0$. One also has $\left|L^{(i)}_{r+1}\right|_p = 0$ if $\left|L^{(i)}_1\right|_p = 0$. We define $K^{(i)}_1 = L^{(i)}_{r+1}$ and we repeat the process in order to define the points $x_i'' \cdot t_e^{-l}$ and their $\langle a \rangle$-orbits of cardinality $K^{(i)}_l$. As for every prime number $p$ satisfying Condition~\eqref{eq: cond2} one has $\left|K^{(i)}_{l+1}\right|_p < \left|K^{(i)}_{l}\right|_p$ as soon as $\alpha_i$ has finite phenotype and $\left|K^{(i)}_l\right|_p \neq 0$, and as $\left|K^{(i)}_{l+1}\right|_p=0$ if $\left|K^{(i)}_l\right|_p = 0$, there exists an integer $l$ such that $\left|K^{(i)}_l\right|_p = 0$ for every $i \in \llbracket 1, \Sigma \rrbracket$ such that $\alpha_i$ has finite phenotype and for every $p$ such that: \begin{itemize} \item either $\left(\sum_{j=1}^r|k_j|_p-|l_j|_p\right) + |m|_p - |n|_p  > 0$;
			\item or $\left|K^{(i)}_l\right|_p \leq \overline{C}_p$ for some $i \in \llbracket 1,\Sigma \rrbracket$.
		\end{itemize} 
		In particular, if $\alpha_i$ has a finite phenotype, the cardinality $K^{(i)}_l =: K^{(i)}$ of the $\langle a \rangle$-orbit of the last created point $x_i'' \cdot t_e^{-l}$ satisfies the condition~\eqref{star} defined as follows:
		\begin{equation}\label{star} 
\left|K^{(i)}\right|_p = 0 \text{ \ for every prime $p$ such that} \ p \notin \mathcal{P}_{C, s_a}\left(K^{(i_0)}\right) \text{ for some $i_0 \in \llbracket 1,\Sigma \rrbracket$}.\tag{$\ast$}
\end{equation}

		We fix such an $l$. Notice that if we reapply the same algorithm in order to define $x_i'' \cdot t_e^{-(l+1)}$ and its $\langle a \rangle$-orbit of cardinality $K_{l+1}^{(i)}$, we get $K_{l+1}^{(i)} = K_l^{(i)}$. \\
		\\
		\textbf{Step 2.A.a.3: Applying the induction hypothesis}. Denoting by $x_i \cdot W := x_i'' \cdot t_e^{-l}$ (with $W := \kappa t_e^{j+1}bt_e^{-{l}}$), the only generators whose domains contain $x_i \cdot W$ are $t_e$ and $a$. Likewise we create new points in $Y_i$ and a word $W'$ such that $y_i \cdot W'$ has a $\langle a \rangle$-orbit of cardinality $K'^{(i)}$ satisfying Condition~\eqref{star} and such that only generators whose domains contain $y_i \cdot W'$ are $t_e$ and $a$. Let us reapply one last time the same algorithm to define the $\langle a_{s_J} \rangle$-orbits of $x_i \cdot W$ and of $y_i \cdot W'$ (the images of both elements under $t_e^{-1}$ are not defined yet).  \\
		Notice that, by the point \ref{pasfeuille} of Lemma \ref{subhgraph}, the equality $\Ph_{\mathcal{H},s_1}\left(K^{(i)}\right) = \Ph_{\mathcal{H},s_1}\left(K'^{(i)}\right)$ and condition~\eqref{star} imply that, for every $i$ such that $\alpha_i$ has finite phenotype: \[\Ph_{\mathcal{H}',s_1}\left(K^{(i)}\right) = \Ph_{\mathcal{H}',s_1}\left(K'^{(i)}\right).\]
		The equality is trivial in the case of an infinite phenotype, as every orbit has cardinality $\infty$.
		Thus, the induction hypothesis delivers a collection of $\Sigma$ auxiliary $\mathcal{H}'$-preactions containing points $\omega_1^{(i)}$, $\omega_2^{(i)}$ whose $\langle a \rangle$-orbits $O_1^{(i)}$ and $O_2^{(i)}$ are disjoint and of cardinals $K^{(i)}$ and $K'^{(i)}$, respectively, and a word $W''$ in the generators of $G$ satisfying $\omega^{(i)}_1 \cdot W'' = \omega^{(i)}_2$. Applying Construction \hyperlink{ext}{A} to $x_i \cdot W$ and $\omega^{(i)}_1$ and to $y_i \cdot W'$ and $\omega^{(i)}_2$, respectively (and the edge $\overline{e}$), allows us to define the images of $x_i \cdot W$ and of $y_i \cdot W'$ under $t_e^{-1}$ and delivers an extension of the $\mathcal{H}$-preactions $\alpha_i$, $\beta_i$ satisfying $x_i \cdot W t_e^{-1}W'' = y_i \cdot W't_e^{-1}$. Hence the word $\mathfrak{M} = \mathfrak{m}Wt_e^{-1}W''t_eW'^{-1}\mathfrak{m}'^{-1}$ is suitable which proves the result in the case where $e$ is not a loop. \\
		\\
		\textbf{Case 2.A.b: Let us assume that $e$ is a loop.} 
		
		\textbf{Case 2.A.b.1: Let us assume that $|m|, |n| \neq 1$.} Applying the algorithm of the Case 1 of the base case leads to the existence of an extension of the $\mathcal{H}$-preactions and a word $W$ (\textit{resp.} $W'$) in $\left\langle a, t_e \right\rangle$ such that the $\langle a \rangle$-orbit of $x_i' \cdot W$ (\textit{resp.} $y_i' \cdot W'$) has cardinality \begin{align*}
			P^{(i)} &= \Ph_{m,n}\left(N'^{(i)}\right) \\
			&= \prod_{p \in \mathcal{P}, \left|N'^{(i)}\right|_p > |m|_p = |n|_p}p^{\left|N'^{(i)}\right|_p}
		\end{align*}
		(\textit{resp.} $P'^{(i)} = \Ph_{m,n}\left(M'^{(i)}\right)$) and such that $x_i' \cdot W, y_i' \cdot W' \notin \dom\left(t_e^{-1}\right)$. One has \[\Ph_{\mathcal{H}',s}\left(P^{(i)}\right) = \Ph_{\mathcal{H}',s}\left(P'^{(i)}\right)\]
		by the point \ref{pasfeuille} of Lemma \ref{subhgraph}. So by induction hypothesis, there exists a collection of $\Sigma$ auxiliary $\mathcal{H}'$-preactions containing points $\omega_1^{(i)}$, $\omega_2^{(i)}$ whose $\langle a \rangle$-orbits $O_1^{(i)}$ and $O_2^{(i)}$ are disjoint and of cardinals $P^{(i)}$ and $P'^{(i)}$, respectively, and a word $W''$ in the generators of $G$ such that $\omega^{(i)}_1 \cdot W'' = \omega^{(i)}_2$. As $x_i' \cdot W, y_i' \cdot W' \notin \dom\left(t_e^{-1}\right)$, and no element of $O_1$ or of $O_2$ belongs to $\dom(t_e)$, and $\frac{P^{(i)}}{P^{(i)} \wedge n} = \frac{P^{(i)}}{P^{(i)} \wedge m}$, we can apply Construction \hyperlink{ext}{A} to $x_i' \cdot W$ and $\omega^{(i)}_1$ and to $y_i' \cdot W'$ and $\omega^{(i)}_2$ (and the edge $\overline{e}$). We obtain extensions of the $\mathcal{H}$-preactions such that $x_i' \cdot W t_e^{-1} \cdot W'' = y_i' \cdot W' t_e^{-1}$, which leads to the conclusion with $\mathfrak{M} = \mathfrak{m}\kappa Wt_e^{-1}W''t_eW'^{-1}\kappa^{-1}\mathfrak{m}'^{-1}$. Hence we deduce the result in the case where $|m|, |n| \neq 1$. 
		
		\textbf{Case 2.A.b.2: Assume that $|m|=1$ or $|n|=1$}. Up to replacing $e$ by $\overline{e}$, one can assume without loss of generality that $|m|=1$. Let us define $N^{(i)}_1$ by 
		\[N^{(i)}_1 = \frac{N'^{(i)}}{N'^{(i)} \wedge n}.\]
		It follows that $\left|N^{(i)}_1\right|_p < \left|N'^{(i)}\right|_p$ for every prime number $p$ satisfying $0 < \left|N'^{(i)}\right|_p \leq |n|_p$ and every $i$ such that $\alpha_i$ has finite phenotype, and $\left|N^{(i)}_1\right|_p = 0$ if $\left|N'^{(i)}\right|_p = 0$. Then we apply Construction \hyperlink{ext'}{A'} to construct the image of $x_i'$ under $t_e$ and define the cardinality of its $\langle a \rangle$-orbit by $N^{(i)}_1$. We repeat the process and build the points $x'_i \cdot t_e^l$ and their $\langle a \rangle$-orbits of cardinals $N^{(i)}_l$ until \[N^{(i)}_l = \prod_{p \in \mathcal{P}, \left|N'^{(i)}\right|_p > |n|_p = 0}p^{\left|N'^{(i)}\right|_p} \] for every $i \in \llbracket 1, \Sigma \rrbracket$ such that $\alpha_i$ has a finite phenotype (hence $N_{l+1}^{(i)} = N_l^{(i)}$).
		Likewise, we create points $y_i' \cdot t_e^k$ in $Y_i$ whose $\langle a_s \rangle$-orbits have cardinals $M^{(i)}_k$ until  \[M^{(i)}_k = \prod_{p \in \mathcal{P}, \left|M'^{(i)}\right|_p > |n|_p = 0}p^{\left|M'^{(i)}\right|_p} \] for every $i \in \llbracket 1, \Sigma \rrbracket$ such that $\alpha_i$ has a finite phenotype. As $\Ph_{\mathcal{H}', s}\left(N^{(i)}_l\right) = \Ph_{\mathcal{H}', s}\left(M^{(i)}_k\right)$ by the point \ref{pasfeuille} of Lemma \ref{subhgraph} if $\alpha_i$ has finite phenotype (the equality being trivial if all orbits are infinite), the induction hypothesis delivers a collection of $\Sigma$ auxiliary transitive $\mathcal{H}'$-preactions containing points $\omega^{(i)}_1$, $\omega^{(i)}_2$ whose $\langle b \rangle$-orbits $O_1^{(i)}$ and $O_2^{(i)}$ are disjoint and of cardinals $N_l^{(i)}$ and $M_k^{(i)}$, respectively, and a word $W$ in the generators of $G$ such that $\omega^{(i)}_1 \cdot W = \omega^{(i)}_2$. Applying Construction \hyperlink{ext}{A} to $x'_i \cdot t_e^l$ and any point of $O_1$ and to $y_i' \cdot t_e^k$ and any point of $O_2$ (and to the edge $e$) gives rise to an extension of the actions such that $x_i' \cdot t_e^{l+1} W = y_i' \cdot t_e^{k+1}$. Hence the word $\mathfrak{M} = \mathfrak{m}\kappa t_e^{l+1}Wt_e^{-(k+1)}\kappa^{-1}\mathfrak{m}'^{-1}$ is suitable. The conclusion follows in this case. \\
		\\
		\textbf{Case 2.B: Let us assume that $e$ cannot be chosen so that $\mathcal{H}'$ satisfies the induction hypothesis} (\textit{i.e.} so that $\mathcal{H}'$ is not a loop with one label equal to $\pm 1$). This is equivalent to $\mathcal{H}$ consisting of two loops $e_0, f$ based at a single vertex $v$ and labeled $(k, \pm 1)$ (or $(\pm 1, k)$) and $(l, \pm 1)$ for some integers $k,l \in \mathbb{Z} \setminus \{0\}$, respectively. In that case, the $\mathcal{H}$-phenotype is defined by \[\bm{\Ph}_{\mathcal{H},v}(N) = \prod_{p \in \mathcal{P} \mid |k|_p = |l|_p = 0}p^{|N|_p}.\] Let $e \in \{e_0, \overline{e_0}\}$ be labeled $(k, \pm 1)$. Let us denote by $s$ the edge generator associated to $e$, by $t$ the edge generator associated to $f$ and by $b=a_v$ the vertex generator. As previously, we first build the image of $x_i$ (\textit{resp.} $y_i$) under the edge generator $t_{e_0}$ (which is either $t$ or $t^{-1}$), and we denote by $N'^{(i)}$ (\textit{resp.} $M'^{(i)}$) its cardinality. Then, we inductively build the image $x_i \cdot t_{e_0} t^m$ (\textit{resp.} $y_i \cdot t_{e_0} t^m$) and its $\langle b \rangle$-orbit of cardinality $N^{(i)}_m$ (\textit{resp.} $M^{(i)}_m$), where $N^{(i)}_0 = N'^{(i)}$ (\textit{resp.} $M_0^{(i)} = M'^{(i)}$) and
		
		\begin{equation*}
		N^{(i)}_{m+1} = \frac{N^{(i)}_m}{N^{(i)}_m \wedge l}
		\end{equation*}
		(and $\left(M^{(i)}_m\right)_m$ satisfies the same defining relation). We stop the process when $N^{(i)}_m$ and $M^{(i)}_m$ are coprime to $l$ for every $i$ such that $\alpha_i$ has finite phenotype. Following the same procedure, we build the images $x_i \cdot t_{e_0} t^m s^n$ (\textit{resp.} $y_i \cdot t_{e_0} t^ms^n$) and their $\langle b \rangle$-orbits in such a way that, for $n$ large enough, their cardinalities $C_i$ (\textit{resp.} $C_i'$) are coprime to $k$ and $l$ for every $i$ such that $\alpha_i$ has finite phenotype. This implies that $C_i=C_i'$, the common phenotype of $\alpha_i$ and $\beta_i$. Then, denoting by $\gamma = t_{e_0}t^ms^n$ and using Construction~\hyperlink{ext}{A}, we extend $\alpha_i$ and $\beta_i$ by setting \[x_i \cdot \gamma t = y_i \cdot \gamma\]
		(legit because $\frac{C_i}{C_i \wedge l} = \frac{C_i'}{C_i' \wedge 1} = C_i$). Hence the word $\mathfrak{M} = \gamma t \gamma'^{-1}$ allows us to conclude in this last case.

	\end{proof}

	\section{Perfect kernel}
	
	With the tools we introduced, we are now able to describe the perfect kernel of a non-amenable GBS group $\Gamma$, and to study the dynamics induced by the action by conjugation on it. 
	
	\subsection{Description of the perfect kernel}
	
	The aim of this section is to show that the perfect kernel of a GBS group $\Gamma$ consists exactly on the set of subgroups whose graph of groups (given by the action of $\Gamma$ on $\mathcal{A}$) is infinite.

	\begin{lemma}
		Let $\Gamma$ be a finitely generated group and let $N$ be a normal subgroup of $\Gamma$ that is Noetherian (\textit{i.e.} all of its subgroups are finitely generated). Let $\pi: \Gamma \to \Gamma/N$ be the projection. Then $\mathcal{K}(\Gamma) \subseteq \pi^{-1}\left(\Sub_{[\infty]}(\Gamma/N)\right)$, the set of subgroups of $\Gamma$ whose image under $\pi$ have infinite index in $\Gamma/N$.
	\end{lemma}
	
	\begin{proof}
		Let $\Lambda$ be a subgroup of $\Gamma$ such $\pi(\Lambda)$ has finite index in $\Gamma/N$. Let us show that $\Lambda \notin \mathcal{K}(\Gamma)$. \\ 
		\\
		Let $\mathcal{V} = \left\{\Lambda' \in \Sub(\Gamma), \pi(\Lambda) \subseteq \pi(\Lambda') \right\}$. For every $\Lambda' \in \mathcal{V}$, $\pi(\Lambda') = \Lambda'/(\Lambda' \cap N)$ has finite index in $\Gamma/N$, which is finitely generated as a quotient of a finitely generated group by a normal subgroup. Hence $\pi(\Lambda')$ is finitely generated. One has an exact sequence \[\xymatrix{
			1 \ar[r] & \Lambda' \cap N \ar[r]^{\iota} & \Lambda' \ar[r]^{\pi} & \pi(\Lambda') \ar[r] & 1
		}.\] By the assumption made on $N$, the group $\Lambda' \cap N$ is finitely generated as a subgroup of $N$. As $\Lambda'$ is generated by the generators of $\Lambda' \cap N$ and the preimages of the generators of $\pi(\Lambda')$ under $\pi$, we deduce that $\Lambda'$ is finitely generated. Hence, all elements of $\mathcal{V}$ are finitely generated so $\mathcal{V}$ is countable. \\
		\\
		The set $\mathcal{V}$ is a neighborhood of $\Lambda$: let $S = \left\{ \pi(t_1), ..., \pi(t_r) \right\}$ be a finite set generating $\pi(\Lambda)$. Then $\mathcal{V} = \cap_{i=1}^r\cup_{n \in N}\{\Lambda' \in \Sub(\Gamma), t_in \in \Lambda'\}$ is open. \\
		\\
		Hence $\mathcal{V}$ is a countable neighborhood of $\Lambda$ in $\Sub(\Gamma)$ so $\Lambda \notin \mathcal{K}(\Gamma)$. 
	\end{proof}
	
	One deduces the following result: 
	
	\begin{corollary}\label{unimod}
		If $\Gamma$ is a unimodular GBS group and $N$ an infinite cyclic normal subgroup of $\Gamma$, then, denoting by $\pi: \Gamma \to \Gamma/N$ the canonical surjection, one has \begin{align*}\mathcal{K}(\Gamma) &\subseteq \pi^{-1}\left(\Sub_{[\infty]}(\Gamma/N)\right)\\
		&= \{\Lambda \leq \Gamma \mid \Lambda \backslash \Gamma / \langle a \rangle \text{ \ is infinite} \} \end{align*} for every elliptic element $a$.
	\end{corollary}
	
	\begin{proof}
	   By the previous lemma, one only need to prove the last equality. As $N$ and $\langle a \rangle$ are commensurable, the set $\Lambda \backslash \Gamma / \langle a \rangle $ is finite if and only if the set $\Lambda \backslash \Gamma / N $ is finite. This precisely means that $\pi(\Lambda)$ has finite index in $\Gamma / N$.
	\end{proof}
	
	In the case of non-unimodular GBS groups, we have a more explicit characterization of subgroups whose graph of groups is infinite: 
	
	\begin{proposition}\label{nonunimod}
		Let $\Gamma$ be a non-unimodular GBS group and let $\Lambda$ be a subgroup of $\Gamma$. Then, the set $\Lambda \backslash \Gamma / \left\langle a \right\rangle$ is finite for some non-trivial elliptic element $a \in \Gamma$ if and only if $\Lambda$ has finite index in $\Gamma$. 
	\end{proposition}
	
	\begin{proof}
	If $\Lambda$ has finite index in $\Gamma$, then $\Lambda \backslash \Gamma / \langle a \rangle$ is finite.
	
		Conversely, let us assume that $\Lambda \backslash \Gamma / \langle a \rangle$ is finite. We denote by $\pi: \Lambda \backslash \mathcal{A} \to \Gamma \backslash \mathcal{A}$ the projection. First notice that the graph $\Lambda \backslash \mathcal{A}$ is finite: as the vertex set of $\mathcal{A}$ is indexed by $\cup_{\alpha \ \text{vertex generator}}\Gamma / \left\langle \alpha \right\rangle$, the vertex set of the quotient graph $\Lambda \backslash \mathcal{A}$ is indexed by $\cup_{\alpha \ \text{vertex generator}}\Lambda \backslash \Gamma / \left\langle \alpha \right\rangle$. As any two non-trivial elliptic elements of a GBS group are commensurable, $\Lambda \backslash \Gamma / \left\langle \alpha \right\rangle$ is finite for every elliptic element $\alpha$. Hence, as the number of vertex generators is also finite, the set of vertices of $\Lambda \backslash \mathcal{A}$ is finite. Since the degree of any vertex of $\mathcal{A}$ is finite (in particular this holds for the graph $\Lambda \backslash \mathcal{A}$), we deduce that $\Lambda \backslash \mathcal{A}$ is finite. 
		
		As $\Gamma$ is non-unimodular, we consider a cycle of edges $e_1,...,e_N$ (where $\trg(e_N) = \src(e_1)$) in $\mathcal{H}$ such that, denoting by $(k_i,l_i)$ the label of the edge $e_i$ one has \[\left|\prod_{i=1}^Nk_i \right| \neq \left|\prod_{i=1}^Nl_i \right|.\]  Assume by contradiction that $\Lambda \cap \left\langle a \right\rangle = \{1\}$ for some elliptic element $a$. By commensurability, this is equivalent to: $\Lambda \cap \left\langle \alpha \right\rangle = \{1\}$ for every elliptic element $\alpha$. In particular, for every $\pi$-preimage $\widetilde{\src(e_i)}$ of the vertex $\src(e_i) = \trg(e_{i-1})$ in $\Lambda \backslash \mathcal{A}$, the vertex $\widetilde{\src(e_i)}$ has exactly $|k_i|$ incident edges which project onto $e_i$ and $|l_i|$ incident edges which project onto $e_{i-1}$. For every $i \in \llbracket 1,N \rrbracket$, let us denote \begin{align*}s_i &= \left|\pi^{-1}\left(\src(e_i)\right)\right| \\ &= \left|\pi^{-1}\left(\trg(e_{i-1})\right)\right| \end{align*} and $a_i = \left|\pi^{-1}(e_i)\right|$. Then we have:  \begin{align*}
			a_i &= |k_i| s_i \\
			&= |l_i| s_{i+1}.
		\end{align*}
		
		Hence, for every $i \in \llbracket 1,N \rrbracket$: \[|k_i| s_i = |l_i| s_{i+1}\]
		and by induction, for every $r \in \llbracket 1,N \rrbracket$: 
		\begin{align*}
			s_1 \prod_{i=1}^N|k_i| &= s_{N+1} \prod_{i=1}^N|l_i| 
		\end{align*}
		so for $r=N$ and noticing that $s_{N+1} = s_1$: \[\left|\prod_{i=1}^Nk_i \right| = \left|\prod_{i=1}^Nl_i\right| \]
		which contradicts the hypothesis made on the cycle $(e_1,...,e_N)$. \\
		Thus there exists $k \in \mathbb{N}^*$ such that $\Lambda \cap \langle a \rangle = \langle a^k \rangle$. Let us consider the function $f : \begin{array}{ccccc}
			\Lambda \backslash\ \Gamma & \to & \Lambda \backslash\ \Gamma / \langle a \rangle \\
			\Lambda \gamma & \mapsto & \Lambda \gamma \langle a \rangle \\
		\end{array}.$
		By commensurability of the vertex stabilizers (\textit{cf.} Proposition \ref{ell}), for every $\gamma \in \Gamma$, there exists non-zero integers $N_{\gamma}$, $M_{\gamma}$ such that $\gamma a^{N_{\gamma}} \gamma^{-1} = a^{M_{\gamma}}$. Denoting by $n_{\gamma} := k N_{\gamma}$, we deduce: $\gamma a^{n_{\gamma}}\gamma^{-1} \in \Lambda$. Hence $f^{-1}(\left\{\Lambda \gamma \langle a \rangle \right\})$ contains at most $n_{\gamma}$ elements, which proves that all the fibers of $f$ are finite. As the set $\Lambda \backslash\ \Gamma / \langle a \rangle$ is finite by assumption and $f$ is surjective we deduce that $\Gamma / \Lambda$ is finite. 
	\end{proof}

	We are now able to understand the perfect kernel of a GBS group. The following theorem generalizes \cite[Theorem A]{solitar1}. 
	
	\begin{theorem}\label{kernel1}
		Let $\Gamma$ be a GBS group. Then:
		\begin{itemize}
			\item if $\Gamma \simeq \BS(1,n)$ for some $n \in \mathbb{Z} \setminus \{0\}$, one has $\mathcal{K}(\Gamma) = \emptyset$;
			\item if $\Gamma$ is unimodular and non-amenable, then, denoting by $N$ an infinite cyclic normal subgroup of $\Gamma$, one has $\mathcal{K}(\Gamma) = \pi^{-1}\left(\Sub_{[\infty]}(\Gamma / N)\right)$;
			\item otherwise \textit{i.e.} if $\Gamma$ is a non-unimodular and non-amenable GBS group, then $\mathcal{K}(\Gamma) = \Sub_{[\infty]}(\Gamma)$.
		\end{itemize}
	\end{theorem}
	
	\begin{proof}
		For every $n \in \mathbb{Z} \setminus \{0\}$, the group $\BS(1,n)$ has countably many subgroups (we refer to \cite[Corollary 8.4]{BS1N} for the proof of this statement). \\
		\\
		Let $\Gamma$ be a non-amenable GBS group. By Corollary \ref{unimod} and Proposition \ref{nonunimod}, one has
		
		\[
\{\Lambda \leq \Gamma \mid \Lambda \backslash \Gamma / \left\langle a \right\rangle \ \text{is infinite} \}  = \left\{
    \begin{array}{ll}
        \pi^{-1}(\Sub_{[\infty]}(\Gamma/N)) & \mbox{if } \Gamma  \text{ has a cyclic normal subgroup $N$;}\\
        \Sub_{[\infty]}(\Gamma) & \mbox{otherwise.}
    \end{array}
\right.
\]
		So by Corollary \ref{unimod}, it suffices to show that, given an elliptic element $a$ of $\Gamma$ \[\{\Lambda \leq \Gamma \mid \Lambda \backslash \Gamma / \left\langle a \right\rangle \ \text{is infinite} \} \subseteq \mathcal{K}(\Gamma).  \]
		Equivalently, using the correspondence between $\mathcal{H}$-graphs and graphs of subgroups of $\Gamma$ (\textit{cf.} Section \ref{corr}) and the fact that the degree of any vertex of $\mathcal{A}$ is finite, we want to show that \[\{\Lambda \leq \Gamma \mid \Lambda \text{ \ has an infinite $\mathcal{H}$-graph} \} \subseteq \mathcal{K}(\Gamma).\]
		Let $\Lambda \leq \Gamma$ having an infinite $\mathcal{H}$-graph $\mathcal{G}$. Let $\alpha$ be the associated transitive and saturated $\mathcal{H}$-preaction on a countable set $(X,x)$, fix a vertex $V \in \mathcal{V}\left(\mathcal{G}\right)$ deriving from $x$ and $R >0$ and let $\beta$ be a sub-$\mathcal{H}$-preaction of $\alpha$ whose $\mathcal{H}$-graph $K \subseteq \mathcal{G}$ is finite and contains the $R$-ball around $V$. The subgraph $K$ is a non-saturated $\mathcal{H}$-graph (otherwise, as the degree of any vertex of an $\mathcal{H}$-graph is finite, $\mathcal{G}$ would be finite). By Lemma \ref{completion1}, there exists a completion of $K$ into a saturated $\mathcal{H}$-graph $\mathcal{G}_1$ such that the quotient $\mathcal{G}_1 / K$ is a forest. So by Lemma \ref{retour}, there exists a subgroup $\Lambda_1$ whose $\mathcal{H}$-graph is $\mathcal{G}_1$. 
		
		We want to complete $K$ in another saturated $\mathcal{H}$-graph $\mathcal{G}_2$ which is non-isomorphic to $\mathcal{G}_1$. By Lemma \ref{arete} and Lemma \ref{boucle}, there exists a non-saturated $\mathcal{H}$-preaction whose $\mathcal{H}$-graph $L$ is non-saturated and non-simply connected and has the same $\mathcal{H}$-phenotype as $K$. Using Lemma \ref{completion1} and Lemma \ref{liaison} (with $\Sigma=1$), we deduce an $\mathcal{H}$-graph $\mathcal{G}_2$ which contains $K$ and $L$ as disjoint subgraphs such that the quotient of $\mathcal{G}_2 / (K \sqcup L)$ is a forest. Hence by Lemma \ref{retour}, there exists a group $\Lambda_2$ whose $\mathcal{H}$-graph is $\mathcal{G}_2$.
		
		The graphs $\mathcal{G}_1$ and $\mathcal{G}_2$ are infinite and the groups $\Lambda_1$ and $\Lambda_2$ are different (using the correspondence between $\mathcal{H}$-graphs and graphs of subgroups of $\Gamma$, because the rank of the fundamental group of $\mathcal{G}_2$ is strictly larger than the one of $\mathcal{G}_1$). Moreover, the $R$-ball around the origin in the Schreier graph (with respect to the vertex and edge generators) of the $\Gamma$ right action corresponding to $\Lambda_i$ coincides with the one corresponding to $\Lambda$ (\textit{cf.} Section \ref{schreier}) for $i = 1,2$. Thus we approximated $\Lambda$ by a subgroup whose $\mathcal{H}$-graph is infinite, which proves that the closure of the set $\{\Lambda \leq \Gamma \mid \Lambda \text{ \ has an infinite $\mathcal{H}$-graph} \}$ is perfect. The conclusion follows in the case where $\Gamma$ is not isomorphic to $\BS(1,n)$ for some $n \in \mathbb{Z} \setminus \{0\}$. 
	\end{proof}

	\subsection{Dynamics on the perfect kernel}
	
	In Theorem \ref{kernel1}, we decomposed the perfect kernel into a countable disjoint union which is preserved by the action by conjugation. A natural question is to describe more precisely the topology of those pieces. The following result gives an answer:
	
	\begin{proposition}\label{topologiepieces}
		For any vertex $s \in \mathcal{V}\left(\mathcal{H}\right)$, \begin{itemize} 
			\item $\bm{\Ph}_{\mathcal{H},s}^{-1}(\infty)$ is closed and non-empty;
			\item for any $N \in \mathcal{Q}_{\mathcal{H}, s}$, the piece $\bm{\Ph}_{\mathcal{H},s}^{-1}(N)$ is open. 
		\end{itemize}
		Moreover, for any $N \in \mathcal{Q}_{\mathcal{H}, s}$, the piece $\bm{\Ph}_{\mathcal{H},s}^{-1}(N)$ is closed if and only if $\Gamma$ is unimodular.
	\end{proposition}
	
	\begin{proof}
		One has \[\bm{\Ph}_{\mathcal{H},s}^{-1}(\infty) = \{\Lambda \in \Sub(\Gamma), a_s^i \notin \Lambda \ \forall i \in \mathbb{N}^* \}\]
		which is closed. As it contains the trivial group, it is non-empty.
		
		Let $N \in \mathcal{Q}_{\mathcal{H}, s}$.
		One has \[\bm{\Ph}_{\mathcal{H},s}^{-1}(N) = \bigsqcup_{M \in \Ph_{\mathcal{H},s}^{-1}(N)} \left\{\Lambda \in \Sub(\Gamma) \mid a_s^M \in \Lambda, a_s^i \notin \Lambda \ \forall i \in \llbracket 1,M-1 \rrbracket \right\}.\]
		Hence, $\bm{\Ph}_{\mathcal{H},s}^{-1}(N)$ is open as a union of basic clopen sets. \\
		\\
		For now assume that $\Gamma$ is unimodular and let us show that $\Ph_{\mathcal{H},s}^{-1}(N)$ is a finite subset of $\mathbb{N}$. For every prime number $p$, let us denote \[a_p := \max \left\{\sum_{j=1}^{i+1}|k_j|_p - \sum_{j=1}^i|l_j|_p \vert (k_1,l_1),...,(k_i,l_i) \text{ \ labels of an edge path based at $s$} \right\}. \]
		Notice that $a_p = 0$ for all but finitely many $p \in \mathcal{P}$. By unimodularity, one has, for every $M \in \Ph_{\mathcal{H},s}^{-1}(N)$: \[N = \prod_{p \in \mathcal{P}, |M|_p > a_p}p^{|M|_p}.\]
		Let $M \in \Ph_{\mathcal{H},s}^{-1}(N)$. If $p$ divides $M$, then \begin{itemize}
			\item either $|M|_p \leq a_p$, which implies that $a_p \neq 0$;
			\item or $|M|_p = |N|_p$, which implies that $|N|_p \neq 0$.
		\end{itemize}
		In other words, $p$ belongs to the set $\{q \in \mathcal{P} | a_q \neq 0 \} \cup \{q \in \mathcal{P} | |N|_q \neq 0\}$, which is finite. Moreover, as for every prime number $p$, one has $|M|_p \leq \max(a_p,|N|_p)$, the set $\Ph_{\mathcal{H},s}^{-1}(N)$ is a finite set of integers, hence $\bm{\Ph}_{\mathcal{H},s}^{-1}(N)$ is closed in $\Sub(\Gamma)$ as a finite union of clopen sets. 
		
		Now assume that the group $\Gamma$ is non-unimodular and let us show that the closure of a non-empty piece $\bm{\Ph}_{\mathcal{H},s}^{-1}(N)$ intersects $\bm{\Ph}_{\mathcal{H},s}^{-1}(\infty)$ non-trivially. Denoting by $b$ the generator associated to the vertex $s$, we first build a subgroup $\Lambda$ such that the set of indices $\left\{\left[\langle b \rangle : g\Lambda g^{-1}\cap \langle b \rangle \right], g \in \Gamma\right\}$ is unbounded. First, let us show that the subset of integers $\Ph_{\mathcal{H},s}^{-1}(N)$ is unbounded. By non-unimodularity, there exists a prime number $p_0$ and a cycle in $\mathcal{H}$ labeled $(k_1,l_1),...,(k_r,l_r)$ such that \[\sum_{i=1}^r |k_i|_{p_0} \neq \sum_{i=1}^r |l_i|_{p_0}.\] In particular, $p_0 \notin \mathcal{P}_{\mathcal{H},s}(N)$ and hence \[\Ph_{\mathcal{H},s}\left(Np_0^n\right) = \Ph_{\mathcal{H},s}(N)\] for every $n \in \mathbb{N}$. By induction on $n$ using Lemma \ref{liaison} (for $\Sigma = 1$), we build a sequence of non-saturated $\mathcal{H}$-graphs $(\mathcal{G}_n)_{n \in \mathbb{N}}$ such that, for every $n \in \mathbb{N}$, the $\mathcal{H}$-graph $\mathcal{G}_n$ is a tree that contains $n+1$ vertices labeled $\left(s, Np_0^k\right)$ (for $k \in \llbracket 0, n \rrbracket$), and $\mathcal{G}_n \subseteq \mathcal{G}_{n+1}$. Thus the $\mathcal{H}$-graph $\mathcal{G} := \cup_{n \in \mathbb{N}}\mathcal{G}_n$ is a tree that contains a vertex labeled $(s, Np_0^n)$ for every $n \in \mathbb{N}$. We saturate $\mathcal{G}$ into an $\mathcal{H}$-graph $\hat{\mathcal{G}}$ using Lemma \ref{completion1} so that the quotient $\hat{\mathcal{G}} / \mathcal{G}$ is a tree. Hence, by Lemma \ref{retour}, there exists a subgroup $\Lambda \in \mathcal{K}(\Gamma)$ whose $\mathcal{H}$-graph is $\hat{\mathcal{G}}$. In particular, there exist elements $g_n \in \Gamma$ such that \[g_n\Lambda g_n^{-1} \cap \langle b \rangle = \left\langle b^{Np_0^n} \right\rangle.\] By compacity of $\mathcal{K}\left(\Gamma\right)$, up to passing to a subsequence, $(g_n\Lambda g_n^{-1})_{n \in \mathbb{N}}$ converges to a subgroup $\Lambda_0 \in \mathcal{K}(\Gamma)$. If $b^k \in \Lambda_0$ for some $k \in \mathbb{N}$, then $b^k \in g_n \Lambda g_n^{-1} \cap \langle b \rangle$ for $n$ large enough, \textit{i.e.} $Np_0^n$ divides $k$ for infinitely many integers $n$. Hence, one has necessarily $k=0$ and $\Lambda_0 \in \bm{\Ph}_{\mathcal{G}_{\Gamma,s}}^{-1}(\infty)$.
	\end{proof}
	
	In fact, the decomposition we obtained is ``minimal'', in the sense that any two subgroups sharing the same phenotype have arbitrarily small neighborhoods that have a non-trivial intersection, up to conjugating them by some elements of $\Gamma$. We have a stronger statement:
	
	\begin{theorem}\label{toptrans}
		Let $\Gamma$ be a non-amenable GBS group defined by a reduced graph of groups $\mathcal{H}$. Let us fix a vertex $s$ of $\mathcal{H}$. For every $N \in \mathcal{Q}_{\mathcal{H}, s} \cup \{\infty\}$, the action by conjugation of $\Gamma$ on $\mathcal{K}(\Gamma) \cap \bm{\Ph}_{\mathcal{H},s}^{-1}(N)$ is highly topologically transitive. 
	\end{theorem}
	
	To prove Theorem \ref{toptrans}, we need to understand how to read words on the level of $\mathcal{H}$-graphs, the main subtlety coming from the fact that $\Gamma$ does not necessarily act on a given $\mathcal{H}$-graph. Let $\alpha$ be a $\Gamma$-action on a countable set $X$ and $\mathcal{G}$ be its $\mathcal{H}$-graph. Observe that the data of an element $x \in X$ and a word $(c, \mu) = \left((e_1, ..., e_r), \left(a_{\src(e_1)}^{k_1}, ..., a_{\src(e_r)}^{k_r}, a_{\trg(e_r)}^{k_{r+1}}\right)\right)$ defines an edge path $E_1, ..., E_r$ of $\mathcal{G}$: for every $i \in \llbracket 1, r \rrbracket$, the edge $E_i$ is the edge of type $e_i$ deriving from $x \cdot \big|(e_1, ..., e_{i-1}),$ $\left(a_{\src(e_1)}^{k_1}, ..., a_{\src(e_{i-1})}^{k_{i-1}}, a_{\src(e_i)}^{k_{i}}\right)\big|$ (\textit{cf.} Section \ref{GBS group}). If the edge path $E_1, ..., E_r$ is reduced, then $(c, \mu)$ is a reduced word, but the converse is not necessarily true: for instance, if $\alpha$ is the trivial action of $\Gamma$ on a set $X$ reduced to one point, its $\mathcal{H}$-graph is itself, hence finite. As there exist reduced words of arbitrarily large length, some of them lead to non-reduced paths in $\mathcal{H}$. But a partial converse relying on the labels of the vertices of $\mathcal{G}$ is given in the following lemma:
	
	\begin{lemma}\label{backtrack}
		Let $\alpha$ be a saturated and transitive $\mathcal{H}$-preaction on a pointed countable set $(X,x)$ such that there exists an edge $E$ labeled $e$ in its $\mathcal{H}$-graph $\mathcal{G}$ deriving from an element $x \in X$ satisfying the following conditions: \begin{itemize}
			\item every reduced edge path of $\mathcal{G}$ beginning with $E$ is a tree;
			\item for every reduced edge path $E_1=E,...,E_r$ labeled $e_1, ..., e_r$, denoting by $(\src(e_i), N_i)$ the label of $\src(E_i)$, the integer $N_{i+1}$ is divisible by $k_{e_i, \trg}$.
		\end{itemize}
		Then, for every reduced word $\left(c, \mu\right)$ of type $e_1=e, e_2, ..., e_r$ (\textit{cf.} Section \ref{GBS group}), the edge path of $\mathcal{G}$ defined by $x$ and $\left(c, \mu\right)$ is reduced. In particular, this edge path lies in the half graph of $E$.
	\end{lemma}
	
	\begin{proof}
		Let us denote $\left(c, \mu\right) = \left((e_1, ..., e_r), \left(a_{\src(e_1)}^{k_1}, a_{\src(e_2)}^{k_2}, ..., a_{\src(e_r)}^{k_r}, a_{\trg(e_r)}^{k_{r+1}}\right)\right)$ and let $E_1=E$, ..., $E_r$ be the edge path defined by $\left(c, \mu\right)$ and $x$ in $\mathcal{G}$. For every $i \in \llbracket 1, r \rrbracket$, let us denote by $e_i$ the label of $E_i$ and by $(\src(e_i), N_i)$ the label of $\src(E_i)$. Assume by contradiction that this edge path is not reduced. Then there exists $i \in \llbracket 1, r-1 \rrbracket$ such that $E_{i+1} = \overline{E_i}$, which implies that $e_{i+1} = \overline{e_i}$ and that, denoting by $y = x \cdot \left|\left(e_1, ..., e_i\right), \left(a_{\src(e_1)}^{k_1}, a_{\src(e_2)}^{k_2}, ..., a_{\src(e_i)}^{k_i}, a_{\trg(e_i)}^{k_{i+1}}\right)\right|$, one has:
		\[y \cdot \left\langle a_{\trg(e_i)}^{k_{e_i, \trg}} \right\rangle = y \cdot a_{\trg(e_i)}^{k_{i+1}} \left\langle a_{\trg(e_i)}^{k_{e_i, \trg}} \right\rangle \]
		which implies that $N_{i+1}$ divides $k_{i+1}$. As by assumption the integer $N_{i+1}$ is itself divisible by $k_{e_i, \trg}$, one deduces that $k_{e_i, \trg}$ divides $k_{i+1}$ which contradicts the irreducibility of $\left(c, \mu\right)$. Conclusion follows.
	\end{proof}
	
	We first prove a weaker version of \cite[Lemma 2.17]{hightrans} to recover the situation of Lemma \ref{liaison}:
	\begin{lemma}\label{unesortie}
		Let $\alpha$ be an $\mathcal{H}$-action on a pointed countable set $(X,x)$ whose $\mathcal{H}$-graph $\mathcal{G}$ contains a finite $\mathcal{H}$-graph $K$ such that the quotient $\mathcal{G} / K$ is a tree which is not reduced to a single vertex. Let $e$ be an arbitrary edge of $\mathcal{H}$ and $g \in \left\langle a_{\src(e)} \right\rangle$. Then there exists a reduced word $(c, \mu) = \left((e, ...), (g, ...)\right)$ beginning with $g$ such that the target of the edge path defined by $x$ and $(c, \mu)$ is not in $K$.
	\end{lemma}
	
	\begin{proof}
		Let $E$ be the edge of type $e$ deriving from $x \cdot g$ and let us denote by $V_0$ the source of $E$. Let us denote by $\mathcal{E}$ the set of reduced edge paths of $\mathcal{G}$ beginning with $E$. Suppose first that it contains a reduced cycle, say $C = (E=E_1,...,E_r)$ and let us consider an edge path $C' = (F_1,...,F_s)$ which is a tree connecting $V_0$ to a vertex outside $K$. If $F_1 \neq \overline{E_r}$, then the edge path $E_1, ..., E_r, F_1, ..., F_s$ is reduced, has its target outside $K$ and is defined by $x$ and a reduced word $(c, \mu) = \left((e, ...), (g, ...)\right)$. Otherwise, let $i_0 = \max \{i \in \llbracket 1, \min(s, r) \rrbracket \mid F_i = \overline{E_{r-i+1}}\}$. As $C'$ is a tree and $C$ is a cycle, one has necessarily $i_0 < r$. Thus, the edge path $E_1, ..., E_{r-i_0}, F_{i_0+1}, ..., F_s$ (which is simply equal to $E_1, ..., E_{r-i_0}$ in the case where $i_0=s$) is reduced, has its target outside $K$ and is defined by $x$ and a reduced word $(c, \mu) = \left((e, ...), (g, ...)\right)$.
		
		Otherwise, every reduced edge path beginning with $E$ is a tree. If there exists an element of $\mathcal{E}$ with target outside $K$, then the conclusion follows. If such an element does not exist, the graph $K$ being finite, let us consider a maximal element $E=E_1,...,E_r$ of $\mathcal{E}$ (for the inclusion of paths) and let us denote by $e_i$ the label of $E_i$ (so that $e_1=e$). By maximality, the vertex $\trg(E_r)$ has a single incident edge (which is $E_r$). Notice that this situation cannot happen in the case where $e_r$ is a loop, because in that case every saturated vertex of type $\src(e_r)$ has at least an incoming and an outgoing edge labeled $e_r$. Denoting by $s_i = \src(e_i)$ for every $i \in \llbracket 1, r \rrbracket$, this edge path leads to a reduced word $(c, \mu) = \left((e_1, ..., e_r), (g, a_{s_2}^{k_2},..., a_{s_r}^{k_r}, 1)\right)$ such that \[x \cdot \left|c, \mu\right| \left\langle a_{\trg(e_r)}^{k_{e_r, \trg}} \right\rangle = x \cdot \left|c, \mu\right| a_{\trg(e_r)} \left\langle a_{\trg(e_r)}^{k_{e_r, \trg}} \right\rangle. \]
		Let us denote $\alpha = a_{\trg(e_r)}$.
		Hence there exists an integer $m \in \mathbb{Z}$ such that \[x \cdot \left|c, \mu\right| = x \cdot \left|c, \mu\right| \alpha^{1+k_{e_r, \trg}m}.\]
		As $\mathcal{H}$ is reduced and $e_r$ is not a loop, one has $k_{e_r, \trg} \neq 1$. Hence, the word \[\left(c', \mu'\right) = \left((e_1, ..., e_r, \overline{e_r}, ..., \overline{e_1}), (g, a_{s_2}^{k_2},..., a_{s_r}^{k_r}, \alpha, a_{s_r}^{-k_r}, ..., a_{s_2}^{-k_2}, g^{-1})\right)\] is reduced and the edge path defined by $x$ and $\left(c', \mu'\right)$ is $E_1, ..., E_r, \overline{E_r}, ..., \overline{E_1}$, so has target $V_0$. \\
		Let us consider a reduced word $\left(d, \delta\right)$ such that the edge path defined by $x$ and $\left|d, \delta \right|$ has target not in $K$ (as no assumption on the type of $\delta$ is required, the existence of $\left(d, \delta\right)$ is simply guarantied by the finiteness of $K$). As $\mathcal{E}$ does not contain any edge path with target not in $K$ by assumption, the concatenation $(c', \mu')*(d, \delta)$ is reduced and of the form $\left((e, ...), (g, ...)\right)$, which allows us to conclude. 
	\end{proof}
	
	Now we can prove the main argument of Theorem \ref{toptrans}:
	
	\begin{lemma}\label{sortie}
		Let $(\alpha_i)_{1 \leq i \leq \Sigma}$ be a collection of $\mathcal{H}$-actions on pointed countable sets $(X_i, x_i)$ whose $\mathcal{H}$-graphs $\mathcal{G}^{(i)}$ satisfy the following conditions: \begin{itemize}
			\item $\mathcal{G}^{(i)}$ contains a finite subgraph $K_i$ such that the quotient $\mathcal{G}^{(i)} / K_i$ is an infinite tree;
			\item for every edge $E_i \in \mathcal{E}\left(\mathcal{G}^{(i)}\right) \setminus \mathcal{E}(K_i)$ labeled $e$ whose half graph is in $\mathcal{G}^{(i)} \setminus K_i$, denoting by $\left(\trg(e), N_i\right)$ the label of $\trg\left(E_i\right)$, the integer $N_i$ is divisible by $k_{e, \trg}$.
		\end{itemize}
		Then, there exists an element $\gamma \in \Gamma$ such that for every $i \in \llbracket 1, \Sigma \rrbracket$, there exists a vertex deriving from $x_i \cdot \gamma$ in $\mathcal{G}^{(i)} \setminus K_i$. 
	\end{lemma}
	
	\begin{proof}
		
		We argue by induction on $\Sigma$. 
		\paragraph{Base case} The case where $\Sigma=1$ results from Lemma \ref{unesortie}. 
		\paragraph{Induction step} Let $\Sigma \geq 2$ and let us assume that the result is true for a collection of $\Sigma-1$ actions satisfying the assumptions of Lemma \ref{sortie}. Let $\gamma \in \Gamma$ such that for every $i \in \llbracket 1, \Sigma-1 \rrbracket$, the vertex of type $\trg(e_r)$ deriving from $x_i \cdot \gamma$ is in $\mathcal{G}^{(i)} \setminus K_i$ and let us write $\gamma = \left| c, \mu \right|$ for some reduced word $\left(c, \mu\right)$ of type $e_1, ..., e_r$. If $e_r$ is a loop, let $e=e_r \neq \overline{e_r}$ and $g=1$, otherwise let $e = \overline{e_r}$ and $g = a_{\trg(e_r)}$. By Lemma \ref{unesortie} applied to $\alpha_{\Sigma}$, there exists a reduced word $(d,\delta) = \left((e, f_2, ..., f_s), (g, ...)\right)$ such that the vertex of type $\trg(f_s)$ deriving from $x_r \cdot \gamma |d, \delta|$ is not in $K_r$. \\
		Let $i \in \llbracket 1, r-1 \rrbracket$ and $V_i$ be the vertex of type $\src(e)$ deriving from $x_i \cdot \gamma g$ in $\mathcal{G}^{(i)}$. Let us denote by $E_i$ the edge labeled $e$ with source $V_i$ deriving from $x_i \cdot \gamma g$. If $e_r$ is a loop, one has $e \neq \overline{e_r}$ so the half graph of $E_i$ is in $\mathcal{G}^{(i)} \setminus K_i$. Likewise, if $e_r$ is not a loop, as $k_{e, \src}$ divides $\left|x_i \cdot \gamma \left\langle a_{\src(e)}^{k_{e, \src}} \right\rangle\right|$ and $k_{e, \src} \neq 1$, one has \[x_i \cdot \gamma \left\langle a_{\src(e)}^{k_{e, \src}} \right\rangle \neq x_i \cdot \gamma a_{\src(e)} \left\langle a_{\src(e)}^{k_{e, \src}} \right\rangle.\] Hence the half graph of $E_i$ is in $\mathcal{G}^{(i)} \setminus K_i$. Hence by Lemma \ref{backtrack}, the vertex of type $\trg(f_s)$ deriving from $x_i \cdot \gamma \left|d, \delta\right|$ is not in $K_i$. So the element $\gamma \left|d, \delta\right|$ is suitable, which concludes the induction step.
		
	\end{proof}
	
	We are now able to prove Theorem \ref{toptrans}:
	
	\begin{proof}[Proof of Theorem \ref{toptrans}]
		Let $\Lambda_1,..., \Lambda_{2S}$ be a collection of $2S$ subgroups which belong to $\mathcal{K}(\Gamma)$ and satisfying \[\bm{\Ph}_{\mathcal{H},s}(\Lambda_i) = \bm{\Ph}_{\mathcal{H},s}(\Lambda_{S+i})\] for every $i \in \llbracket 1, S \rrbracket$.
		For every $i \in \llbracket 1, 2S \rrbracket$, let $\alpha_i$ be a $\Gamma$-right action on a countable set $(X_i,x_{i})$ satisfying $\Stab(x_{i}) = \Lambda_i$ and let $\mathcal{G}^{(i)}$ be the $\mathcal{H}$-graph of $\alpha_i$ (seen as an $\mathcal{H}$-action). Let $V^{(i)} \in \mathcal{V}\left(\mathcal{G}^{(i)}\right)$ be the vertex of type $s$ deriving from $x_{i}$ and fix $R > 0$. Let $\beta_i$ be a sub $\mathcal{H}$-preaction of $\alpha_i$ defined on a countable subset $Y_i \subseteq X_i$ containing $x_{i}$ whose $\mathcal{H}$-graph $K_i$ contains the $R$-ball $B_{\mathcal{G}^{(i)}}\left(V^{(i)}, R\right)$. Let $\mathcal{F}^{(i)}$ be the completion of $K_i$ given by Lemma \ref{completion1} and $\widetilde{\alpha_i}$ be an extension of $\beta_i$ on a countable set $\left(\widetilde{X_i}, x_i\right)$ containing $Y_i$ whose $\mathcal{H}$-graph is $\mathcal{F}^{(i)}$ (whose existence is given by Lemma \ref{retour}). 
		
		By Lemma \ref{sortie} there exists a word $(c, \mu)$ such that there exists a vertex deriving from $x_i \cdot |c, \mu|$ in $\mathcal{F}^{(i)} \setminus K_i$ for every $i \in \llbracket 1, 2S \rrbracket$. By Proposition \ref{stab1}, we can apply Lemma \ref{liaison} to the sub $\mathcal{H}$-preactions of $\widetilde{\alpha_i}$ defined by $\beta_i$ and the prefixes of $(c, \mu)$: there exist $\mathcal{H}$-preactions $\gamma_i$ extending $\beta_i$ and $\beta_{S+i}$ and a word $m$ in the generators of $\Gamma$ such that $x_i \cdot m = x_{i+S}$ for every $i \in \llbracket 1, 2S \rrbracket$. Denoting by $\mathcal{L}^{(i)}$ the $\mathcal{H}$-graph of $\gamma_i$, we use Lemma \ref{completion1} to build an $\mathcal{H}$-graph $\mathcal{H}^{(i)}$ containing $\mathcal{L}^{(i)}$ such that the quotient $\mathcal{H}^{(i)} / \mathcal{L}^{(i)}$ is a forest. By construction, $\mathcal{H}^{(i)}$ contains $K_i$ and $K_{i+S}$ as disjoint subgraphs and the quotient $\mathcal{H}^{(i)} / \left(K_i \sqcup K_{i+S}\right)$ is a tree. Applying Lemma \ref{retour} to $\mathcal{H}^{(i)}$ for every $i \in \llbracket 1, 2S \rrbracket$ leads to the existence of an $\mathcal{H}$-action (hence a subgroup $\tilde{\Lambda_i}$ of $\Gamma$) satisfying the following conditions: \begin{itemize}
			\item the Schreier graph of $\tilde{\Lambda_i}$ and the one of $\Lambda_i$ share the same $R$-ball around the origin;
			\item the Schreier graph of $m\tilde{\Lambda_i}m^{-1}$ and the one of $\Lambda_{S+i}$ share the same $R$-ball around the origin.
		\end{itemize} 
		Hence, the conclusion follows.
	\end{proof}
	
	The following result then relies on the topological transitivity on each piece of the partition and on an application of Baire's theorem:
	
	\begin{corollary}\label{gdelta}
		For every $N \in \mathbb{N} \cup \{\infty\}$ there exists a dense $G_{\delta}$ subset of $\bm{\Ph}_{\mathcal{H}, s}^{-1}(N) \cap \mathcal{K}(\Gamma)$ whose elements consist of subgroups having dense conjugacy class in $\bm{\Ph}_{\mathcal{H}, s}^{-1}(N) \cap \mathcal{K}(\Gamma)$. 
	\end{corollary}
	
	Notice that the piece $\bm{\Ph}_{\mathcal{H}, s}^{-1}(\infty)$ has an intrinsic characterization in the case of a non-amenable GBS group $\Gamma$: this is the set of subgroups of $\Gamma$ which do not contain any non-trivial elliptic element, \textit{i.e.} any non-trivial element which is commensurable to its conjugates by Proposition \ref{ell}. Hence we will simply denote by $\mathcal{K}_{\infty}$ the set $\mathcal{K}(\Gamma) \cap \bm{\Ph}_{\mathcal{H},s}^{-1}(\infty)$. The open pieces of the phenotypical decomposition also have an intrinsic characterization: they are the maximal open subsets of $\mathcal{K}(\Gamma) \setminus \mathcal{K}_{\infty}$ (for the inclusion) on which the action by conjugation is topologically transitive. Let us call such an open set \textbf{transitive}. This fact relies on the following lemma: 
	
	\begin{lemma}\label{decompdyn}
		Let $\Gamma$ be a group acting on a topological space $X$. Then \begin{itemize}
			\item every transitive open subset of $X$ is contained in a maximal transitive open set;
			\item two maximal transitive open sets are either equal or disjoint. 
		\end{itemize}   
	\end{lemma}
	
	\begin{proof}
		By Zorn's lemma, it suffices to show that, given a chain $(U_i)_{i \in I}$ of transitive open sets, the union $U = \cup_{i \in I}U_i$ is transitive. Let $\Omega_1$ and $\Omega_2$  be two non-empty open subsets of $U$. Let $(i, j) \in I^2$ such that $\Omega_1 \cap U_i \neq \emptyset$ and $\Omega_2 \cap U_j \neq \emptyset$. By symmetry, one can assume that $U_i \subseteq U_j$. Hence, $U_j$ being topologically transitive, there exists $\gamma \in \Gamma$ such that $\gamma (\Omega_1 \cap U_i) \cap (\Omega_2 \cap U_j) \neq \emptyset$. In particular, $\gamma \Omega_1 \cap \Omega_2 \neq \emptyset$, so $U$ is topologically transitive, which proves the first point.
		
		Let $U_1$ and $U_2$ be two maximal topologically transitive open sets. Assume that $U_1 \cap U_2 \neq \emptyset$ and let us show that $U_1 \cup U_2$ is topologically transitive. Let $\Omega_1$ and $\Omega_2$ be non-empty open subsets of $U_1 \cup U_2$. For $i \in \{1,2\}$, there exists $k_i \in \{1,2\}$ such that $\Omega_i \cap U_{k_i} \neq \emptyset$. By topological transitivity of $U_{k_i}$, there exists $\gamma_i \in \Gamma$ such that $\gamma_i  (\Omega_i \cap U_{k_i}) \cap U_1 \cap U_2 \neq \emptyset$. In particular, $\gamma_i \Omega_i \cap (U_1 \cap U_2)$ are non-empty open subsets of $U_1$ so there exists $\gamma \in \Gamma$ such that $\gamma (\gamma_1 \Omega_1 \cap (U_1 \cap U_2)) \cap (\gamma_2 \Omega_2 \cap (U_1 \cap U_2)) \neq \emptyset$, which implies that $\gamma_2^{-1}\gamma \gamma_1 \Omega_1 \cap \Omega_2 \neq \emptyset$. Thus, $U_1 \cup U_2$ is topologically transitive, which implies that $U_1 = U_2$ by maximality of $U_1$ and $U_2$.
	\end{proof}
	This immediately leads to the following proposition:
	
	\begin{proposition}\label{intrinsic}
		For any non-amenable GBS group $\Gamma$ defined by a finite graph of groups $\mathcal{H}$ with infinite cyclic vertex and edge groups, the decomposition $\mathcal{K}(\Gamma) = \bigsqcup_{N \in \mathcal{Q}_{\mathcal{H}, s}}\bm{\Ph}_{\mathcal{H}, s}^{-1}(N) \cap \mathcal{K}(\Gamma)$ depends neither on the choice of the graph $\mathcal{H}$ nor on the choice of the vertex $s \in \mathcal{V}\left(\mathcal{H}\right)$.
	\end{proposition}
	
	\begin{proof}
		The pieces $\bm{\Ph}_{\mathcal{H},s}^{-1}(N) \cap \mathcal{K}(\Gamma)$ (for $N \in \mathcal{Q}_{\mathcal{H},s} \setminus \{\infty\}$) form the maximal transitive open subsets for the action by conjugation of $\Gamma$ on $\mathcal{K}(\Gamma) \setminus \mathcal{K}_{\infty}$.
	\end{proof}

	\bibliographystyle{alpha}

	\bigskip
	{\footnotesize
		
		\noindent
		{\textsc{ENS-Lyon,
				Unité de Mathématiques Pures et Appliquées,  69007 Lyon, France}}
		\par\nopagebreak \texttt{sasha.bontemps@ens-lyon.fr}
	}
	
\end{document}